\documentclass[12pt,a4paper,reqno]{amsart}

\usepackage[a4paper,textwidth=13.5cm,textheight=22.8cm,centering]{geometry}
\usepackage[T1]{fontenc}
\usepackage{lmodern}
\usepackage{microtype}
\usepackage{amsmath,amssymb,amsthm,mathtools}
\usepackage{mathrsfs}
\usepackage{enumitem}
\usepackage{aliascnt}

\theoremstyle{plain}
\newtheorem{theorem}{Theorem}[section]
\newaliascnt{corollary}{theorem}
\newtheorem{corollary}[corollary]{Corollary}
\aliascntresetthe{corollary}
\newaliascnt{proposition}{theorem}
\newtheorem{proposition}[proposition]{Proposition}
\aliascntresetthe{proposition}
\newaliascnt{lemma}{theorem}
\newtheorem{lemma}[lemma]{Lemma}
\aliascntresetthe{lemma}

\theoremstyle{definition}
\newaliascnt{definition}{theorem}
\newtheorem{definition}[definition]{Definition}
\aliascntresetthe{definition}
\newaliascnt{example}{theorem}
\newtheorem{example}[example]{Example}
\aliascntresetthe{example}

\theoremstyle{remark}
\newaliascnt{remark}{theorem}
\newtheorem{remark}[remark]{Remark}
\aliascntresetthe{remark}
\newaliascnt{question}{theorem}
\newtheorem{question}[question]{Question}
\aliascntresetthe{question}

\usepackage{xurl}
\usepackage{cite}
\usepackage[hidelinks,breaklinks=true,pdfencoding=auto,psdextra]{hyperref}
\hypersetup{
  pdftitle={Logarithmic source curves in polynomial fiber products},
  pdfauthor={Henry Shin},
  pdfsubject={Algebraic geometry; logarithmic curves; polynomial fiber products;
    arithmetic applications},
  pdfkeywords={polynomial fiber products, logarithmic source curves,
    integral points, boundary contact orders, Bilu--Tichy pairs}
}
\newcommand{\doilink}[1]{\href{https://doi.org/#1}{doi:#1}}
\newcommand{\arxivlink}[1]{\href{https://arxiv.org/abs/#1}{arXiv:#1}}
\usepackage[nameinlink,noabbrev]{cleveref}
\crefname{theorem}{Theorem}{Theorems}
\Crefname{theorem}{Theorem}{Theorems}
\crefname{corollary}{Corollary}{Corollaries}
\Crefname{corollary}{Corollary}{Corollaries}
\crefname{proposition}{Proposition}{Propositions}
\Crefname{proposition}{Proposition}{Propositions}
\crefname{lemma}{Lemma}{Lemmas}
\Crefname{lemma}{Lemma}{Lemmas}
\crefname{definition}{Definition}{Definitions}
\Crefname{definition}{Definition}{Definitions}
\crefname{remark}{Remark}{Remarks}
\Crefname{remark}{Remark}{Remarks}
\crefname{example}{Example}{Examples}
\Crefname{example}{Example}{Examples}
\crefname{question}{Question}{Questions}
\Crefname{question}{Question}{Questions}

\AddToHook{env/theorem/begin}{\crefalias{section}{theorem}}
\AddToHook{env/corollary/begin}{\crefalias{theorem}{corollary}}
\AddToHook{env/proposition/begin}{\crefalias{theorem}{proposition}}
\AddToHook{env/lemma/begin}{\crefalias{theorem}{lemma}}
\AddToHook{env/definition/begin}{\crefalias{theorem}{definition}}
\AddToHook{env/example/begin}{\crefalias{theorem}{example}}
\AddToHook{env/remark/begin}{\crefalias{theorem}{remark}}
\AddToHook{env/question/begin}{\crefalias{theorem}{question}}

\DeclareMathOperator{\Spec}{Spec}
\DeclareMathOperator{\Frac}{Frac}
\DeclareMathOperator{\ord}{ord}
\DeclareMathOperator{\genus}{genus}

\newcommand{\A}{\mathbb A}
\newcommand{\Pj}{\mathbb P}
\newcommand{\Q}{\mathbb Q}
\newcommand{\Z}{\mathbb Z}
\newcommand{\Gm}{\mathbb G_m}
\newcommand{\OO}{\mathcal O}
\newcommand{\HH}{\mathcal H}
\newcommand{\RR}{\mathcal R}
\newcommand{\red}{\mathrm{red}}
\newcommand{\new}{\mathrm{new}}
\newcommand{\act}{\mathrm{act}}
\newcommand{\adm}{\mathrm{adm}}

\title[Logarithmic source curves]
{Logarithmic source curves in polynomial fiber products}
\author[H. Shin]{Henry Shin}
\thanks{Corresponding author. Address: San Diego, CA, United States of America.
E-mail: \texttt{hkshin@gmail.com}.}
\address{San Diego, CA, United States of America}
\email{hkshin@gmail.com}
\date{}
\subjclass[2020]{Primary 14G05; Secondary 11D45, 14H50, 12E05, 11R27}
\keywords{polynomial fiber products, logarithmic source curves, integral points,
boundary contact orders, Bilu--Tichy pairs}

\begin{document}

\begin{abstract}
Let \(k\) be a field of characteristic zero and let \(f,g\in k[x]\) be
nonconstant. We study rational lifts of \(f(a)\) through \(g\) that do not arise
from a composition \(f=g\circ h\). To each non-graph component of \(f(X)=g(Y)\)
we attach its logarithmic source curve, namely the smooth compactification of
its normalization with reduced boundary. The main geometric result is a sharp
contact formula at infinity: if \(N=\deg g/\gcd(\deg f,\deg g)\), then every
one-infinity non-graph source has \(X\)-degree \(N\), and in general the
\(X\)-degree is \(N\) times the number of boundary points. Over number fields
this yields a finite symmetric-difference expansion of \(S\)-integral new lifts.
Active one-infinity sources give exactly the power terms in height counts;
positive-rank admissible two-infinity sources give logarithmic \(S\)-unit
families; and inactive one-infinity sources, rank-zero two-infinity sources, and
the remaining components contribute only finitely many inputs. Primitive
one-infinity source classes have only polylogarithmic overlap, and ordered
configuration covers introduce no new exponent. Over \(\mathbb Q\), the
\(B^{1/2}\) boundary is precisely the quadratic Bilu--Tichy source cell.
\end{abstract}

\maketitle

\section{Introduction}

Let \(k\) be a field of characteristic zero and let \(f,g\in k[x]\) be
nonconstant.  This paper studies the reduced polynomial fiber product
\[
        C_{f,g}=(f(X)-g(Y)=0)_{\red}
\]
after the graph components \(Y=h(X)\), \(f=g\circ h\), have been removed.  For
any field \(K\) over which \(f\) and \(g\) are defined, write
\[
        \HH_{f,g}(K)=\{h\in K[X]: f=g\circ h\}.
\]
A point above \(X=a\) records a rational lift of the value \(f(a)\) through
\(g\).  The graph components account for the compositional lifts.  The main
result is a componentwise geometric framework: after graph removal, sparse
lifting is controlled by boundary geometry rather than by a case-by-case
classification of separated-variable components.  The logarithmic type and
primitive contact at infinity determine the possible source degrees, and hence
the possible height exponents of the corresponding \(S\)-integral inputs.

For a geometrically integral component \(C\), let \(C^\nu\) be its affine
normalization, let \(\widetilde C\) be its smooth projective normalization, and
let
\[
        D_C=\widetilde C\setminus C^\nu
\]
be the reduced geometric boundary.  We use the logarithmic degree
\[
        \lambda(C)=\deg\Omega^1_{\widetilde C}(\log D_C)
        =2\genus(\widetilde C)-2+\#D_C.
\]
Thus log-negative components are polynomial \(\A^1\)-sources, log-zero
components are \(\Gm\)-type after splitting the two boundary points, and
log-positive components are hyperbolic for integral-point purposes.  Here and
throughout, ``logarithmic'' refers to the curve-level data of a smooth
projective curve with marked boundary, together with the pole orders of the
coordinate functions along that boundary.

We keep the following field-of-definition convention throughout.  Geometric
statements about components, boundary points, and log type are made after base
change to \(\bar k\).  Arithmetic source sets are instead indexed by
\(k\)-components which are geometrically integral and whose relevant normalized
\(\A^1\)- or \(\Gm\)-structure is defined over \(k\), possibly after the
explicit splitting extension used in the two-boundary case.  Irreducible
\(k\)-components which are not geometrically integral are treated separately;
over number fields they contribute only finitely many rational points.  Thus a
geometric one-boundary branch and a \(k\)-rational polynomial source are kept
distinct until descent is explicitly imposed.

The basic geometric theorem is an exact primitive-contact formula at infinity.
Write
\[
        m=\deg f,\qquad n=\deg g,\qquad e=\gcd(m,n),\qquad
        M=m/e,\qquad N=n/e.
\]
The two polynomial maps are totally ramified over \(\infty\), and the completed
local fiber product at \((\infty,\infty)\) is analytically
\[
        \bar k[[u,v]]/(u^m-v^n).
\]
It has exactly \(e\) reduced normalized branches over \((\infty,\infty)\), and
on each branch
\[
        -\ord_P X=N,
        \qquad
        -\ord_P Y=M.
\]
Consequently every geometrically integral component satisfies
\[
        d_X(C)=N\#D_C,
        \qquad
        d_Y(C)=M\#D_C.
\]
In particular, every non-graph one-infinity source has the single possible
\(X\)-degree \(N\).  Equivalently, after graph removal the one-infinity
\(X\)-degree spectrum is the singleton \(\{N\}\) when such sources occur, and is
empty otherwise.

Over a number field \(k\), with \(S\) containing the Archimedean places, define
\[
\begin{split}
 \mathscr N^{\new}_{f,g,k,S}
 =\Bigl\{a\in\OO_{k,S}:{}&
       \exists y\in k\text{ with }g(y)=f(a),\\
       &y\ne h(a)\text{ for every }h\in\HH_{f,g}(k)\Bigr\}.
\end{split}
\]
We prove a finite symmetric-difference source expansion for this set.  Up to
finitely many inputs,
\(\mathscr N^{\new}_{f,g,k,S}\) is a finite union of polynomial images from
active \(\A^1\)-sources and Laurent images of admissible unit-coset sets from
\(\Gm\)-sources.  Thus active \(\A^1\)-sources give exactly the power terms,
positive-rank admissible \(\Gm\)-sources give exactly logarithmic unit-rank terms,
and inactive one-infinity sources, rank-zero admissible \(\Gm\)-sources, and all
remaining components contribute finitely.

Throughout the arithmetic statements, we distinguish the \emph{source exponent}
\(\theta=1/d_X\) from the actual height-count exponent over a number field.  A
polynomial source of \(X\)-degree \(d\) contributes a power
\(B^{[k:\Q]/d}\), up to the indicated \(S\)-unit logarithmic factor; over
\(\Q\) with ordinary-integer counting this is \(B^{1/d}\).

For the global form of the result, source separation is essential.
All active one-infinity sources have the same forced \(X\)-degree \(N\), so one
must control possible power-sized overlaps among their value sets.  We prove that,
after quotienting active one-infinity sources by affine reparametrization of
their \(X\)-maps, distinct source classes have only polylogarithmic overlap.  Hence,
whenever active one-infinity sources exist,
\[
        N^{\new}_{f,g,k,S}(B)
        =
        \sum_{\xi}
        \#\{A_\xi(t)\in\OO_{k,S}:H(A_\xi(t))\le B\}
        +O((\log B)^\rho).
\]
The dominant power contribution is therefore canonical; the lower-order error
comes from two-infinity unit terms and polylogarithmic overlaps.

All arithmetic source expansions below are finite symmetric-difference
statements: the displayed source unions agree with the corresponding input sets
after a finite set of input values is removed.

The same logarithmic mechanism controls higher multiplicity.  Inputs with at
least \(r\) distinct new rational lifts are governed by the ordered
\(r\)-configuration cover of the non-graph part of \(C_{f,g}\).  The
primitive-contact formula is stable under this construction, so every active
multiplicity level has the same source exponent \(1/N\) and, over a number
field, height-count power \(B^{[k:\Q]/N}\); otherwise it has only logarithmic or
finite growth.

Finally, over \(\Q\) the square-root boundary is completely explicit.  A
non-graph one-infinity component of \(X\)-degree two exists exactly in the
quadratic-source case.  Equivalently, after linear changes,
\[
        g(Y)=G((Y-c)^2),
        \qquad
        f(X)=G(UE(U)^2),
\]
so the switched first-kind cell \((UE(U)^2,V^2)\) is the source-level
Bilu--Tichy core.  Ordinary integer activity is then decided by a single
quadratic congruence.

\subsection*{Relation with classification results}

The separated-variable components of \(f(X)-g(Y)\) have a long classification
history.  Ritt's decomposition theory and Engstrom's cancellation theorem are
part of the polynomial-decomposition background used here \cite{Ritt,Engstrom}.
Davenport, Lewis and Schinzel initiated the reducibility problem for separated
polynomials \cite{DLS61,DLS}; Fried--MacRae studied curves with separated
variables through polynomial decompositions \cite{FriedMacRae}; Bilu--Tichy
classified pairs \((f,g)\) for which \(f(x)=g(y)\) has infinitely many rational
solutions with bounded denominator \cite{BiluTichy}; and Dvornicich--Zannier
studied closely related value questions for algebraic functions
\cite{DvornicichZannierI,DvornicichZannierII}.  Avanzi--Zannier's work on
separated-variable genus-one curves and polynomial Pell equations is close to the
log-zero examples below \cite{AvanziZannier}.  Pakovich's work on Laurent
polynomial decompositions and genus bounds for fiber products gives another
geometric point of contact with low-complexity components
\cite{PakovichLaurent,PakovichFiberProducts}.  K\"onig--Neftin study reducible
fibers of polynomial maps from a quantitative Hilbert irreducibility viewpoint
\cite{KoenigNeftin}.  Behajaina--K\"onig--Neftin have posted a preprint on
the Davenport--Lewis--Schinzel reducibility problem over \(\mathbb C\)
\cite{BKN}.

The present paper develops a componentwise framework complementary to those
classifications.  Rather than classifying all separated-variable components, it
starts from a given component and uses its normalization and boundary contact to
determine its log type, its exact projection degree, and its arithmetic
contribution.  In particular, the source exponent is not an
auxiliary counting parameter but a primitive boundary-contact invariant.  For
the square-root boundary, a complete source-level classification is needed, and
in that case the classification reduces to the single Bilu--Tichy cell
\((UE(U)^2,V^2)\).

\subsection*{Main theorem map}

The main results are grouped as follows.

\emph{Theorem A: primitive contact.}  Log-negative components are primitive
polynomial source diagrams, and every geometric boundary point has
\[
        -\ord_P X=\frac{\deg g}{\gcd(\deg f,\deg g)},\qquad
        -\ord_P Y=\frac{\deg f}{\gcd(\deg f,\deg g)}.
\]
Hence
\[
        d_X(C)=\frac{\deg g}{\gcd(\deg f,\deg g)}\#D_C,
        \qquad
        d_Y(C)=\frac{\deg f}{\gcd(\deg f,\deg g)}\#D_C.
\]
In particular, a one-infinity component has the forced \(X\)-degree \(N\).

\emph{Theorem B: finite symmetric-difference source expansion.}  Over a number field, after deleting
finitely many input values, the new-lift set is exactly a finite union of
polynomial images from active \(\A^1\)-sources and Laurent images from
admissible unit cosets on \(\Gm\)-sources.  This structural statement yields
the height estimates.

\emph{Theorem C: primitive source separation.}  After quotienting active
one-infinity sources by affine reparametrization of the \(X\)-source, distinct
classes have only polylogarithmic overlap.  Consequently the power part of the
source expansion is canonical:
\[
        N^{\new}_{f,g,k,S}(B)
        =
        \sum_{\xi}
        \#\{A_\xi(t)\in\OO_{k,S}:H(A_\xi(t))\le B\}
        +O((\log B)^\rho).
\]
The lower-order error contains the admissible two-infinity unit terms and the
polylogarithmic overlaps among distinct primitive \(X\)-source classes.

\emph{Theorem D: sparse trichotomy.}  The finite symmetric-difference source expansion and
source separation give the arithmetic alternatives:
\[
        \text{active one-infinity sources exist}
        \quad\Longrightarrow\quad
        \theta_{f,g,k,S}
        =
        \frac{\gcd(\deg f,\deg g)}{\deg g},
\]
while absence of active one-infinity sources leaves only positive-rank admissible
unit-coset growth or boundedness.  When
\[
        N=\frac{\deg g}{\gcd(\deg f,\deg g)}
\]
equals \(1\), no non-graph one-infinity source remains after graph removal.

\emph{Theorem E: configuration stability.}  The same finite symmetric-difference source expansion and
the same primitive contact formula hold on ordered configuration covers.  Higher
multiplicity may introduce or eliminate source curves, but it introduces no new
source exponent: at every active level the source exponent is \(1/N\), and the
corresponding number-field height-count power is \(B^{[k:\Q]/N}\).

\emph{Theorem F: quadratic boundary.}  The square-root case \(N=2\) is
classified at source level.  It is exactly the source-even switched first-kind
Bilu--Tichy cell, with ordinary-integer activity decided by one quadratic
congruence.

The bound is sharp in the family
\[
        f(X)=X^{eM},\qquad g(Y)=Y^{eN},\qquad (M,N)=1,
\]
where the source curve \(X=t^N,\ Y=t^M\) realizes the source exponent \(1/N\)
for \(N\ge2\); over a number field the corresponding height-count exponent is
\([k:\Q]/N\).

\subsection*{Organization}

\Cref{sec:log-source-curves} develops the algebraic geometry of logarithmic
source curves: graph removal, polynomial source diagrams, boundary contact
orders, the degree spectrum, sharpness examples, and constructibility of source
strata.  \Cref{sec:component-count} proves the arithmetic log-trichotomy for
integral points on affine curves and strengthens the log-zero part to an
admissible unit-coset theorem.  \Cref{sec:number-field,sec:active-lower} prove
the finite symmetric-difference source expansion, the primitive source-separation theorem, and the
sparse lifting trichotomy.  \Cref{sec:configuration-multiplicity}
treats higher multiplicity by ordered configuration covers and proves the
analogous level-\(r\) source expansion.  \Cref{sec:algorithm} records an
effective ordinary-integer exponent criterion.  \Cref{sec:quadratic} proves the
quadratic-source classification of the square-root boundary and the geometry of
the source-even locus.  The concluding section records several consequences and
related directions, including the primitive source stratum and a finite-field
analogue.

\section{Logarithmic source curves and boundary contact orders}
\label{sec:log-source-curves}

This section contains the algebraic-geometric core of the paper.  The results
are stated over an arbitrary characteristic-zero field; arithmetic hypotheses
enter only later.  The logarithmic data used here consist of a smooth projective
curve together with a reduced boundary divisor, along with the pole orders of
the coordinate functions at that boundary.

Let \(k\) be a field of characteristic zero and let \(f,g\in k[x]\) be
nonconstant.  Put
\[
        C_{f,g}=(f(X)-g(Y)=0)_{\red}\subset\A^2_k.
\]

\begin{lemma}[Degree-one components are graph components]
\label{lem:degree-one-graph-components}
Let \(K\) be a field, and let \(f,g\in K[x]\) be nonconstant.  The irreducible
components of \(C_{f,g}\) of generic degree one over \(\A^1_X\) are exactly the
graph components \(Y=h(X)\), where \(h\in K[X]\) and \(f=g\circ h\).
\end{lemma}

\begin{proof}
If \(D\) is a component of generic degree one over the \(X\)-line, then
\(K(D)=K(X)\).  Hence the image of \(Y\) in \(K(D)\) has the form \(h(X)\) for
some \(h\in K(X)\), and the identity on \(D\) gives
\[
        f(X)=g(h(X)).
\]
If \(h\) had a finite pole, then \(g(h)\) would have a finite pole: at that place
the leading term of \(g(h)\) has strictly larger pole order than every lower
term.  This contradicts \(f\in K[X]\).  Hence \(h\in K[X]\), and the component is
the graph \(Y=h(X)\).  Conversely, if \(h\in K[X]\) and \(f=g\circ h\), then
\(Y-h(X)\) divides \(f(X)-g(Y)\), so \(Y=h(X)\) is an irreducible component of
\(C_{f,g}\) of generic degree one over \(\A^1_X\).
\end{proof}

With the notation
\[
        \HH_{f,g}(K)=\{h\in K[X]: f=g\circ h\},
\]
the set \(\HH_{f,g}(K)\) is finite by the lemma, since \(C_{f,g}\) has only
finitely many irreducible components.

For a geometrically integral component \(C\) of \(C_{f,g}\), let \(C^\nu\) be
the affine normalization, let \(\widetilde C\) be the smooth projective
normalization, and write
\[
        D_C=\widetilde C\setminus C^\nu
\]
for the reduced geometric boundary.  The coordinate functions \(X,Y\) define
rational functions on \(\widetilde C\), regular on \(C^\nu\).  We write
\[
        d_X(C)=\deg(X:\widetilde C\to\Pj^1),
        \qquad
        d_Y(C)=\deg(Y:\widetilde C\to\Pj^1).
\]

\begin{definition}[Logarithmic source curve]
A geometrically integral component \(C\) of \(C_{f,g}\), together with the
boundary pair \((\widetilde C,D_C)\) and the two functions \(X,Y\), is called a
\emph{logarithmic source curve} for \((f,g)\).  Its logarithmic degree is
\[
        \lambda(C)=\deg\Omega^1_{\widetilde C}(\log D_C)
        =2\genus(\widetilde C)-2+\#D_C.
\]
We call \(C\) log-negative, log-zero, or log-positive according as
\(\lambda(C)<0\), \(=0\), or \(>0\).
\end{definition}

\begin{lemma}[Curve-level log trichotomy]
\label{lem:curve-log-trichotomy}
Let \(C\) be a geometrically integral affine curve over \(k\), and suppose
\(D_C\ne\varnothing\).  Over \(\bar k\), the sign of
\(\lambda(C)=2\genus(\widetilde C)-2+\#D_C\) has the following meaning.
\begin{enumerate}[label=\textup{(\roman*)}]
    \item \(\lambda(C)<0\) if and only if
    \((\genus(\widetilde C),\#D_C)=(0,1)\).  Then
    \(C^\nu_{\bar k}\cong\A^1_{\bar k}\).
    \item \(\lambda(C)=0\) if and only if
    \((\genus(\widetilde C),\#D_C)=(0,2)\).  Then, after splitting the two boundary
    points, \(C^\nu\) is \(\Gm\).
    \item \(\lambda(C)>0\) in all remaining cases.
\end{enumerate}
\end{lemma}

\begin{proof}
Since \(D_C\ne\varnothing\), put \(r=\#D_C\ge1\).  The formula
\(\lambda=2\gamma-2+r\), with \(\gamma\ge0\), gives \(\lambda<0\) only for
\((\gamma,r)=(0,1)\), and \(\lambda=0\) only for \((\gamma,r)=(0,2)\).  The corresponding
affine curves over \(\bar k\) are \(\Pj^1\setminus\{\infty\}\cong\A^1\) and
\(\Pj^1\setminus\{0,\infty\}\cong\Gm\).  The remaining cases have
\(2\gamma-2+r>0\).
\end{proof}

\begin{remark}[Curve-level logarithmic data]
The logarithmic input in this paper is the curve-level pair consisting of the
smooth projective normalization and its reduced boundary divisor.  The boundary
divisor records the allowed poles, and \(\Omega^1(\log D)\) records the
log-canonical degree.  This is the one-dimensional setting of logarithmic
structures in the sense of Fontaine--Illusie--Kato \cite{KatoLog}.
\end{remark}

\begin{definition}[Primitive polynomial source]
\label{def:primitive-source}
A \emph{primitive polynomial source} for \((f,g)\) over \(k\) is a pair of
nonconstant polynomials \((A,B)\in k[t]^2\) such that
\[
        f(A(t))=g(B(t))
\]
and
\[
        k(A(t),B(t))=k(t).
\]
It is a graph source if \(B\in k[A]\), and a non-graph source otherwise.
Two primitive sources are identified if they differ by an affine change of
parameter \(t\mapsto at+b\), \(a\in k^*\), \(b\in k\).
\end{definition}

\begin{theorem}[Log-negative components as primitive polynomial sources]
\label{thm:one-infinity-source-diagrams}
Let \(k\) be a field of characteristic zero, and let \(f,g\in k[x]\) be
nonconstant.  There is a bijection between:
\begin{enumerate}[label=\textup{(\roman*)}]
    \item geometrically integral irreducible components \(C\) of
    \(C_{f,g}\), defined over \(k\), whose affine normalization is isomorphic to
    \(\A^1_k\);
    \item primitive polynomial sources \((A,B)\) for \((f,g)\), up to affine
    change of the parameter.
\end{enumerate}
Under this correspondence
\[
        X=A(t),\qquad Y=B(t),
\]
\[
        d_X(C)=\deg A,
        \qquad
        d_Y(C)=\deg B,
\]
and
\[
        \deg f\cdot\deg A=\deg g\cdot\deg B.
\]
Moreover, \(C\) is a graph component if and only if \(B\in k[A]\).
\end{theorem}

\begin{proof}
Let \(C\) be a component whose affine normalization is \(\A^1_k\), and choose a
coordinate \(t\) on this normalization.  We first record that the coordinate
projections from every irreducible component of \(f(X)-g(Y)=0\) are finite and
dominant.  If \(n=\deg g\) and \(b_n\) is the leading coefficient of \(g\), then
in the coordinate ring of any component the element \(Y\) satisfies the monic
relation
\[
        Y^n+\frac{b_{n-1}}{b_n}Y^{n-1}+\cdots+
        \frac{b_0-f(X)}{b_n}=0
\]
over \(k[X]\).  Thus the component coordinate ring is finite over \(k[X]\).  The
same argument, with \(f\) and \(g\) interchanged, gives finiteness over
\(k[Y]\).  Since the component has dimension one, neither finite image can be a
closed point of the corresponding affine line; hence both projections are
dominant and \(X,Y\) are nonconstant on \(C\).

The restrictions of the coordinate functions \(X,Y\) are regular functions on
\(\A^1_k\), hence
\[
        X=A(t),\qquad Y=B(t)
\]
for some nonconstant \(A,B\in k[t]\).  The equation of the component gives
\(f(A(t))=g(B(t))\).  Since \(t\) is a normalization
parameter, the function field of the image curve is
\(k(A(t),B(t))=k(t)\).  The projection degrees are the degrees of the rational
functions \(A\) and \(B\), hence \(\deg A\) and \(\deg B\).  Comparing degrees in
\(f(A(t))=g(B(t))\) gives
\[
        \deg f\cdot\deg A=\deg g\cdot\deg B.
\]
Changing the coordinate on \(\A^1_k\) changes \((A,B)\) by an affine
reparametrization.

Conversely, a primitive source defines a morphism
\[
        \A^1_t\longrightarrow \A^2_{X,Y},
        \qquad
        t\longmapsto (A(t),B(t)),
\]
whose image lies in \(C_{f,g}\).  Let
\[
        \varphi:k[X,Y]\longrightarrow k[t],
        \qquad X\longmapsto A(t),\quad Y\longmapsto B(t).
\]
The kernel of \(\varphi\) is a prime ideal containing \(f(X)-g(Y)\), and its
quotient is the image coordinate ring \(k[A(t),B(t)]\).  Since the image has
dimension one and is contained in the plane curve \(C_{f,g}\), its Zariski
closure is an irreducible component of the reduced curve \(C_{f,g}\), with
coordinate ring \(k[A(t),B(t)]\).  The condition \(k(A(t),B(t))=k(t)\) says that
\(\A^1_t\) is birational to this component.  The affine normalization is exactly
\(\A^1_k\): indeed, write
\[
        A(T)=a_dT^d+a_{d-1}T^{d-1}+\cdots+a_0,
        \qquad a_d\in k^*.
\]
Then the element \(t\) satisfies the monic equation
\[
        T^d+\frac{a_{d-1}}{a_d}T^{d-1}+\cdots+
        \frac{a_0-A(t)}{a_d}=0
\]
over \(k[A(t)]\subseteq k[A(t),B(t)]\), so \(t\) is integral over
\(k[A(t),B(t)]\).  Therefore \(k[t]\) is integral over \(k[A(t),B(t)]\).
Conversely, any element of \(k(t)\) integral over \(k[A(t),B(t)]\) is
integral over \(k[t]\), hence lies in \(k[t]\), since \(k[t]\) is integrally
closed.

Finally, if \(B\in k[A]\), say \(B=h(A)\), then
\(f(A)=g(h(A))\).  Since \(A\) is nonconstant, \(f=g\circ h\), and the component
is the graph \(Y=h(X)\).  Conversely, every graph component has
\(B=h(A)\).
\end{proof}

\begin{lemma}[Boundary branches and completed local normalization]
\label{lem:boundary-branches-completion}
Let \(\bar C\) be a reduced curve over an algebraically closed field \(K\), let
\(\nu:\bar C^\nu\to\bar C\) be its normalization, and let
\(p\in\bar C\) be a closed point.  Then the points of \(\nu^{-1}(p)\) are in
natural bijection with the factors of the normalization of the reduced completed
local ring \(\widehat{\OO_{\bar C,p}}_{\red}\).  Equivalently, they are the
normalized analytic branches of \(\bar C\) at \(p\).
\end{lemma}

\begin{proof}
Curves over a field are excellent, so normalization is finite and behaves in the
usual way after completion at a closed point; see the standard excellence and
normalization facts in \cite[Tags 07QS, 032E and 035S]{StacksProject}.  Since
\(\nu\) is finite, base change to the completed local ring gives
\[
        \OO_{\bar C^\nu,\nu^{-1}(p)}
        \otimes_{\OO_{\bar C,p}}\widehat{\OO_{\bar C,p}}
        \cong
        \prod_{Q\in\nu^{-1}(p)}\widehat{\OO_{\bar C^\nu,Q}}.
\]
Each factor on the right is a complete discrete valuation ring.  The product is
finite and integral over \(\widehat{\OO_{\bar C,p}}_{\red}\), has the same total
quotient ring, and is normal.  It is therefore the normalization of the reduced
completed local ring.  Thus the factors, or equivalently the normalized analytic
branches, are indexed exactly by the points of the projective normalization over
\(p\).
\end{proof}

\begin{theorem}[Primitive contact at infinity]
\label{thm:primitive-contact-infinity}
Let \(k\) be a field of characteristic zero, let \(f,g\in k[x]\) be
nonconstant, and put
\[
        m=\deg f,
        \qquad
        n=\deg g,
        \qquad
        e=\gcd(m,n),
        \qquad
        M=\frac m e,
        \qquad
        N=\frac n e.
\]
Let \(C_1,\ldots,C_s\) be the geometrically irreducible components of
\(C_{f,g,\bar k}\).  For each \(i\), let \(D_i\) be the boundary of the affine
normalization of \(C_i\).  Then
\[
        \sum_{i=1}^s \#D_i=e.
\]
Moreover, for every \(i\) and every boundary point \(P\in D_i\),
\[
        -\ord_P X=N,
        \qquad
        -\ord_P Y=M.
\]
Consequently every geometrically integral component \(C_i\) satisfies
\[
        d_X(C_i)=N\#D_i,
        \qquad
        d_Y(C_i)=M\#D_i.
\]
\end{theorem}

\begin{proof}
It is enough to work over \(\bar k\).  Compactify the two polynomial maps to
finite morphisms
\[
        f:\Pj^1_X\longrightarrow\Pj^1_T,
        \qquad
        g:\Pj^1_Y\longrightarrow\Pj^1_T.
\]
Let
\[
        \mathscr F=\Pj^1_X\times_{\Pj^1_T}\Pj^1_Y.
\]
Over \(\A^1_T\), the reduced scheme underlying \(\mathscr F\) is the affine
curve \(C_{f,g,\bar k}\).  Since \(f\) and \(g\) are polynomials, the only point
of \(\Pj^1_X\) above \(T=\infty\) is \(X=\infty\), and the only point of
\(\Pj^1_Y\) above \(T=\infty\) is \(Y=\infty\).  Hence every boundary branch of
every affine component lies over the single point \((\infty,\infty)\) of
\(\mathscr F\).

Let \(z=1/T\) be a uniformizer at \(\infty\) on the target, and let
\(u=1/X\), \(v=1/Y\) be uniformizers at \(\infty\) on the two source lines.  In
the completed local rings of the two source lines one has
\[
        z=u^m\varepsilon(u),
        \qquad
        z=v^n\delta(v),
\]
where \(\varepsilon(0),\delta(0)\ne0\).  Since \(\bar k\) is algebraically closed
of characteristic zero, the units \(\varepsilon\) and \(\delta\) admit \(m\)-th
and \(n\)-th roots in the corresponding complete local rings.  After replacing
\(u\) and \(v\) by unit multiples, we may therefore write
\[
        z=u^m,
        \qquad
        z=v^n.
\]
Thus the completed local ring of \(\mathscr F\) at \((\infty,\infty)\) is
analytically isomorphic to
\[
        \bar k[[u,v]]/(u^m-v^n).
\]
Writing \(m=eM\) and \(n=eN\), with \((M,N)=1\), gives
\[
        u^m-v^n
        =u^{eM}-v^{eN}
        =\prod_{\zeta^e=1}\bigl(u^M-\zeta v^N\bigr).
\]
For each \(e\)-th root of unity \(\zeta\), the factor
\(u^M-\zeta v^N\) is analytically irreducible.  Indeed, choose nonzero
\(\alpha,\beta\in\bar k\) with \(\alpha^M=\zeta\beta^N\); then
\[
        u=\alpha s^N,
        \qquad
        v=\beta s^M
\]
defines the normalization of that branch.  Since \((M,N)=1\), the kernel of
\(\bar k[[u,v]]\to\bar k[[s]]\) is the height-one prime generated by
\(u^M-\zeta v^N\).

Passing now to the reduced completed local ring and then to its normalization,
we obtain exactly \(e\) normalized branches above \((\infty,\infty)\).  More
explicitly, the formal discs displayed above are the normalization factors of
\(\widehat{\OO_{\mathscr F,(\infty,\infty)}}_{\red}\).  Nilpotent structure at
the completed fiber product does not change the reduced affine curve.  By
\Cref{lem:boundary-branches-completion}, these normalized reduced analytic
branches are exactly the points of the projective normalizations of the reduced
geometric components lying over \((\infty,\infty)\).  Conversely, every point of
\(D_i\) for any affine component \(C_i\) lies over \(X=\infty\) and
\(Y=\infty\).  Indeed, if neither coordinate had a pole, the point would belong
to the affine normalization, while if exactly one coordinate had a pole then the
identity \(f(X)=g(Y)\) would have a pole on exactly one side.  Hence the displayed
analytic branches are precisely the boundary points of the affine normalizations
of the geometric components \(C_1,\ldots,C_s\), and
\(\sum_i\#D_i=e\).  Along every such normalized branch,
\[
        \ord_s(u)=N,
        \qquad
        \ord_s(v)=M.
\]
The original parameters \(1/X\) and \(1/Y\) differ from the displayed \(u\) and
\(v\) only by units, so the pole orders of the coordinate functions are
\[
        -\ord_P X=N,
        \qquad
        -\ord_P Y=M.
\]
Finally, the degree of a rational function on a smooth projective curve is the
degree of its pole divisor.  Thus, for every geometric component \(C_i\),
\[
\begin{aligned}
        d_X(C_i)&=\sum_{P\in D_i}(-\ord_P X)=N\#D_i,\\
        d_Y(C_i)&=\sum_{P\in D_i}(-\ord_P Y)=M\#D_i.
\end{aligned}
\]
\end{proof}

\begin{corollary}[Primitive geometric degree spectrum]
\label{cor:geometric-degree-spectrum}
With notation as in \Cref{thm:primitive-contact-infinity}, every geometrically
integral component \(C\) of \(C_{f,g}\) satisfies
\[
        d_X(C)=N\#D_C,
        \qquad
        d_Y(C)=M\#D_C.
\]
In particular, every log-negative component has
\[
        d_X(C)=N,
        \qquad
        d_Y(C)=M.
\]
If such a component is non-graph, then \(N\ge2\).  Equivalently, after graph
removal the one-infinity power-producing \(X\)-degree set is either empty or
\(\{N\}\) when \(N\ge2\), and is empty when \(N=1\).
\end{corollary}

\begin{proof}
The exact degree formulas are \Cref{thm:primitive-contact-infinity}.  A
log-negative component has \(\#D_C=1\), so \(d_X(C)=N\) and \(d_Y(C)=M\).  If it
is non-graph and \(N=1\), then \(d_X(C)=1\), but the degree-one components over
the \(X\)-line are exactly graph components by
\Cref{lem:degree-one-graph-components}.  This contradiction proves \(N\ge2\) in
the non-graph case.
\end{proof}

\begin{corollary}[Ground-field primitive degree spectrum]
\label{cor:ground-field-degree-spectrum}
With notation as in \Cref{cor:geometric-degree-spectrum}, let \(C\) be an
irreducible component of \(C_{f,g}\) over \(k\) which dominates the \(X\)-line.
After base change to \(\bar k\), write the reduced geometric decomposition as
\[
        C_{\bar k}=C_1\cup\cdots\cup C_s.
\]
Then
\[
        d_X(C)=N\sum_{i=1}^s\#D_{C_i},
        \qquad
        d_Y(C)=M\sum_{i=1}^s\#D_{C_i}.
\]
In particular
\[
        d_X(C)\in N\Z_{\ge1},
        \qquad
        d_Y(C)\in M\Z_{\ge1}.
\]
If \(C\) is geometrically integral and log-negative, then \(d_X(C)=N\); if it is
also non-graph, then \(N\ge2\).
\end{corollary}

\begin{proof}
Since \(k\) has characteristic zero, \(k\) is perfect and every irreducible
component is geometrically reduced after base change.  The coordinate ring
\(k[C]\) is finite and torsion-free over \(k[X]\), because the component
dominates the \(X\)-line and the projection is finite.  Hence \(d_X(C)\) is the
rank of the finite \(k[X]\)-module \(k[C]\).  After tensoring with
\(\bar k[X]\), this rank is unchanged, while the reduced geometric decomposition
gives the total quotient algebra as the product of the function fields of the
geometric components.  Therefore the rank, equivalently the degree of the finite
map, is the sum of the projection degrees of the geometric components:
\[
        d_X(C)=\sum_{i=1}^s d_X(C_i).
\]
By \Cref{cor:geometric-degree-spectrum}, each summand is \(N\#D_{C_i}\), giving
the formula for \(d_X(C)\).  The proof for \(d_Y(C)\) is identical.  The final
assertions are the geometrically integral case and
\Cref{lem:degree-one-graph-components}.
\end{proof}

\begin{example}[Sharpness of primitive contact]
\label{ex:gcd-spectrum-sharp}
Let
\[
        m=eM,
        \qquad
        n=eN,
        \qquad
        (M,N)=1,
        \qquad
        N\ge2.
\]
Take
\[
        f(X)=X^{eM},
        \qquad
        g(Y)=Y^{eN}.
\]
Then \(X^{eM}-Y^{eN}\) has the irreducible component
\[
        X^M-Y^N=0.
\]
Indeed, the binomial \(X^M-Y^N\) is irreducible when \((M,N)=1\).  It is
parametrized by
\[
        X=t^N,
        \qquad
        Y=t^M.
\]
It is log-negative, primitive, and non-graph, with
\[
        d_X=N=\frac{\deg g}{\gcd(\deg f,\deg g)}.
\]
Thus the primitive one-infinity degree \(d_X=N\) is realized for every \(N\ge2\).
\end{example}

\begin{lemma}[Primitivity and divided differences]
\label{lem:primitive-divided-differences}
Let \(k\) be an algebraically closed field of characteristic zero, and let
\(A,B\in k[t]\) be nonconstant.  Put
\[
        F_A(T,U)=\frac{A(T)-A(U)}{T-U},
        \qquad
        F_B(T,U)=\frac{B(T)-B(U)}{T-U}.
\]
Then
\[
        k(A(t),B(t))=k(t)
\]
if and only if \(F_A\) and \(F_B\) have no common nonconstant factor in
\(k[T,U]\).
\end{lemma}

\begin{proof}
Let
\[
        \phi:\A^1_t\longrightarrow \A^2,
        \qquad t\longmapsto(A(t),B(t)).
\]
The extension \(k(t)/k(A,B)\) is finite.  Indeed, if
\[
        A(T)=a_dT^d+a_{d-1}T^{d-1}+\cdots+a_0,
        \qquad a_d\in k^*,
\]
then \(t\) satisfies
\[
        T^d+\frac{a_{d-1}}{a_d}T^{d-1}+\cdots+
        \frac{a_0-A(t)}{a_d}=0
\]
over \(k[A]\subseteq k[A,B]\).  Hence \(t\) is integral over \(k[A,B]\), and
\(k(t)\) is finite over \(k(A,B)\).  Since the characteristic is zero, this
finite extension is separable.  Thus \(k(A,B)=k(t)\) is equivalent to \(\phi\)
being generically one-to-one onto its scheme-theoretic image.

Let \(R=k[A(t),B(t)]\subset k[t]\), so the scheme-theoretic image of
\(\phi\) is \(\Spec R\).  The fiber product
\(\A^1\times_{\Spec R}\A^1\) is cut out in the \((T,U)\)-plane by
\[
        A(T)-A(U)=0,
        \qquad
        B(T)-B(U)=0.
\]
Both equations contain the diagonal factor \(T-U\).  After removing this
factor, the off-diagonal fiber product is contained in
\[
        V(F_A,F_B)\subset\A^2_{T,U},
\]
and every point of \(V(F_A,F_B)\) with \(T\ne U\) belongs to the off-diagonal
fiber product.  The only possible extra points on the diagonal satisfy
\(A'(T)=B'(T)=0\), and this is a finite set because \(A\) and \(B\) are
nonconstant in characteristic zero.  Therefore the off-diagonal fiber product
has a one-dimensional component if and only if \(V(F_A,F_B)\) has one, which is
equivalent to \(F_A\) and \(F_B\) having a common nonconstant factor.

If \(\phi\) has generic degree greater than one, then over a general point of
its image there are two distinct parameters, so the off-diagonal fiber product
has dimension one.  Conversely, a one-dimensional off-diagonal component gives a
generic family of distinct parameters with the same image, so \(\phi\) is not
generically one-to-one.  Hence \(\phi\) is birational onto its image if and only
if \(F_A\) and \(F_B\) have no common nonconstant factor, as claimed.
\end{proof}

\begin{proposition}[Constructibility of the primitive source stratum]
\label{prop:source-strata-constructible}
Let \(k\) be an algebraically closed field of characteristic zero.  Fix integers
\(m,n,d\ge1\), and let \(\mathcal P_m\) and \(\mathcal P_n\) be the coefficient
spaces of degree-exactly-\(m\) and degree-exactly-\(n\) polynomials.  Put
\[
        e=\gcd(m,n),
        \qquad
        M=\frac m e,
        \qquad
        N=\frac n e.
\]
The locus of pairs
\[
        (f,g)\in\mathcal P_m\times\mathcal P_n
\]
for which \(C_{f,g}\) has a log-negative component of \(X\)-degree \(d\) is a
constructible subset.  The non-graph sublocus is also constructible.  Both loci
are empty unless
\[
        d=N.
\]
When \(d=N\), any such source has \(Y\)-degree \(M\).
\end{proposition}

\begin{proof}
By \Cref{cor:geometric-degree-spectrum}, every log-negative component has
\(X\)-degree \(N\).  Hence the required loci are empty unless \(d=N\).  Assume
from now on that \(d=N\).  Then degree comparison in
\[
        f(A(t))=g(B(t))
\]
forces \(\deg B=M\).

By \Cref{thm:one-infinity-source-diagrams}, a log-negative component of
\(X\)-degree \(N\) is equivalent to a primitive polynomial source
\[
        f(A(t))=g(B(t)),
        \qquad
        \deg A=N,
        \qquad
        \deg B=M,
        \qquad
        k(A,B)=k(t),
\]
up to affine reparametrization.  Introduce coefficient variables for
\(f\in\mathcal P_m\), \(g\in\mathcal P_n\), \(A\in k[t]\) of degree exactly
\(N\), and \(B\in k[t]\) of degree exactly \(M\).  The identity
\[
        f(A(t))-g(B(t))=0
\]
is equivalent to finitely many polynomial equations in these coefficients.  It
therefore defines a locally closed incidence scheme \(I_{m,n}\).

It remains to impose primitivity.  By
\Cref{lem:primitive-divided-differences}, this is equivalent to the condition
that
\[
        F_A(T,U)=\frac{A(T)-A(U)}{T-U},
        \qquad
        F_B(T,U)=\frac{B(T)-B(U)}{T-U}
\]
have no common nonconstant factor in \(k[T,U]\).  The condition that two
polynomials of bounded bidegree in \(k[T,U]\) have a common factor of positive
degree is constructible.  Indeed, one takes the finite union over all possible
positive bidegrees, or equivalently over all possible total degrees together
with bounded supports, of the locally closed incidence conditions
\[
        F_A=HK,
        \qquad
        F_B=HL,
        \qquad
        H\ne0,
        \qquad
        \deg H\ge1.
\]
Exact-degree or exact-bidegree conditions are imposed by the corresponding open
leading-coefficient conditions inside the finite coefficient spaces for
\(H,K,L\).  The images of these locally closed incidences under projection are
constructible by Chevalley's theorem.  Therefore the common-factor locus is
constructible, and its complement, the no-common-factor condition, is
constructible inside \(I_{m,n}\).

The graph condition is constructible as well.  Inside the primitive incidence,
if \(B\in k[A]\), then
\[
        k(t)=k(A,B)=k(A),
\]
so \(\deg A=N=1\).  Thus, for \(N>1\), the graph locus in the primitive stratum
is empty.  For \(N=1\), the condition is that \(B=h(A)\) for some polynomial
\(h\) of degree \(M\); for this fixed degree the equation \(B=h(A)\) is a closed
incidence condition in the coefficients of \(h,A,B\).  Hence the non-graph
condition is constructible.

The desired source strata in \(\mathcal P_m\times\mathcal P_n\) are projections
of constructible subsets of this finite-dimensional incidence space.  Chevalley's
theorem gives constructibility of their images; see Hartshorne
\cite[Ch. II, Ex. 3.19]{Hartshorne}.
\end{proof}

\section{The logarithmic component trichotomy}
\label{sec:component-count}

Let \(k\) be a number field, \(S\) a finite set of places containing the
Archimedean places, and \(\OO_{k,S}\) the ring of \(S\)-integers.  We write \(H\)
for the absolute multiplicative height on \(\Pj^1(k)\), put \(n_k=[k:\Q]\), and
put \(q_{k,S}=\operatorname{rk}\OO_{k,S}^{*}=|S|-1\).

\begin{lemma}[Counting \(S\)-integers of bounded height]
\label{lem:S-height-count}
Let \(k\) be a number field, let \(S\) be a finite set of places containing the
Archimedean places, and put
\[
        n_k=[k:\Q],
        \qquad
        q_{k,S}=\operatorname{rk}\OO_{k,S}^{*}=|S|-1.
\]
There are constants \(c_1,c_2>0\) such that, for \(R\ge2\),
\[
      c_1R^{n_k}(\log R)^{q_{k,S}}
      \le
      \#\{a\in\OO_{k,S}:H(a)\le R\}
      \le
      c_2R^{n_k}(\log R)^{q_{k,S}}.
\]
For every nonempty affine coset \(a_0+\lambda\OO_{k,S}\), \(\lambda\ne0\),
the count
\[
      \#\{a\in a_0+\lambda\OO_{k,S}:H(a)\le R\}
\]
satisfies the same upper and lower bounds, with constants depending on the
coset.  The same is true for
\[
      \#\{a\in I:H(a)\le R\}
\]
for every nonzero fractional \(\OO_{k,S}\)-ideal \(I\), with constants depending
on \(I\).
\end{lemma}

\begin{proof}
The displayed estimate is the standard bounded-height asymptotic for
\(S\)-integers.  We use it only in this degree-one \(S\)-integer form; it follows
directly from Barroero's \(S\)-integer counting theorem \cite{Barroero}.  When
\(S\) consists only of the Archimedean places, the same estimate is the usual
lattice/unit count for algebraic integers, in the same normalization as
Schanuel's height theorem \cite{Schanuel}.  The logarithmic exponent is the full
\(S\)-unit rank, not only the number of finite places in \(S\):
\[
        q_{k,S}=r_1+r_2-1+|S_f|=|S|-1,
\]
where \(S_f\) is the set of finite places in \(S\).  For fractional ideals
and affine cosets, the standard reduction is as follows.  If \(I\) is a
nonzero fractional \(\OO_{k,S}\)-ideal, choose nonzero
\(\alpha,\beta\in k\) with
\[
        \alpha\OO_{k,S}\subseteq I\subseteq \beta\OO_{k,S}.
\]
The existence of \(\beta\) is the definition of a fractional ideal, and any
nonzero \(\alpha\in I\) gives the left inclusion.  Multiplication by a fixed
nonzero element distorts absolute height by a bounded factor, so the same upper
and lower estimates hold for \(I\).  If \(a_0+I\) is a nonempty affine coset,
the inclusions
\[
        a_0+\alpha\OO_{k,S}\subseteq a_0+I\subseteq a_0+\beta\OO_{k,S}
\]
reduce the matter to cosets of principal fractional ideals.  For fixed
\(a_0\in k\) and \(\gamma\in k^*\), the elementary height inequalities
\[
        H(a_0+\gamma u)\le C_{a_0,\gamma}H(u),
        \qquad
        H(u)\le C'_{a_0,\gamma}H(a_0+\gamma u)
\]
hold for all \(u\in k\).  The first inequality follows from the standard
height bounds for addition and multiplication, and the second follows by
writing \(u=\gamma^{-1}((a_0+\gamma u)-a_0)\).  Hence replacing
\(u\) by \(a_0+\gamma u\) distorts height balls only by fixed multiplicative
constants, and the same upper and lower estimates hold for the coset.
\end{proof}

\begin{lemma}[Height growth under rational functions]
\label{lem:height-growth}
Let \(A:\Pj^1_k\to\Pj^1_k\) be the morphism associated to a nonconstant rational
function \(A\in k(t)\), and let \(d=\deg A\).  There are positive constants
\(c_1,c_2\), depending on \(A\), such that
\[
      c_1H(t)^d\le H(A(t))\le c_2H(t)^d
\]
for all \(t\in\Pj^1(k)\).
\end{lemma}

\begin{proof}
This is the standard functoriality of Weil heights for morphisms
\(\Pj^1\to\Pj^1\); the value \(A(t)=\infty\) is interpreted in \(\Pj^1\).  See,
for example, \cite[Ch. 3]{Lang}.
\end{proof}

\begin{lemma}[Integer-valued cosets]
\label{lem:integer-valued-coset}
Let \(A\in k[t]\), and suppose \(A(t_0)\in\OO_{k,S}\) for some \(t_0\in k\).
Then there exists \(\lambda\in k^*\) such that
\[
        A(t_0+\lambda u)\in\OO_{k,S}
        \qquad(u\in\OO_{k,S}).
\]
\end{lemma}

\begin{proof}
Write the Taylor expansion
\[
        A(t_0+Z)=A(t_0)+c_1Z+\cdots+c_dZ^d,
        \qquad c_i\in k.
\]
Choose \(\lambda\in k^*\) so that \(c_i\lambda^i\in\OO_{k,S}\) for all \(i\).
Then \(A(t_0+\lambda u)\in\OO_{k,S}\) for all \(u\in\OO_{k,S}\).
\end{proof}

\begin{lemma}[Polynomial \(S\)-integer value sets]
\label{lem:polynomial-S-value-set}
Let \(A\in k[t]\) have degree \(d\ge1\).  Put
\[
        V_A(B)=
        \#\{A(t):t\in k,
             \ A(t)\in\OO_{k,S},\ H(A(t))\le B\}.
\]
If \(A(k)\cap\OO_{k,S}=\varnothing\), then \(V_A(B)=0\).  Otherwise
\[
        V_A(B)\asymp_{A,k,S}
        B^{n_k/d}(\log(2B))^{q_{k,S}}.
\]
\end{lemma}

\begin{proof}
The inactive case is immediate.  Assume \(A(k)\cap\OO_{k,S}\ne\varnothing\).

For the upper bound, there exists a nonzero \(\delta\in k\), depending only on
\(A\) and \(S\), such that
\[
        A(t)\in\OO_{k,S}\quad\Longrightarrow\quad \delta t\in\OO_{k,S}.
\]
The local denominator bound is as follows.  Write
\[
        A(T)=a_dT^d+a_{d-1}T^{d-1}+\cdots+a_0,
        \qquad a_d\ne0.
\]
Let \(v\notin S\) be finite.  For all but finitely many such \(v\), every
coefficient \(a_i\) is \(v\)-integral and \(a_d\) is a \(v\)-unit.  At such a
good place, if \(v(t)<0\), then the leading term has strictly smaller valuation
than every lower term, and therefore
\[
        v(A(t))=d v(t)<0.
\]
Thus \(A(t)\in\OO_{k,S}\) forces \(v(t)\ge0\) at every good place.

It remains to treat the finite set \(E\) of bad finite places outside \(S\).
For each \(v\in E\), choose an integer \(N_v\le0\) so small that, whenever
\(v(t)<N_v\), the leading term strictly dominates every nonzero lower term and
also \(v(a_d)+d v(t)<0\).  For example, take \(N_v\) strictly smaller than all
finite numbers
\[
        \frac{v(a_i)-v(a_d)}{d-i}\quad(a_i\ne0,\ i<d),
        \qquad\text{and}\qquad
        -\frac{v(a_d)}d.
\]
Then \(v(t)<N_v\) gives
\[
        v(A(t))=v(a_d)+d v(t)<0,
\]
so \(A(t)\in\OO_{k,S}\) forces \(v(t)\ge N_v\) at this bad place.  Choose
\(\delta\ne0\) in the integral ideal \(\prod_{v\in E}\mathfrak p_v^{-N_v}\).
Then \(v(\delta)\ge -N_v\) for all \(v\in E\), and
\(v(\delta)\ge0\) for all other finite \(v\notin S\).  Hence
\(A(t)\in\OO_{k,S}\) implies \(v(\delta t)\ge0\) for every finite
\(v\notin S\), hence \(\delta t\in\OO_{k,S}\).  Therefore the relevant
parameters lie in a fixed fractional \(\OO_{k,S}\)-ideal.  If \(H(A(t))\le B\),
then \Cref{lem:height-growth} gives
\(H(t)\ll_A B^{1/d}\).  By \Cref{lem:S-height-count}, the number of such
parameters is \(O_{A,k,S}(B^{n_k/d}(\log(2B))^{q_{k,S}})\), and the number of
values is no larger.

For the lower bound, choose \(t_0\in k\) with \(A(t_0)\in\OO_{k,S}\).  By
\Cref{lem:integer-valued-coset}, there is \(\lambda\in k^*\) such that
\(A(t_0+\lambda u)\in\OO_{k,S}\) for all \(u\in\OO_{k,S}\).  For a sufficiently
small constant \(c>0\), the height estimate gives
\[
        H(A(t_0+\lambda u))\le B
        \qquad\text{whenever}\qquad H(u)\le cB^{1/d}.
\]
The number of such \(u\) is
\(\gg_{A,k,S}B^{n_k/d}(\log(2B))^{q_{k,S}}\) by
\Cref{lem:S-height-count}.  Each value of \(A\) has at most \(d\) preimages in
\(\Pj^1\), so the same lower order holds for distinct values.  This proves the
lemma.
\end{proof}

\begin{lemma}[Geometrically reducible curves have finitely many rational points]
\label{lem:ngeo-finite}
Let \(C/k\) be an integral affine curve over a number field.  If \(C\) is not
geometrically integral, then \(C(k)\) is finite.
\end{lemma}

\begin{proof}
Let \(\nu:C^\nu\to C\) be the normalization.  Since \(C\) is an affine curve over
a number field, \(C^\nu=\Spec R\) is affine, finite over \(C\), and normal.  The
nonnormal locus of \(C\) is finite, and \(\nu\) is an isomorphism over its
complement.

Let \(L\) be the algebraic closure of \(k\) in \(k(C)=\Frac R\).  Since \(k\) is
perfect, non-geometric integrality is equivalent to \(L\ne k\).  Every element of
\(L\) is algebraic over \(k\), hence satisfies a monic polynomial with
coefficients in \(k\subset R\).  Thus every element of \(L\) is integral over
\(R\).  Because \(R\) is normal, it is integrally closed in \(\Frac R\), and
therefore \(L\subset R\).

If \(C^\nu\) had a \(k\)-rational point, evaluation \(R\to k\) would restrict to
a \(k\)-embedding \(L\hookrightarrow k\), which is impossible when \(L\ne k\).
Hence \(C^\nu(k)=\varnothing\).  Any \(k\)-point of \(C\) outside the finite
nonnormal locus would lift uniquely to \(C^\nu(k)\), so all \(k\)-points of
\(C\) lie in that finite locus.
\end{proof}

Let \(C/k\) be an integral affine curve, let \(C^\nu\) be its affine
normalization, and let \(\widetilde C\) be the smooth projective normalization.
When \(C\) is geometrically integral, write
\[
        D_C=\widetilde C\setminus C^\nu
\]
for the reduced geometric boundary; throughout, \(\#D_C\) means the number of
geometric boundary points.

\begin{definition}[Logarithmic degree]
\label{def:log-degree}
Let \(C/k\) be a geometrically integral affine curve with affine normalization
\(C^\nu=\widetilde C\setminus D_C\).  Its logarithmic degree is
\[
        \lambda(C)=\deg\Omega^1_{\widetilde C}(\log D_C)
        =2\genus(\widetilde C)-2+\#D_C.
\]
For curves, the sign of \(\lambda(C)\) is exactly the trichotomy used below:
\(\lambda=-1\) gives \(\A^1\), \(\lambda=0\) gives \(\Gm\)-type after
splitting the two boundary points, and \(\lambda>0\) is the log-hyperbolic case.
This is a statement over \(\bar k\); the arithmetic estimates keep track of the
field of definition separately through activity and through the finite extension
used to split a log-zero boundary.
\end{definition}

\begin{lemma}[Unit parameter on a split two-boundary normalization]
\label{lem:log-zero-unit-parameter}
Let \(C/k\) be a geometrically integral affine curve over a number field, and
let \(\mathscr C\) be a fixed affine \(\OO_{k,T}\)-model of \(C\).  Suppose that,
after a finite extension \(L/k\), the two boundary points of the affine
normalization split and the affine normalization over \(L\) is identified with
\(\Gm=\Spec L[u,u^{-1}]\).  Then, after enlarging to a finite set of places
\(U\) of \(L\), depending only on \(C\), \(T\), the model, and the chosen
coordinate \(u\), every \(T\)-integral point of \(C\) outside the finite locus
where the normalization is not an isomorphism satisfies
\[
        u(P),u(P)^{-1}\in\OO_{L,U}
\]
after base change to \(L\).  Equivalently, \(u(P)\in\OO_{L,U}^{*}\).
\end{lemma}

\begin{proof}
Write \(\mathscr C=\Spec R_T\), and put
\[
        R_L=k[C]\otimes_k L\subset L(u)
\]
for the generic-fiber coordinate ring after base change to \(L\).  Since the
generic-fiber affine normalization is
\(\Spec L[u,u^{-1}]\), the integral closure of the image of \(R_L\) in
\(L(u)\) is \(L[u,u^{-1}]\).  Thus \(u\) and \(u^{-1}\) are integral over the
generic ring \(R_L\).  Choose monic equations over \(R_L\) satisfied by \(u\) and
by \(u^{-1}\), with coefficients \(r_i,s_i\in R_L\).

The coefficients \(r_i,s_i\) are finitely many elements of
\(R_T\otimes_{\OO_{k,T}}L\).  Enlarge the finite set of places of \(L\) to a set
\(U\), containing all places above \(T\), so that all these coefficients lie in
the image of
\[
        R_T\otimes_{\OO_{k,T}}\OO_{L,U}
        \longrightarrow R_L.
\]
If \(P\) is a \(T\)-integral point outside the finite nonnormal locus, then its
base change is a \(U\)-integral point on the base-changed model, and evaluation
therefore sends all \(r_i\) and \(s_i\) to \(\OO_{L,U}\).  Evaluating the two
monic equations shows that \(u(P)\) and \(u(P)^{-1}\) are integral over
\(\OO_{L,U}\).  Since \(\OO_{L,U}\) is integrally closed, both lie in
\(\OO_{L,U}\), and therefore \(u(P)\in\OO_{L,U}^{*}\).
\end{proof}

\begin{proposition}[Logarithmic component trichotomy for sparse lifting]
\label{prop:component-count}
Let \(S\) and \(T\) be finite sets of places of \(k\) containing the Archimedean
places.  Let \(C/k\) be an integral affine curve, let \(x\in k[C]\) be
nonconstant, and let \(\mathcal A\subseteq C(k)\) be contained in the set of
\(T\)-integral points for a fixed affine model of \(C\).  Define
\[
        M_C(B;S,T)=\#\{a\in\OO_{k,S}:H(a)\le B,
        \ a=x(P)\text{ for some }P\in\mathcal A\}.
\]
Then:
\begin{enumerate}[label=\textup{(\alph*)}]
    \item If \(C\) is not geometrically integral, then \(M_C(B;S,T)=O_C(1)\).
    \item If \(C\) is geometrically integral and either \(\genus(\widetilde C)>0\) or
    \(\#D_C\ge3\), then \(M_C(B;S,T)=O_{C,k,T}(1)\).
    \item If \(C\) is geometrically integral, \(\genus(\widetilde C)=0\), and
    \(\widetilde C\not\cong\Pj^1_k\), then \(M_C(B;S,T)=O_C(1)\).
    \item If \(C\) is geometrically integral, \(\widetilde C\cong\Pj^1_k\), and
    \(\#D_C=1\), then there is a coordinate \(t\) on
    \(\widetilde C\setminus D_C\cong\A^1\) such that
    \[
          x=A(t),\qquad A\in k[t].
    \]
    Let \(d=\deg A\).  If \(A(k)\cap\OO_{k,S}=\varnothing\), then
    \(M_C(B;S,T)=O_C(1)\).  Otherwise
    \[
          M_C(B;S,T)\ll_{C,k,S,T} B^{n_k/d}(\log(2B))^{q_{k,S}}.
    \]
    \item If \(C\) is geometrically integral, \(\widetilde C\cong\Pj^1_k\), and
    \(\#D_C=2\), then there is a finite extension \(L/k\) splitting the two
    boundary points and a finite set of places \(U\) of \(L\), depending only on
    \(C,S,T\) and the affine model, such that
    \[
          M_C(B;S,T)\ll_{C,k,S,T} (\log(2B))^{\rho_C},
          \qquad
          \rho_C=\operatorname{rk}\OO_{L,U}^{*}.
    \]
\end{enumerate}
\end{proposition}

\begin{proof}
The cases in the statement are the arithmetic form of the sign decomposition in
\Cref{def:log-degree}.  Part (a) is \Cref{lem:ngeo-finite}.  Passing from \(C\) to its affine
normalization changes the count by at most a constant, since the normalization is
an isomorphism away from finitely many affine points.  In all cases below we
therefore discard this finite nonnormal affine locus once and for all, and count
points on the affine normalization.  We also record the integrality transfer used
below.  Let \(\mathscr C\) be the fixed affine
\(\OO_{k,T}\)-model of \(C\), and let \(\mathscr C^\nu\) be the normalization of
this model in \(k(C)\), after removing finitely many additional primes from the
base.  Equivalently, after replacing \(T\) by a finite set \(T^\nu\supseteq T\),
the affine normalization is finite and integral over \(\mathscr C\).  Hence every
\(T\)-integral point of \(C\) lying outside the finite nonnormal locus lifts to a
\(T^\nu\)-integral point of the affine normalization: the lifted coordinate
values are integral over \(\OO_{k,T^\nu}\), and \(\OO_{k,T^\nu}\) is integrally
closed.  Enlarging \(T\) in this fixed way only changes the constants in the
following estimates.

For the remaining geometrically integral cases, note first that \(D_C\) is
nonempty.  Indeed, if \(D_C=\varnothing\), then \(C^\nu=\widetilde C\) is
projective, so \(x\) extends to a regular function on \(\widetilde C\), hence is
constant, contradicting the hypothesis.

Part (b) is Siegel--Mahler: an affine curve over a number field has finitely many
\(T\)-integral points unless its normalization has genus zero and at most two
points at infinity; see Lang \cite[Ch. 7]{Lang} or the subspace-theorem proof of
Corvaja--Zannier \cite{CorvajaZannierSiegel}.

For (c), a smooth projective genus-zero curve over \(k\) which is not isomorphic
to \(\Pj^1_k\) has no \(k\)-rational point.  Thus rational points of the affine
curve can occur only in the finite exceptional locus where the normalization map
is not an isomorphism.

For (d), the unique geometric point of \(D_C\) is Galois invariant, hence defined
over \(k\).  Its complement is \(\A^1_k\); choose a coordinate \(t\).  Since
\(x\) is regular away from the unique point at infinity, \(x=A(t)\) with
\(A\in k[t]\).  The normalized part contributes only values in
\(A(k)\cap\OO_{k,S}\), so if this set is empty the only possible contribution is
from the finite normalization-exceptional locus.

Assume now that \(A(k)\cap\OO_{k,S}\ne\varnothing\).  The normalized part of the
component contributes only values belonging to
\[
        \{A(t):t\in k,\ A(t)\in\OO_{k,S},\ H(A(t))\le B\}.
\]
By \Cref{lem:polynomial-S-value-set}, this value set has size
\(O_{A,k,S}(B^{n_k/d}(\log(2B))^{q_{k,S}})\).  This gives the asserted upper
bound, with the dependence on \(T\) entering only through the fixed integral
model and the finite normalization-exceptional locus.

For (e), the preliminary reduction has already discarded the finite affine
locus where the normalization map is not an isomorphism; it is therefore enough
to bound the lifted points on the affine normalization.  Pass to a finite extension
\(L/k\) over which the two boundary points are rational.  Choose a coordinate
\(u\) sending them to \(0\) and \(\infty\).  The affine normalization becomes
\(\Gm\), and \(x=A(u)\) for a Laurent polynomial \(A\in L[u,u^{-1}]\).
By \Cref{lem:log-zero-unit-parameter}, after enlarging to a fixed finite set of
places \(U\) of \(L\), every lifted integral point satisfies
\[
        u(P)\in\OO_{L,U}^{*}.
\]

The absolute multiplicative Weil height is unchanged by passing from \(k\) to
the finite extension \(L\), since it is normalized with respect to the product
formula.  Therefore the condition \(H(x(P))\le B\), with \(x(P)\in k\), is the
same height condition after viewing \(x(P)\) as an element of \(L\).  The Laurent
polynomial \(A(u)\) defines a nonconstant rational map
\(\Pj^1_L\to\Pj^1_L\).  By the height-growth estimate for rational maps,
\Cref{lem:height-growth}, the condition \(H(x(P))=H(A(u(P)))\le B\) gives
\(H(u(P))\ll_A B^{C_A}\) for some constant \(C_A>0\).  Dirichlet's unit theorem
then gives
\[
        \#\{u\in\OO_{L,U}^{*}:H(u)\le B^{C_A}\}=O((\log(2B))^{\rho_C}),
\]
where \(\rho_C=\operatorname{rk}\OO_{L,U}^{*}\).  This proves (e).
\end{proof}

\begin{lemma}[Counting Laurent images of finitely generated unit cosets]
\label{lem:unit-coset-laurent-count}
Let \(L\) be a number field, let \(\Gamma\subset L^*\) be a finitely generated
subgroup of rank \(r\), let \(\eta\in L^*\), and let
\(A\in L(u)\) be a nonconstant rational function.  Then
\[
        \#\{A(u):u\in \eta\Gamma,
              \ H(A(u))\le B\}
        \ll_{A,L,\eta,\Gamma} (\log(2B))^r.
\]
If \(r\ge1\), then the reverse bound also holds:
\[
        \#\{A(u):u\in \eta\Gamma,
              \ H(A(u))\le B\}
        \gg_{A,L,\eta,\Gamma} (\log(2B))^r.
\]
For \(r=0\) the set is finite.  In both displays, the finitely many
parameters with \(A(u)=\infty\) are omitted.
\end{lemma}

\begin{proof}
Choose generators of \(\Gamma\), and choose a finite set \(U\) of places of
\(L\), containing the Archimedean places, all finite places at which one of the
generators has nonzero valuation, and all finite places at which \(\eta\) has
nonzero valuation.  Then
\[
        \Gamma\subset \OO_{L,U}^*,
        \qquad
        \eta\in\OO_{L,U}^*.
\]
By Dirichlet's unit theorem, \(\OO_{L,U}^*\) modulo torsion is a lattice under
the logarithmic absolute-value map; see, for example, \cite{Neukirch}.  The
subgroup \(\Gamma/\Gamma_{\mathrm{tors}}\) is therefore a lattice of rank
\(r\) in its real span.  The restriction of the logarithmic height is a proper
norm on this lattice: if all logarithmic absolute values are bounded, then only
finitely many elements of \(\Gamma\) occur modulo torsion.  The inequalities
\(H(u)\le R\) cut out a bounded symmetric convex body in this rank-\(r\) real
space whose volume is \(\asymp (\log R)^r\).  Standard lattice-point counting
gives
\[
        \#\{u\in\eta\Gamma:H(u)\le R\}\asymp_{L,\eta,\Gamma}
        (\log(2R))^r
\]
for \(r\ge1\), and the set is finite for \(r=0\).

The finitely many poles of \(A\) in \(\eta\Gamma\) are discarded and do not
affect the estimates.  Let \(D=\deg(A:\Pj^1_L\to\Pj^1_L)\).  By height
functoriality,
\[
        c_1H(u)^D\le H(A(u))\le c_2H(u)^D
\]
uniformly on \(\Pj^1(L)\), since this is an inequality on \(\Pj^1\).  Therefore
\(H(A(u))\le B\) implies \(H(u)\ll_A B^{1/D}\), giving the upper bound.  For the
lower bound, \(H(u)\le cB^{1/D}\) implies \(H(A(u))\le B\).  The rational map
\(A\) has degree \(D\), so every value has at most \(D\) preimages in \(\Pj^1\).
Passing from parameters to distinct \(A\)-values therefore loses only a bounded
factor.  This proves the lemma.
\end{proof}

\begin{lemma}[Two-sided Laurent integrality on a unit lattice]
\label{lem:two-sided-laurent-integrality}
Let \(L\) be a number field, let \(\Gamma\subset L^*\) be a finitely generated
subgroup, and let \(\eta\in L^*\).  Let
\[
        A_1,\ldots,A_q\in L[u,u^{-1}]
\]
be Laurent polynomials such that every \(A_i\) has both a positive and a negative
extreme exponent.  For each \(i\), let \(S_i\) be a finite set of places of
\(L\) containing the Archimedean places.  Then
\[
        \{u\in\eta\Gamma: A_i(u)\in\OO_{L,S_i}
          \text{ for }1\le i\le q\}
\]
is a finite union of cosets of subgroups of \(\Gamma\).  This finite union is
effective once generators for \(\Gamma\), the coefficients of the \(A_i\), and
the sets \(S_i\) are given.
\end{lemma}

\begin{proof}
Choose generators \(\gamma_1,\ldots,\gamma_t\) of \(\Gamma\).  Enlarge a finite
set \(R\) of finite places of \(L\) so that, outside \(R\), all coefficients of
all \(A_i\) are integral, the two extreme coefficients of every \(A_i\) are
units, and
\[
        v_w(\eta)=v_w(\gamma_1)=\cdots=v_w(\gamma_t)=0.
\]
At a place \(w\notin R\), every \(u\in\eta\Gamma\) is a unit.  Hence every
monomial \(u^\ell\), positive or negative, is a unit, and all integrality
conditions are automatic at \(w\).  It remains to treat
\[
        R_0=\bigcup_i(R\setminus S_i).
\]

Fix \(i\) and a place \(w\in R_0\setminus S_i\).  Write
\[
        A_i(u)=\sum_{\ell=-r_i}^{s_i}a_{i,\ell}u^\ell,
        \qquad r_i,s_i>0,
        \qquad a_{i,-r_i}a_{i,s_i}\ne0.
\]
If \(v_w(u)\) is sufficiently large, the negative extreme term
\(a_{i,-r_i}u^{-r_i}\) has strictly smaller valuation than every other term and
has negative valuation.  Then \(A_i(u)\notin\OO_{L,w}\).  If \(v_w(u)\) is
sufficiently negative, the positive extreme term \(a_{i,s_i}u^{s_i}\) gives the
same conclusion.  Thus the condition \(A_i(u)\in\OO_{L,w}\) forces
\(v_w(u)\) to lie in an explicitly computable finite interval.  Intersecting
these intervals over the finitely many pairs \((i,w)\) gives a finite set of
allowed valuation vectors
\[
        \mathbf c=(c_w)_{w\in R_0}\in\Z^{R_0}.
\]

The valuation map
\[
        \Gamma\longrightarrow \Z^{R_0},
        \qquad \gamma\longmapsto (v_w(\gamma))_{w\in R_0}
\]
is a homomorphism.  For a fixed allowed vector \(\mathbf c\), the set of
\(u\in\eta\Gamma\) with \(v_w(u)=c_w\) for all \(w\in R_0\) is either empty or a
coset
\[
        \eta_{\mathbf c}\Gamma_0,
        \qquad
        \Gamma_0=\{\gamma\in\Gamma:v_w(\gamma)=0\text{ for all }w\in R_0\}.
\]
It remains to impose residual congruence conditions on this fixed coset.

Fix such a coset and a place \(w\in R_0\).  Choose a uniformizer \(\pi_w\).  For
\(u\in\eta_{\mathbf c}\Gamma_0\), write
\[
        u=\pi_w^{c_w}\varepsilon_w(u),
        \qquad \varepsilon_w(u)\in\OO_{L,w}^{*}.
\]
For each relevant pair \((i,w)\), set
\[
        b_{i,w}=\min_{\ell}\bigl(v_w(a_{i,\ell})+\ell c_w\bigr).
\]
Then
\[
        B_{i,w}(Z)=\pi_w^{-b_{i,w}}A_i(\pi_w^{c_w}Z)
        \in \OO_{L,w}[Z,Z^{-1}]
\]
and at least one coefficient of \(B_{i,w}\) is a \(w\)-adic unit.  The condition
\(A_i(u)\in\OO_{L,w}\) is equivalent to
\[
        B_{i,w}(\varepsilon_w(u))\in \pi_w^{-b_{i,w}}\OO_{L,w}.
\]
If \(b_{i,w}\ge0\), this condition is automatic.  If \(b_{i,w}<0\), it is the
congruence
\[
        B_{i,w}(\varepsilon_w(u))\equiv0\pmod{\mathfrak p_w^{-b_{i,w}}}.
\]
Because \(B_{i,w}\) is a Laurent polynomial with integral coefficients, and the
inverse map is continuous on \(\OO_{L,w}^*\), this congruence is determined by
the class of \(\varepsilon_w(u)\) in the finite group
\[
        \OO_{L,w}^{*}/(1+\mathfrak p_w^{N_{i,w}}),
        \qquad N_{i,w}=\max\{1,-b_{i,w}\}.
\]
Equivalently, on \(\eta_{\mathbf c}\Gamma_0\) every remaining local integrality
condition is the inverse image of a subset of a finite quotient of \(\Gamma_0\).
The intersection of finitely many such inverse images is a finite union of cosets
of a finite-index subgroup of \(\Gamma_0\).

Taking the finite union over the allowed valuation vectors \(\mathbf c\) proves
that the original set is a finite union of cosets of subgroups of \(\Gamma\).  The
same construction is effective: the dominance inequalities give the allowed
valuation vectors, and the finite groups
\(\OO_{L,w}^{*}/(1+\mathfrak p_w^{N_{i,w}})\) give the residual congruence
conditions.
\end{proof}

\begin{remark}[Two-sidedness is part of the model]
\label{rem:two-sidedness-model-hypothesis}
The two-sided pole condition is a hypothesis on the chosen affine model and
generating set.  It is not automatic for an arbitrary affine model of a
two-boundary rational curve.  Without both extreme exponents, the valuation
bounding step in \Cref{lem:two-sided-laurent-integrality} can fail; for instance,
one-sided monomials on a rank-one unit group can impose a half-line of valuations,
not a finite union of cosets.  In the fiber-product applications, the standard
generators \(X,Y\), and later \(X,Y_1,\ldots,Y_r\), satisfy the two-sided
condition by primitive contact at infinity, because each has a pole at each
geometric boundary point.
\end{remark}

\begin{lemma}[Descent cosets for split two-boundary parameters]
\label{lem:split-two-boundary-descent-coset}
Let \(L/k\) be a finite Galois extension, and let a split two-boundary
normalization be identified with \(\Gm=\Spec L[u,u^{-1}]\).  Suppose the
semilinear Galois action on the split coordinate has the form
\[
        {}^\sigma u=\alpha_\sigma u^{\varepsilon_\sigma},
        \qquad
        \alpha_\sigma\in L^*,\quad
        \varepsilon_\sigma\in\{\pm1\}
        \qquad(\sigma\in\operatorname{Gal}(L/k)).
\]
Then the set of split parameters descending to \(k\)-rational points is either
empty or a coset
\[
        z_0\Gamma^0,
        \qquad
        \Gamma^0=
        \{\gamma\in L^*: \sigma(\gamma)=\gamma^{\varepsilon_\sigma}
          \text{ for every }\sigma\}.
\]
For every finite set of places \(U\) of \(L\), its intersection with
\(\OO_{L,U}^*\) is either empty or a coset of
\(\Gamma^0\cap\OO_{L,U}^*\), hence a coset of a finitely generated group.
\end{lemma}

\begin{proof}
A split point with parameter \(z\in L^*\) is fixed by descent precisely when
\[
        \sigma(z)=({}^\sigma u)(P)=\alpha_\sigma z^{\varepsilon_\sigma}
        \qquad(\sigma\in\operatorname{Gal}(L/k)).
\]
Conversely, a parameter satisfying these equations defines an \(L\)-point of the
split normalization fixed by the Galois descent datum, hence descends to a
\(k\)-point of the original normalization.  If there is no solution, the descent
set is empty.  If \(z_0\) is one solution and \(z\) is another, then
\(z/z_0\in\Gamma^0\).  Conversely, multiplying \(z_0\) by any element of
\(\Gamma^0\) gives another solution.  This proves the coset statement.
Intersecting the coset with \(\OO_{L,U}^*\) is either empty or, after choosing
one element \(z_1\) in the intersection, equal to
\[
        z_1(\Gamma^0\cap\OO_{L,U}^*).
\]
The subgroup \(\Gamma^0\cap\OO_{L,U}^*\) is finitely generated because it is a
subgroup of the finitely generated group \(\OO_{L,U}^*\).
\end{proof}

\begin{theorem}[Admissible unit lattice for two-infinity components]
\label{thm:admissible-unit-lattice}
Let \(k\) be a number field, let \(S\subseteq T\) be finite sets of places
containing the Archimedean places, and let \(C\) be a geometrically integral
affine curve.  Assume
\[
        \widetilde C\cong\Pj^1_k,
        \qquad
        \#D_C=2.
\]
Fix nonconstant regular functions
\[
        F_0=X,\ F_1,\ldots,F_q\in k[C]
\]
which, together with \(k\), generate \(k[C]\) as a \(k\)-algebra.  Assume that \(T\) has been chosen
large enough, and denominators have been cleared, so that these generators define
a fixed affine \(\OO_{k,T}\)-model \(\mathscr C\).  Assume that every \(F_i\) has
a pole at both geometric boundary points.

Define
\[
\begin{aligned}
        M_C^{\adm}(B;S,T)=\#\{a\in\OO_{k,S}:{}& H(a)\le B,\\
        & a=X(P)\text{ for some }P\in \mathscr C(\OO_{k,T})\cap C(k)\}.
\end{aligned}
\]
Then there are a finite extension \(L/k\), a coordinate \(u\) identifying the
split affine normalization \(C^\nu_L\) with \(\Gm\), a finite set of places
\(U\) of \(L\), and a finite union of cosets
\[
        \mathcal U_C=\bigcup_{j=1}^{s_C}\eta_j\Gamma_j
        \subset \OO_{L,U}^*,
\]
where each \(\Gamma_j\) is a finitely generated subgroup of \(\OO_{L,U}^*\),
with the following property.  The set \(\mathcal U_C\) is the set of split
parameters satisfying descent, the \(S\)-integrality condition for \(X\), and the
\(T\)-integrality conditions for the remaining generators.  In particular,
\(A_0(u)\in\OO_{k,S}\) for every \(u\in\mathcal U_C\).  Outside finitely many
values of \(X\), the inputs counted by \(M_C^{\adm}(B;S,T)\) are exactly
\[
        \{A_0(u):u\in\mathcal U_C\},
\]
where \(A_0\in L[u,u^{-1}]\) is the Laurent expression for \(X\) on
\(C^\nu_L\).

For the fixed affine model and generating set above, put
\[
        \rho_C^{\adm}(S,T)
        =\max_{1\le j\le s_C}\operatorname{rk}\Gamma_j,
\]
with \(\rho_C^{\adm}(S,T)=-1\) if \(\mathcal U_C=\varnothing\).  The notation
suppresses the fixed model and generators and does not define an
intrinsic invariant of the abstract curve.  Then
\[
        M_C^{\adm}(B;S,T)=O_{C,S,T}(1)
        \qquad\text{if }\rho_C^{\adm}(S,T)\le0,
\]
and, if \(\rho_C^{\adm}(S,T)\ge1\),
\[
        M_C^{\adm}(B;S,T)
        \asymp_{C,S,T} (\log(2B))^{\rho_C^{\adm}(S,T)}.
\]
For the fixed integral model and generators above, the positive logarithmic
exponent \(\max\{0,\rho_C^{\adm}(S,T)\}\) is independent of the auxiliary
splitting field, of the chosen split coordinate, and of enlarging the auxiliary
place set \(U\), up to the finite exceptional set above.  Equivalently, the
actual admissible parameter set is determined by descent together with the fixed
\(S\)- and \(T\)-integrality conditions for the chosen generators; the auxiliary
\(U\)-unit group is only an ambient lattice used to describe it.  No independence from changing the integral model is asserted, and the
distinction between \(\rho_C^{\adm}=-1\) and \(\rho_C^{\adm}=0\) has no
intrinsic significance once finite exceptional sets are allowed to vary.
Moreover, the finite union of cosets, and hence the displayed exponent, is
effectively computable relative to the following standard computations: the
normalization of the curve, the two boundary points and a splitting field for
them, a split coordinate \(u\), the Laurent expressions \(A_i(u)\), the
semilinear Galois descent action on \(u\), the resulting descent equations
\(\sigma(z)=\alpha_\sigma z^{\varepsilon_\sigma}\), including deciding
whether they have a solution and finding a base solution when they do, unit
groups, and congruence images in number fields.
\end{theorem}

\begin{proof}
Choose a finite Galois extension \(L/k\) splitting the two boundary points, and
choose a coordinate \(u\) sending them to \(0\) and \(\infty\).  Then
\[
        C^\nu_L\cong\Gm=\Spec L[u,u^{-1}],
        \qquad
        F_i=A_i(u)\quad(0\le i\le q)
\]
with \(A_i\in L[u,u^{-1}]\).  By the two-sided hypothesis on the generators,
every Laurent polynomial \(A_i\) has both a positive and a negative extreme
exponent.

\emph{Step 1: descent.}  We record the descent calculation in this fixed
coordinate; it is the situation isolated in
\Cref{lem:split-two-boundary-descent-coset}.  For \(\sigma\in G=\operatorname{Gal}(L/k)\), use the
semilinear action on functions
\[
        ({}^\sigma h)(P)=\sigma\bigl(h(\sigma^{-1}P)\bigr).
\]
It fixes \(k(C)\) and sends the coordinate \(u\) on the split pair
\((\Pj^1,\{0,\infty\})\) to another coordinate with the same unordered boundary
pair.  Hence
\[
        {}^\sigma u=\alpha_\sigma u^{\varepsilon_\sigma},
        \qquad
        \alpha_\sigma\in L^*,\quad \varepsilon_\sigma\in\{\pm1\}.
\]
The cocycle identities are
\[
        \varepsilon_{\sigma\tau}=\varepsilon_\sigma\varepsilon_\tau,
        \qquad
        \alpha_{\sigma\tau}=\sigma(\alpha_\tau)\alpha_\sigma^{\varepsilon_\tau}.
\]
If \(P\in C^\nu(k)\) and \(z=u(P)\), then
\[
        \sigma(z)=({}^\sigma u)(P)=\alpha_\sigma z^{\varepsilon_\sigma}
        \qquad(\sigma\in G).
\]
Conversely, a point \(z\in L^*\) satisfying these equations is fixed by descent
and therefore gives a \(k\)-rational point of the normalization.  Thus the set
of split parameters descending to \(k\) is exactly the solution set of these
relations.  If it is nonempty and \(z_0\) is one solution, then it is the coset
\[
        z_0\Gamma^0,
        \qquad
        \Gamma^0=
        \{\gamma\in L^*: \sigma(\gamma)=\gamma^{\varepsilon_\sigma}
          \text{ for every }\sigma\in G\}.
\]
Indeed, the quotient of two solutions belongs to \(\Gamma^0\), and multiplying
one solution by an element of \(\Gamma^0\) gives another solution.  For the
effectivity assertion, this step requires computing the semilinear action, hence
the data \(\alpha_\sigma,\varepsilon_\sigma\), and solving the resulting
multiplicative descent torsor, or proving it has no solution.  This is included
among the effective inputs in the statement; once a base solution is known, the
remaining admissible set is obtained by imposing unit and congruence conditions
on the subgroup \(\Gamma^0\).

\emph{Step 2: restriction to units.}  By
\Cref{lem:log-zero-unit-parameter}, after enlarging a finite set of places
\(U\) of \(L\), every \(\OO_{k,T}\)-integral point counted by
\(M_C^{\adm}\) has \(u(P)\in\OO_{L,U}^*\).  We enlarge \(U\) once more, if
necessary, so that it contains all places above \(S\) and \(T\) and all places
at which the coefficients appearing below are not integral.  Intersecting the
descent coset with \(\OO_{L,U}^*\) gives either the empty set or a coset
\[
        z_1\Gamma_U,
        \qquad
        \Gamma_U=\Gamma^0\cap\OO_{L,U}^*,
\]
of a finitely generated subgroup of \(\OO_{L,U}^*\).  The finite generation is
Dirichlet's unit theorem applied to the subgroup \(\Gamma_U\).

\emph{Step 3: the local integrality conditions.}  On the descended unit coset
we impose
\[
        A_0(z)\in\OO_{L,S_L},
        \qquad
        A_i(z)\in\OO_{L,T_L}\quad(1\le i\le q),
\]
where \(S_L\) and \(T_L\) are the sets of places of \(L\) above \(S\) and
\(T\).  Apply \Cref{lem:two-sided-laurent-integrality} to the coset
\(z_1\Gamma_U\), with the place set \(S_L\) for \(A_0\) and the place set
\(T_L\) for \(A_i\), \(i\ge1\).  Since all \(A_i\) are two-sided, the result is a
finite union of cosets
\[
        \mathcal U_C=\bigcup_j\eta_j\Gamma_j\subseteq\OO_{L,U}^*,
\]
which is exactly the subset of split parameters satisfying the descent equations
and all displayed integrality conditions.  The construction in
\Cref{lem:two-sided-laurent-integrality} shows explicitly what these cosets are:
first finitely many bad-place valuation vectors are allowed, and then finitely
many residue conditions in finite quotients of unit groups are imposed.

\emph{Step 4: comparison with the original model.}  If \(z\in\mathcal U_C\),
descent gives a point \(P_z\in C^\nu(k)\).  Since the functions \(F_i\) are
defined over \(k\), the values \(A_i(z)=F_i(P_z)\) lie in \(k\).  Hence
\[
        A_0(z)\in k\cap\OO_{L,S_L}=\OO_{k,S},
        \qquad
        A_i(z)\in k\cap\OO_{L,T_L}=\OO_{k,T}\quad(1\le i\le q).
\]
Because the generators \(F_0=X,F_1,\ldots,F_q\) define the fixed affine
\(\OO_{k,T}\)-model \(\mathscr C\), these conditions are exactly the assertion
that \(P_z\in\mathscr C(\OO_{k,T})\) and \(X(P_z)\in\OO_{k,S}\).  Conversely,
any counted \(\OO_{k,T}\)-point outside the finite affine locus where the
normalization is not an isomorphism gives, after base change to \(L\), a split
parameter satisfying the same descent and integrality conditions.  Therefore,
outside finitely many affine points and hence finitely many \(X\)-values, the
inputs counted by this component are precisely
\[
        \{A_0(u):u\in\mathcal U_C\}\subseteq\OO_{k,S}.
\]

\emph{Step 5: counting, independence, and effectivity.}  The counting statement
follows from \Cref{lem:unit-coset-laurent-count}, applied to the nonconstant
Laurent map \(A_0(u)=X\).  If the maximal subgroup rank among the cosets is
zero, the admissible parameter set is finite.  If the maximal rank is positive,
a coset of that rank gives the matching lower bound, and the finite degree of
\(A_0:\Pj^1\to\Pj^1\) loses only a bounded factor when parameters are replaced
by distinct \(X\)-values.

It remains to justify the independence claim.  First fix one splitting field
\(L\) and compare two split coordinates after passing to a common finite Galois
extension containing both.  Any two coordinates on the split pair
\((\Pj^1,\{0,\infty\})\) have the form
\[
        u'=a u \quad\text{or}\quad u'=a/u
        \qquad(a\in L^*).
\]
After enlarging the finite set of places to make \(a\) a unit, the map
\(u\mapsto u'\) is a bijection between the two descended integral parameter
sets.  It sends a coset \(\eta\Gamma\) to either \((a\eta)\Gamma\) or
\((a/\eta)\Gamma^{-1}\), so subgroup ranks are preserved.

Now compare a splitting field \(L\) with a larger finite Galois splitting field
\(L'/k\), using the same coordinate \(u\) viewed over \(L'\).  For every
\(\tau\in\operatorname{Gal}(L'/L)\), the semilinear action fixes \(u\), so the
\(L'\)-descent equation includes
\[
        \tau(z)=z.
\]
Thus every \(L'\)-admissible split parameter actually lies in \(L^*\).  After
choosing the place set of \(L'\) to contain exactly the places above the chosen
place set of \(L\), an element \(z\in L^*\) is an \(L'\)-unit outside those places
if and only if it is an \(L\)-unit outside the original places.  The local
integrality and descent conditions are the same conditions on the same descended
\(k\)-point.  Hence the admissible parameter set over \(L'\) is identified with
the admissible parameter set over \(L\).  More concretely, the extra valuation
and residue conditions obtained in the larger field restrict to inverse images
of subsets of finite quotients of the same finitely generated subgroup of
\(L^*\).  Thus every nonempty refined coset is a finite union of cosets of
finite-index subgroups of the corresponding original subgroup.  These subgroups
have the same rank, so the maximal positive rank is unchanged.

It is also independent of the auxiliary place set \(U\) once \(U\) contains the
finite set required in \Cref{lem:log-zero-unit-parameter} and the finitely many
coefficient and descent denominators used above.  For the fixed affine model and
the fixed generators \(F_0,\ldots,F_q\), the actual admissible parameter set is
defined by descent together with the displayed \(S\)- and \(T\)-integrality
conditions on those generators; \(U\) is only an ambient unit group in which this
set is presented.  If \(U'\supseteq U\) is a larger such set, then any parameter
counted using \(U'\) and satisfying the same descent and generator-integrality
conditions gives a \(\OO_{k,T}\)-integral point of the fixed model.  Applying \Cref{lem:log-zero-unit-parameter} to that
point forces the parameter already to be a \(U\)-unit.  Thus the actual
admissible parameter set is the same; only its presentation as a finite union of
cosets may be refined by finite-index congruence conditions, and the maximal
positive rank is unchanged.

Changing the finite normalization-exceptional set changes only finitely many
\(X\)-values.  Thus \(\max\{0,\rho_C^{\adm}(S,T)\}\) is independent of the
splitting field, split coordinate, and auxiliary place set for the fixed model
and generators.  The integral model itself is part of the data, and changing it
may change the admissible cosets.

The construction is effective in the stated relative sense: compute the affine and
projective normalizations, the two boundary points, a splitting field for those
points, a coordinate \(u\) on the split pair, the Laurent expressions \(A_i(u)\),
and the semilinear descent data
\(\alpha_\sigma,\varepsilon_\sigma\).  Solve the descent equations
\(\sigma(z)=\alpha_\sigma z^{\varepsilon_\sigma}\), or prove that they have
no solution; if a solution \(z_0\) exists, the descent set is
\(z_0\Gamma^0\).  Then compute the relevant \(U\)-unit group, enumerate the
bad-place valuation vectors from \Cref{lem:two-sided-laurent-integrality}, and
compute the finite congruence images of the resulting finitely generated
abelian groups.
\end{proof}

\section{Arithmetic corollaries over number fields}
\label{sec:number-field}

Let \(k\) be a number field, let \(S\) be a finite set of places containing the
Archimedean places, and let \(f,g\in k[x]\) be nonconstant.  Define
\[
        C_{f,g}=(f(X)-g(Y)=0)_{\red}\subset\A^2_k.
\]
A component \(C\) of \(C_{f,g}\) is a \emph{graph component} if
\[
        C=(Y=h(X))
\]
for some \(h\in\HH_{f,g}(k)\).  By
\Cref{lem:degree-one-graph-components}, these are exactly the components of
generic degree one over \(\A^1_X\).  For a non-graph
component \(C\), put
\[
        d_X(C)=[k(C):k(X)].
\]
Equivalently, \(d_X(C)\) is the degree of the projection \(C\to\A^1_X\).  Since
graph components have been removed, \(d_X(C)\ge2\).

Let \(\RR_1(f,g;k)\) be the set of non-graph irreducible components \(C\) of
\(C_{f,g}\) such that \(C\) is geometrically integral,
\(\widetilde C\cong\Pj^1_k\), and \(\#D_C=1\).  For \(C\in\RR_1(f,g;k)\), choose
\[
        X=A_C(t),\qquad Y=B_C(t),\qquad A_C,B_C\in k[t].
\]
We call \(C\) \emph{\(S\)-active} if \(A_C(k)\cap\OO_{k,S}\ne\varnothing\), and
write
\[
        \RR_1^{\act}(f,g;k,S)
        =\{C\in\RR_1(f,g;k):A_C(k)\cap\OO_{k,S}\ne\varnothing\}.
\]
This definition is independent of the chosen affine coordinate \(t\), because
another coordinate has the form \(at+b\) with \(a\in k^*\), \(b\in k\).  Let
\(\RR_2(f,g;k)\) be the set of non-graph irreducible components \(C\) such
that \(C\) is geometrically integral, \(\widetilde C\cong\Pj^1_k\), and
\(D_C\) consists of two geometric points.  These two points may be conjugate
over \(k\).

\begin{lemma}[Uniform denominator control]
\label{lem:denominator-number-field}
There is a finite set of places \(S'\supseteq S\), depending only on
\(f,g,k,S\), such that for every \(a\in\OO_{k,S}\) and every \(y\in k\),
\[
        g(y)=f(a)
        \quad\Longrightarrow\quad
        y\in\OO_{k,S'}.
\]
\end{lemma}

\begin{proof}
Let \(S'\) be obtained from \(S\) by adjoining the finite places at which some
coefficient of \(f\) or \(g\) is not integral, or at which the leading coefficient
of \(g\) is not a unit.  Let \(v\notin S'\) be non-Archimedean and suppose
\(a\in\OO_{k,S}\).  Since \(S\subseteq S'\), the element \(a\) is \(v\)-integral,
and hence \(f(a)\) is \(v\)-integral.  If \(v(y)<0\), the leading term of
\(g(y)\) has strictly smaller \(v\)-adic valuation than all lower terms, because
the leading coefficient is a \(v\)-adic unit and the other coefficients are
\(v\)-integral.  Thus \(v(g(y))<0\), contradicting \(g(y)=f(a)\).  Hence
\(v(y)\ge0\) for all \(v\notin S'\), which is precisely \(y\in\OO_{k,S'}\).
\end{proof}

Define the new lifting set
\[
\begin{aligned}
 \mathscr N^{\new}_{f,g,k,S}
 =\Bigl\{a\in\OO_{k,S}:{}&
       \exists y\in k\text{ with }g(y)=f(a),\\
       &y\ne h(a)\text{ for every }h\in\HH_{f,g}(k)\Bigr\}.
\end{aligned}
\]
Its height truncation is
\[
        N^{\new}_{f,g,k,S}(B)
        =
        \#\{a\in\mathscr N^{\new}_{f,g,k,S}:H(a)\le B\}.
\]

\begin{lemma}[New lifts lie on non-graph components]
\label{lem:new-lifts-nongraph}
Let \(a\in k\) and \(y\in k\) satisfy \(g(y)=f(a)\).  If
\(y\ne h(a)\) for every \(h\in\HH_{f,g}(k)\), then \((a,y)\) lies on at least
one non-graph component of \(C_{f,g}\), and on no graph component.  Consequently,
the inputs counted by \(N^{\new}_{f,g,k,S}(B)\) are accounted for by non-graph
components; component intersections affect only finitely many affine points and
affect the stated upper bounds only by a constant.
\end{lemma}

\begin{proof}
If \((a,y)\) lies on a graph component \(Y=h(X)\), then \(y=h(a)\), contrary to
the defining condition for a new lift.  Since \((a,y)\in C_{f,g}\), it lies on
some irreducible component; if no graph component can contain it, at least one
component containing it is non-graph.  Intersections of distinct components form
a finite subset of the affine curve and therefore affect the input count by at
most a constant.
\end{proof}

\begin{theorem}[Global source expansion]
\label{thm:global-source-expansion}
Let \(k\) be a number field, let \(S\) be a finite set of places containing the
Archimedean places, and let \(f,g\in k[x]\) be nonconstant.  Choose a
denominator-control set \(S'\supseteq S\) as in
\Cref{lem:denominator-number-field}.  For each
\(C\in\RR_1(f,g;k)\), choose a polynomial source parametrization
\[
        X=A_C(t),\qquad Y=B_C(t),
        \qquad A_C,B_C\in k[t].
\]
For each \(C\in\RR_2(f,g;k)\), choose the splitting data supplied by
\Cref{thm:admissible-unit-lattice}, applied to the standard affine
\(\OO_{k,S'}\)-model generated by \(X,Y\); the required two-sidedness follows
from \Cref{thm:primitive-contact-infinity}:
\[
        X=A_C(u),\qquad Y=B_C(u),\qquad
        \mathcal U_C=\bigcup_j\eta_{C,j}\Gamma_{C,j}.
\]
Then there is a finite set \(E_{f,g,k,S}\subset\OO_{k,S}\) such that
\[
\begin{aligned}
 \mathscr N^{\new}_{f,g,k,S}\ \triangle\
 \biggl(
     &\bigcup_{C\in\RR_1^{\act}(f,g;k,S)}
       \{A_C(t):t\in k,\ A_C(t)\in\OO_{k,S}\}   \\
     &\cup
       \bigcup_{C\in\RR_2(f,g;k)}
       \{A_C(u):u\in\mathcal U_C,\ A_C(u)\in\OO_{k,S}\}
 \biggr)
 \subseteq E_{f,g,k,S}.
\end{aligned}
\]
Here \(\triangle\) denotes symmetric difference.  Equivalently, after finitely
many input values are discarded, the new-lift set is a finite union of
polynomial images from active one-infinity sources and Laurent images from
admissible unit cosets on two-infinity sources.  Every active one-infinity source
appearing in this expansion has
\[
        \deg A_C=d_X(C)
        =
        \frac{\deg g}{\gcd(\deg f,\deg g)}.
\]
\end{theorem}

\begin{proof}
Let \(\mathcal S\) be the union of source images displayed in the statement.

First consider the inclusion \(\mathcal S\subseteq\mathscr N^{\new}_{f,g,k,S}\)
up to finitely many values.  If \(C\in\RR_1^{\act}(f,g;k,S)\), then
\[
        (A_C(t),B_C(t))\in C(k)
\]
and \(g(B_C(t))=f(A_C(t))\).  Since \(C\) is non-graph, for each
\(h\in\HH_{f,g}(k)\) the polynomial
\[
        B_C(t)-h(A_C(t))
\]
is not identically zero.  Thus graph intersections remove only finitely many
parameters, hence finitely many \(X\)-values.  All remaining values in the
polynomial image give new lifts.

If \(C\in\RR_2(f,g;k)\), \Cref{thm:admissible-unit-lattice} constructs
\(\mathcal U_C\) precisely so that, outside the finite normalization-exceptional
locus, \(u\in\mathcal U_C\) gives a \(k\)-rational point on the standard
\(S'\)-integral model of \(C\) with \(X=A_C(u)\in\OO_{k,S}\).  Write the Laurent
expression for \(Y\) as \(B_C(u)\).  For each \(h\in\HH_{f,g}(k)\), the Laurent
polynomial
\[
        B_C(u)-h(A_C(u))
\]
is not identically zero, because \(C\) is non-graph.  After multiplying by a
suitable power of \(u\), it becomes a nonzero ordinary polynomial, and therefore
has only finitely many zeros on \(\bar k^*\).  Hence graph intersections remove
only finitely many parameters and therefore finitely many \(X\)-values.  Thus the
two-infinity part of \(\mathcal S\) is also contained in
\(\mathscr N^{\new}_{f,g,k,S}\), up to finitely many inputs.

Conversely, let \(a\in\mathscr N^{\new}_{f,g,k,S}\), and choose
\(y\in k\) with \(g(y)=f(a)\) and with \(y\) not equal to any graph value
\(h(a)\).  By \Cref{lem:new-lifts-nongraph}, the point \((a,y)\) lies on a
non-graph component of \(C_{f,g}\), after excluding only finitely many component
intersection values.  By \Cref{lem:denominator-number-field}, it is an
\(S'\)-integral point on the standard affine model.

Apply \Cref{prop:component-count} component by component.  Geometrically
nonintegral components, positive-genus components, components with at least three
boundary points, and genus-zero components whose smooth compactification is not
\(\Pj^1_k\), contribute only finitely many input values.  A one-infinity
component contributes infinitely only when it is \(S\)-active, and then the input
has the form \(A_C(t)\) with \(t\in k\) and \(A_C(t)\in\OO_{k,S}\).  A
two-infinity component is exhausted, outside finitely many values, by the
admissible unit-coset expansion of \Cref{thm:admissible-unit-lattice}.  This
proves the finite symmetric-difference statement.

Concretely, the finite exceptional set may be chosen to contain the
\(X\)-coordinates of the affine nonnormal loci of the relevant components, the
component intersections, the intersections with graph components, the finitely
many points discarded when passing to normalizations, and the finite integral
contributions of the geometrically nonintegral, log-positive, and nonsplit
rational genus-zero cases.  For the stated symmetric-difference assertion, only
the finiteness of this set is used.

Finally, the degree formula for active one-infinity sources is
\Cref{cor:geometric-degree-spectrum}.
\end{proof}

\begin{definition}[\(X\)-equivalence of active one-infinity sources]
\label{def:X-equivalence-sources}
Let \(C,C'\in \RR_1^{\act}(f,g;k,S)\), and choose polynomial source
parametrizations
\[
        X=A_C(t),\qquad Y=B_C(t),
\]
\[
        X=A_{C'}(u),\qquad Y=B_{C'}(u).
\]
We say that \(C\) and \(C'\) are \emph{\(X\)-equivalent}, and write
\[
        C\sim_X C',
\]
if there is an affine linear polynomial
\[
        \ell(t)=at+b,
        \qquad a\in k^*,\ b\in k,
\]
such that
\[
        A_C=A_{C'}\circ\ell.
\]
This relation is independent of the chosen source coordinates.  For an
\(X\)-equivalence class \(\xi\), choose a representative \(C_\xi\), write
\[
        A_\xi=A_{C_\xi},
\]
and define
\[
        \mathscr P_\xi(B)
        =
        \{A_\xi(t):t\in k,\ A_\xi(t)\in\OO_{k,S},\ H(A_\xi(t))\le B\}.
\]
\end{definition}

\begin{theorem}[Primitive source separation and canonical power expansion]
\label{thm:primitive-source-separation}
Let \(k\) be a number field, let \(S\) be a finite set of places containing the
Archimedean places, and let \(f,g\in k[x]\) be nonconstant.  Put
\[
        n_k=[k:\Q],
        \qquad
        q_{k,S}=\operatorname{rk}\OO_{k,S}^{*},
\]
and
\[
        N=\frac{\deg g}{\gcd(\deg f,\deg g)}.
\]
Assume that \(\RR_1^{\act}(f,g;k,S)\ne\varnothing\), and let
\[
        \Xi_{f,g,k,S}=\RR_1^{\act}(f,g;k,S)/{\sim_X}
\]
be the finite set of active one-infinity \(X\)-source classes.  Choose a
denominator-control set \(S'\supseteq S\) as in
\Cref{lem:denominator-number-field}.  For each \(C\in\RR_2(f,g;k)\), let
\[
        \rho_C^{\adm}=\rho_C^{\adm}(S,S')
\]
be the admissible rank from \Cref{thm:admissible-unit-lattice}, computed for the
standard affine \(\OO_{k,S'}\)-model generated by \(X,Y\), and put
\[
        \rho_{f,g,k,S}^{\adm}
        =\max_{C\in\RR_2(f,g;k)}\rho_C^{\adm},
\]
with value \(-1\) if \(\RR_2(f,g;k)=\varnothing\).

Then the following hold.

\begin{enumerate}[label=\textup{(\alph*)}]
    \item For every \(\xi\in\Xi_{f,g,k,S}\),
    \[
        \#\mathscr P_\xi(B)
        \asymp_{\xi,k,S}
        B^{n_k/N}(\log(2B))^{q_{k,S}}.
    \]

    \item If \(\xi\ne\eta\), let
    \[
        W_{\xi,\eta}=(A_\xi(T)-A_\eta(U)=0)_{\red}.
    \]
    Using the denominator-control argument of \Cref{lem:polynomial-S-value-set},
    fix nonzero elements \(\delta_\xi,\delta_\eta\in k^*\) and a finite set of
    places \(T_{\xi,\eta}\supseteq S\) such that
    \[
        A_\xi(t)\in\OO_{k,S}
        \quad\Longrightarrow\quad
        \delta_\xi t\in\OO_{k,T_{\xi,\eta}},
    \]
    and
    \[
        A_\eta(u)\in\OO_{k,S}
        \quad\Longrightarrow\quad
        \delta_\eta u\in\OO_{k,T_{\xi,\eta}}.
    \]
    Enlarge \(T_{\xi,\eta}\), without changing notation, so that
    \(\delta_\xi,\delta_\eta\) are \(T_{\xi,\eta}\)-units and all coefficients and
    relations needed for the affine model generated below are integral over
    \(\OO_{k,T_{\xi,\eta}}\).  Put
    \[
        T_0=\delta_\xi T,
        \qquad
        U_0=\delta_\eta U.
    \]
    In particular, after writing \(T=T_0/\delta_\xi\) and
    \(U=U_0/\delta_\eta\), the coefficients in
    \[
        A_\xi(T_0/\delta_\xi),\qquad A_\eta(U_0/\delta_\eta)
    \]
    and in the relations
    \(x=A_\xi(T_0/\delta_\xi)=A_\eta(U_0/\delta_\eta)\) have cleared
    denominators over \(\OO_{k,T_{\xi,\eta}}\).
    For every such two-boundary component \(D\), primitive contact for the
    polynomial pair \((A_\xi,A_\eta)\) shows that \(T\) and \(U\), and therefore
    also \(x=A_\xi(T)=A_\eta(U)\), have poles at both boundary points; hence the
    fixed model generated by \(x,T_0,U_0\) satisfies the two-sided hypothesis of
    \Cref{thm:admissible-unit-lattice}.  Define \(\rho_{\xi,\eta}\) to be the
    maximum of the admissible ranks attached, via
    \Cref{thm:admissible-unit-lattice}, to these geometrically integral
    \(k\)-components \(D\) of \(W_{\xi,\eta}\) with
    \(\widetilde D\cong\Pj^1_k\) and with two geometric boundary points, for this
    fixed affine \(\OO_{k,T_{\xi,\eta}}\)-model.  If there are no such rational
    two-infinity components, set \(\rho_{\xi,\eta}=-1\).  Then
    \[
        \#\bigl(\mathscr P_\xi(B)\cap\mathscr P_\eta(B)\bigr)
        \ll_{\xi,\eta,k,S}
        (\log(2B))^{\max\{0,\rho_{\xi,\eta}\}}.
    \]
    In particular, distinct \(X\)-source classes have no overlap of positive
    \(B\)-power order.

    \item Put
    \[
        \rho_{f,g,k,S}^{\mathrm{ov}}
        =
        \max_{\xi\ne\eta}\rho_{\xi,\eta},
    \]
    with value \(-1\) if there is at most one \(X\)-source class.  Then
    \[
        N^{\new}_{f,g,k,S}(B)
        =
        \sum_{\xi\in\Xi_{f,g,k,S}}\#\mathscr P_\xi(B)
        +
        O_{f,g,k,S}\!\left(
            (\log(2B))^{\max\{0,\rho_{f,g,k,S}^{\adm},
                                \rho_{f,g,k,S}^{\mathrm{ov}}\}}
        \right).
    \]
\end{enumerate}
Thus the power part of the new-lifting count is canonically decomposed by active
primitive \(X\)-source classes.  Two-infinity components and overlaps among
distinct one-infinity source classes contribute only logarithmic lower-order
terms.
\end{theorem}

\begin{proof}
The relation \(\sim_X\) is well-defined.  Changing the affine coordinate on a
one-infinity source replaces \(A_C(t)\) by \(A_C(\ell(t))\), with
\(\ell\in\operatorname{Aff}_1(k)\).  Hence the condition
\(A_C=A_{C'}\circ\ell\) is independent of coordinates.  If \(C\sim_X C'\), then
\(A_C(k)=A_{C'}(k)\), because \(\ell:k\to k\) is bijective.  Therefore
\(\mathscr P_\xi(B)\) is independent of the chosen representative of \(\xi\).

By primitive contact at infinity, every active one-infinity non-graph component
has
\[
        \deg A_C=d_X(C)=N.
\]
Since the component is active, \(A_C(k)\cap\OO_{k,S}\ne\varnothing\).  The
polynomial \(S\)-integer value-set estimate, \Cref{lem:polynomial-S-value-set},
gives
\[
        \#\mathscr P_\xi(B)
        \asymp_{\xi,k,S}
        B^{n_k/N}(\log(2B))^{q_{k,S}},
\]
which proves \textup{(a)}.

We prove \textup{(b)}.  Let \(\xi\ne\eta\), choose representatives
\(A_\xi,A_\eta\in k[t]\), and consider the reduced affine overlap curve
\[
        W_{\xi,\eta}
        =
        \bigl(A_\xi(T)-A_\eta(U)=0\bigr)_{\red}
        \subset \A^2_{T,U}.
\]
Every value in \(\mathscr P_\xi(B)\cap\mathscr P_\eta(B)\) comes from a
\(k\)-rational point \((t,u)\in W_{\xi,\eta}(k)\) with
\[
        A_\xi(t)=A_\eta(u)\in\OO_{k,S},
        \qquad
        H(A_\xi(t))\le B.
\]
By the denominator-clearing choices made in the statement, after the fixed
affine change of variables
\[
        T_0=\delta_\xi T,
        \qquad
        U_0=\delta_\eta U,
\]
the rational points relevant to the overlap are therefore
\(T_{\xi,\eta}\)-integral points on the fixed affine model of
\(W_{\xi,\eta}\) generated by \(x=A_\xi(T)=A_\eta(U)\), \(T_0\), and \(U_0\).
The enlargement of \(T_{\xi,\eta}\) in the statement clears the denominators of
this model and of its defining relations, so the model satisfies the hypotheses
needed below, including the hypotheses of \Cref{thm:admissible-unit-lattice} on
its rational two-boundary components.  We apply \Cref{prop:component-count} to
the components of this fixed model with the regular function
\[
        x=A_\xi(T)=A_\eta(U).
\]

The only components that could contribute a positive power of \(B\) are
geometrically integral one-infinity components with projective normalization
\(\Pj^1_k\).  We claim that no such component exists.  Indeed, apply the
primitive-contact theorem to the polynomial pair \(A_\xi,A_\eta\).  Both
polynomials have degree \(N\).  Hence, for this auxiliary fiber product,
\[
        \gcd(\deg A_\xi,\deg A_\eta)=N,
\]
and the primitive \(T\)-projection degree of a one-infinity component is
\[
        \frac{\deg A_\eta}{\gcd(\deg A_\xi,\deg A_\eta)}=1.
\]
Thus any one-infinity \(k\)-component capable of contributing a power term has
degree one over the \(T\)-line.  By the degree-one graph lemma,
\Cref{lem:degree-one-graph-components}, it must be of the form
\[
        U=h(T),
        \qquad h\in k[T],
\]
with
\[
        A_\xi(T)=A_\eta(h(T)).
\]
Since \(\deg A_\xi=\deg A_\eta=N\), the polynomial \(h\) has degree one.  This
is exactly the condition that \(\xi=\eta\), contrary to assumption.

Therefore the auxiliary overlap curve has no \(k\)-rational one-infinity
component capable of producing a power term.  Geometrically nonintegral
components contribute finitely many \(k\)-points, log-positive components
contribute finitely many integral points by Siegel--Mahler, and genus-zero
components whose compactification is not \(\Pj^1_k\) again contribute finitely
many \(k\)-points.  Genus-positive components with two boundary points are
included in the log-positive case, so the only possible infinite contributions
come from geometrically integral \(k\)-components \(D\) with
\(\widetilde D\cong\Pj^1_k\) and with two geometric boundary points.

For such a rational two-infinity component \(D\) of \(W_{\xi,\eta}\), primitive
contact applied to the pair \((A_\xi,A_\eta)\) shows that both \(T\) and \(U\),
hence also \(x=A_\xi(T)=A_\eta(U)\), have poles at both boundary points of
\(D\).  Thus
the fixed model generated by \(x,T_0,U_0\) satisfies the two-sided hypothesis of
\Cref{thm:admissible-unit-lattice}.  By the definition of \(\rho_{\xi,\eta}\) in
the statement, the union of all rational two-infinity overlap contributions is therefore
\[
        O_{\xi,\eta,k,S}
        \bigl((\log(2B))^{\max\{0,\rho_{\xi,\eta}\}}\bigr).
\]
This proves \textup{(b)}.

Finally we prove \textup{(c)}.  By the global source expansion,
\Cref{thm:global-source-expansion}, outside finitely many input values the
new-lift set is the union of the active one-infinity source images and the
admissible two-infinity Laurent images.  After quotienting active one-infinity
sources by \(\sim_X\), the one-infinity part is
\[
        \bigcup_{\xi\in\Xi_{f,g,k,S}}\mathscr P_\xi.
\]
The two-infinity part contributes
\[
        O_{f,g,k,S}
        \bigl((\log(2B))^{\max\{0,\rho_{f,g,k,S}^{\adm}\}}\bigr)
\]
by \Cref{thm:admissible-unit-lattice}.  Since \(\Xi_{f,g,k,S}\) is finite,
\[
\left|
\#\bigcup_{\xi}\mathscr P_\xi(B)
-
\sum_{\xi}\#\mathscr P_\xi(B)
\right|
\le
\sum_{\xi\ne\eta}
\#\bigl(\mathscr P_\xi(B)\cap\mathscr P_\eta(B)\bigr).
\]
Part \textup{(b)} bounds every pairwise intersection by a logarithmic term.
Combining this with the two-infinity contribution and the finite exceptional set
gives the stated expansion.
\end{proof}

\begin{remark}[Sharpness of the logarithmic overlap error]
\label{rem:log-overlap-sharp}
The overlap term in \Cref{thm:primitive-source-separation} cannot in general be
replaced by \(O(1)\).  For example, the two degree-two maps
\[
        A_1(T)=T^2,
        \qquad
        A_2(U)=2U^2+1
\]
are not affinely equivalent over \(\Q\), but their overlap curve
\[
        T^2=2U^2+1
\]
is a Pell conic with two points at infinity.  Its integer points give
logarithmically many common values.  The separation theorem gives the sharp
form: distinct primitive \(X\)-source classes have no power-sized overlap,
but logarithmic overlap is genuinely possible.
\end{remark}

\begin{theorem}[Sparse exceptional lifting over number fields]
\label{thm:number-field-sparse}
Let \(k\) be a number field, let \(S\) be a finite set of places containing the
Archimedean places, and let \(f,g\in k[x]\) be nonconstant.  Put
\[
        q_{k,S}=\operatorname{rk}\OO_{k,S}^{*}=|S|-1.
\]
Fix a denominator-control set \(S'\supseteq S\) as in
\Cref{lem:denominator-number-field}, and for each
\(C\in\RR_2(f,g;k)\) let
\[
        \rho_C^{\adm}=\rho_C^{\adm}(S,S')
\]
be the admissible rank from \Cref{thm:admissible-unit-lattice}, computed for the
closure of \(C\) in the standard affine \(\OO_{k,S'}\)-model of \(C_{f,g}\).
Then
\[
\begin{aligned}
      N^{\new}_{f,g,k,S}(B)
      &\ll_{k,S,f,g}
      1+
      \sum_{C\in\RR_1^{\act}(f,g;k,S)}
          B^{n_k/d_X(C)}(\log(2B))^{q_{k,S}} \\
      &\qquad
      +\sum_{C\in\RR_2(f,g;k)}
          (\log(2B))^{\max\{0,\rho_C^{\adm}\}}.
\end{aligned}
\]
Moreover, if \(C\in\RR_2(f,g;k)\) has \(\rho_C^{\adm}\ge1\), then the inputs
arising from \(C\), after removing graph intersections, satisfy the lower bound
\[
        \gg_{k,S,f,g,C} (\log(2B))^{\rho_C^{\adm}}.
\]
\end{theorem}

\begin{proof}
By \Cref{thm:global-source-expansion}, the new-lift input set is, up to finitely
many values, the union of the polynomial source images from
\(\RR_1^{\act}(f,g;k,S)\) and the admissible Laurent images from
\(\RR_2(f,g;k)\).  For \(C\in\RR_1^{\act}(f,g;k,S)\), the count of values
\[
        A_C(t)\in\OO_{k,S},\qquad H(A_C(t))\le B,
\]
is
\[
        O_{C,k,S}\bigl(B^{n_k/d_X(C)}(\log(2B))^{q_{k,S}}\bigr)
\]
by \Cref{lem:polynomial-S-value-set}.  For \(C\in\RR_2(f,g;k)\),
\Cref{thm:admissible-unit-lattice} gives the displayed logarithmic upper bound
with the chosen admissible rank \(\rho_C^{\adm}\), and also gives the stated
lower bound whenever that rank is positive.  In the lower-bound case, removing
graph intersections does not change the order of growth: for each
\(h\in\HH_{f,g}(k)\), the regular function \(Y-h(X)\) is not identically zero on
the non-graph component \(C\), so its zero locus on \(C\) is finite, and deleting
these finitely many \(X\)-values preserves the positive logarithmic lower bound.
Summing over the finitely many source components and absorbing the finite
symmetric-difference set proves the theorem.
\end{proof}

\begin{corollary}[Componentwise source exponent over number fields]
\label{cor:number-field-exponent}
If \(\RR_1^{\act}(f,g;k,S)\ne\varnothing\), define
\[
        \theta_{f,g,k,S}
        =\max_{C\in\RR_1^{\act}(f,g;k,S)}\frac1{d_X(C)}.
\]
Then
\[
\begin{aligned}
        N^{\new}_{f,g,k,S}(B)
        &\ll_{k,S,f,g}
        B^{n_k\theta_{f,g,k,S}}(\log(2B))^{q_{k,S}} \\
        &\qquad +(\log(2B))^c,
\end{aligned}
\]
where \(c\ge0\) comes only from two-infinity components.  If
\(\RR_1^{\act}(f,g;k,S)=\varnothing\), then
\[
        N^{\new}_{f,g,k,S}(B)\ll_{k,S,f,g}(\log(2B))^c.
\]
\end{corollary}

\begin{proof}
This is the power part of \Cref{thm:number-field-sparse}, with the finitely many
two-infinity contributions absorbed into a single logarithmic exponent \(c\).
\end{proof}

\begin{theorem}[Primitive-contact power bound]
\label{thm:gcd-exponent-bound}
Let \(k\) be a number field, let \(S\) contain the Archimedean places, and let
\(f,g\in k[x]\) be nonconstant.  Put
\[
        m=\deg f,
        \qquad n=\deg g,
        \qquad e=\gcd(m,n),
        \qquad N=\frac ne.
\]
If \(N=1\), then there are no active one-infinity non-graph components.  If
\(N\ge2\), then every active one-infinity non-graph component \(C\) has
\[
        d_X(C)=N.
\]
Consequently, for some \(c\ge0\),
\[
        N^{\new}_{f,g,k,S}(B)
        \ll_{k,S,f,g}
        B^{[k:\Q]/N}(\log(2B))^{q_{k,S}}
        +(\log(2B))^c
        \qquad (N\ge2),
\]
with the convention that, if there is no active one-infinity component, only the
polylogarithmic term remains.  For \(N=1\) only the polylogarithmic term remains.
\end{theorem}

\begin{proof}
Let \(C\) be an active one-infinity non-graph component.  By
\Cref{cor:geometric-degree-spectrum}, every one-infinity component has
\(d_X(C)=N\).  If \(N=1\), then \(d_X(C)=1\), and
\Cref{lem:degree-one-graph-components} says that \(C\) is a graph component, a
contradiction.  Thus active non-graph one-infinity components can occur only when
\(N\ge2\), and then all of them have degree \(N\).  The displayed estimate follows
from \Cref{cor:number-field-exponent}.
\end{proof}

\begin{corollary}[One-infinity source degree spectrum]
\label{cor:one-infinity-degree-spectrum}
With notation as in \Cref{thm:gcd-exponent-bound}, after graph removal the
one-infinity source degree set is empty when \(N=1\), and is either empty or
\(\{N\}\) when \(N\ge2\).  Equivalently, \(N\) is the only possible
one-infinity source degree after graph removal, but such a source need not exist.
Hence, if there is no active one-infinity non-graph component of degree \(N\),
then there is no active one-infinity non-graph component at all, and the sparse
new-lifting count is purely polylogarithmic or bounded.
\end{corollary}

\begin{proof}
The degree statement is \Cref{cor:geometric-degree-spectrum} together with
\Cref{lem:degree-one-graph-components}.  The final counting assertion follows
from \Cref{thm:number-field-sparse}.
\end{proof}

\section{Active components and lower bounds}
\label{sec:active-lower}

The upper bound in \Cref{thm:number-field-sparse} distinguishes geometric
possibility from arithmetic activity.  We now show that activity gives the
expected lower power.

\begin{theorem}[Lower bound from an active component]
\label{thm:lower-bound}
Let \(k\) be a number field, let \(S\) contain the Archimedean places, and let
\(f,g\in k[x]\) be nonconstant.  Suppose
\(C\in\RR_1^{\act}(f,g;k,S)\) is a non-graph component with \(d=d_X(C)\).  Then
\[
        N^{\new}_{f,g,k,S}(B)
        \gg_{k,S,f,g,C} B^{n_k/d}(\log(2B))^{q_{k,S}}.
\]
In particular, the power-log order attached to \(C\) is sharp.
\end{theorem}

\begin{proof}
Choose a parametrization
\[
        X=A(t),\qquad Y=B(t),\qquad A,B\in k[t],
\]
with \(\deg A=d\).  Since \(C\) is active, \Cref{lem:polynomial-S-value-set}
gives
\[
        \#\{A(t):t\in k,
             \ A(t)\in\OO_{k,S},\ H(A(t))\le B\}
        \gg_{A,k,S}B^{n_k/d}(\log(2B))^{q_{k,S}}.
\]
For each such value \(a=A(t)\), the point
\((a,B(t))\) lies on \(C\) and gives a rational lift \(g(B(t))=f(a)\).

It remains only to remove graph intersections.  For a fixed
\(h\in\HH_{f,g}(k)\), the identity
\[
        B(t)=h(A(t))
\]
can hold for infinitely many \(t\) only if the component \(C\) is contained in
the graph \(Y=h(X)\), contrary to the hypothesis that \(C\) is non-graph.  Since
there are finitely many graph components, excluding all graph intersections
removes only finitely many parameters and therefore only finitely many
\(X\)-values.  The remaining values give new lifts counted by
\(N^{\new}_{f,g,k,S}(B)\).  This proves the lower bound.
\end{proof}

\begin{theorem}[Sparse trichotomy over number fields]
\label{thm:number-field-exact}
Let \(k\) be a number field, let \(S\) be a finite set of places containing the
Archimedean places, and let \(f,g\in k[x]\) be nonconstant.  Put
\[
        n_k=[k:\Q],
        \qquad
        q_{k,S}=\operatorname{rk}\OO_{k,S}^{*}.
\]
With \(\rho_C^{\adm}\) as in \Cref{thm:number-field-sparse}, let
\[
        \rho_{f,g,k,S}^{\adm}
        =\max_{C\in\RR_2(f,g;k)}\rho_C^{\adm},
\]
with value \(-1\) if \(\RR_2(f,g;k)=\varnothing\).  Put
\[
        N=\frac{\deg g}{\gcd(\deg f,\deg g)}.
\]
If \(\RR_1^{\act}(f,g;k,S)\ne\varnothing\), then \(N\ge2\), every active
one-infinity component has \(d_X=N\), and
\[
        N^{\new}_{f,g,k,S}(B)
        \asymp_{k,S,f,g}
        B^{n_k/N}(\log(2B))^{q_{k,S}}.
\]
Equivalently, in this case
\[
        \theta_{f,g,k,S}
        =\max_{C\in\RR_1^{\act}(f,g;k,S)}\frac1{d_X(C)}
        =\frac1N.
\]
If \(\RR_1^{\act}(f,g;k,S)=\varnothing\) and
\(\rho_{f,g,k,S}^{\adm}\ge1\), then
\[
        N^{\new}_{f,g,k,S}(B)
        \asymp_{k,S,f,g}
        (\log(2B))^{\rho_{f,g,k,S}^{\adm}}.
\]
If \(\RR_1^{\act}(f,g;k,S)=\varnothing\) and
\(\rho_{f,g,k,S}^{\adm}\le0\), then
\[
        N^{\new}_{f,g,k,S}(B)=O_{k,S,f,g}(1).
\]
Thus, after graph components are removed, active \(\A^1\)-sources give exactly
the power terms, positive-rank admissible \(\Gm\)-sources give exactly the
logarithmic-rank terms, and inactive one-infinity sources, rank-zero admissible
\(\Gm\)-sources, and all remaining components are finite for the integral lifting
problem.
\end{theorem}

\begin{proof}
If an active one-infinity component exists, \Cref{cor:geometric-degree-spectrum}
gives \(d_X(C)=N\) for every such component.  The component is non-graph, so
\Cref{lem:degree-one-graph-components} gives \(N\ge2\).  The upper bound is then
\Cref{cor:number-field-exponent}, and the lower bound follows by choosing any
active one-infinity component and applying \Cref{thm:lower-bound}.  The
two-infinity terms are polylogarithmic and are therefore lower order than any
positive power of \(B\).

Assume now that no active one-infinity component exists.  The upper bound in
\Cref{thm:number-field-sparse} leaves only the two-infinity admissible-rank terms
and the finite contributions.  If \(\rho_{f,g,k,S}^{\adm}\ge1\), choose a
component \(C\in\RR_2(f,g;k)\) of maximal admissible rank and use the lower bound
in \Cref{thm:number-field-sparse}.  If the maximal admissible rank is at most
zero, each two-infinity admissible unit set is finite by
\Cref{thm:admissible-unit-lattice}, and every other component has already been
shown to contribute \(O(1)\).  This proves all cases.
\end{proof}

For the following ordinary-integer corollaries, write
\[
\begin{aligned}
 N^{\new}_{f,g}(B)
 =\#\{n\in\Z: &\ |n|\le B,
       \ \exists y\in\Q\text{ with }g(y)=f(n), \\
       &\ y\ne h(n)\text{ for every }h\in\HH_{f,g}(\Q)\}.
\end{aligned}
\]
For \(B\ge1\), equivalently,
\(N^{\new}_{f,g}(B)=N^{\new}_{f,g,\Q,\{\infty\}}(B)\), since
\(H(n)=\max\{1,|n|\}\) for \(n\in\Z\).

\begin{corollary}[Sharp integer exponents over \(\Q\)]
\label{cor:Q-sparse}
Let \(f,g\in\Q[x]\) be nonconstant.  Put
\[
        N=\frac{\deg g}{\gcd(\deg f,\deg g)}.
\]
With \(\rho_C^{\adm}\) as in \Cref{thm:number-field-sparse} for \(k=\Q\) and
\(S=\{\infty\}\), put
\[
        \rho_{f,g}^{\adm}
        =\max_{C\in\RR_2(f,g;\Q)}\rho_C^{\adm},
\]
with value \(-1\) if there are no two-infinity components.  If
\(\RR_1^{\act}(f,g;\Q,\{\infty\})\ne\varnothing\), then \(N\ge2\) and
\[
      N^{\new}_{f,g}(B)\asymp_{f,g}B^{1/N}.
\]
If there is no active one-infinity component but \(\rho_{f,g}^{\adm}\ge1\), then
\[
      N^{\new}_{f,g}(B)\asymp_{f,g}(\log(2B))^{\rho_{f,g}^{\adm}}.
\]
If there is no active one-infinity component and \(\rho_{f,g}^{\adm}\le0\), then
\[
      N^{\new}_{f,g}(B)=O_{f,g}(1).
\]
In particular, square-root growth
\[
      N^{\new}_{f,g}(B)\asymp_{f,g}B^{1/2}
\]
occurs if and only if \(N=2\) and there is an active non-graph
\(\Q\)-rational one-infinity component.
\end{corollary}

\begin{proof}
This is \Cref{thm:number-field-exact} with \(k=\Q\) and \(S=\{\infty\}\),
combined with \Cref{cor:geometric-degree-spectrum}.  For this \(S\), there is no
logarithmic factor in the one-infinity power term.  A square-root power term is
therefore the same as \(d_X=N=2\) for an active non-graph \(\Q\)-rational
one-infinity component.
\end{proof}

\begin{corollary}[Primitive-contact integer dichotomy over \(\Q\)]
\label{cor:Q-gcd-bound}
Let \(f,g\in\Q[x]\) be nonconstant, and put
\[
        m=\deg f,
        \qquad n=\deg g,
        \qquad e=\gcd(m,n),
        \qquad N=\frac ne.
\]
If an active one-infinity non-graph component exists, then \(N\ge2\) and
\[
        N^{\new}_{f,g}(B)\asymp_{f,g}B^{1/N}.
\]
If no active one-infinity non-graph component exists, then
\[
        N^{\new}_{f,g}(B)
        \asymp_{f,g}(\log(2B))^{\rho_{f,g}^{\adm}}
\]
when the maximal admissible two-infinity rank is positive, and
\(N^{\new}_{f,g}(B)=O_{f,g}(1)\) otherwise.  In particular, when \(N=1\) there is
no one-infinity power term after graph removal.
\end{corollary}

\begin{proof}
This is \Cref{cor:Q-sparse}.  The assertion for \(N=1\) is
\Cref{cor:geometric-degree-spectrum}: a one-infinity component would have
\(d_X=1\), hence would be a graph component by
\Cref{lem:degree-one-graph-components}.
\end{proof}

\begin{corollary}[No-composition form]
\label{cor:no-composition}
If \(f,g\in\Q[x]\) are nonconstant and there is no \(h\in\Q[x]\) with
\(f=g\circ h\), then, for \(B\ge1\),
\[
      \#\{n\in\Z:|n|\le B,\ f(n)\in g(\Q)\}=N^{\new}_{f,g}(B).
\]
Consequently the primitive-contact dichotomy of \Cref{cor:Q-gcd-bound} applies
directly to this value set.  In particular it is always
\[
      O_{f,g}(B^{1/2})+O_{f,g}((\log(2B))^c)
\]
for some \(c\ge0\), and if a power term occurs then its exponent is exactly
\(1/N\), where \(N=\deg g/\gcd(\deg f,\deg g)\).
\end{corollary}

\begin{proof}
Under the no-composition hypothesis \(\HH_{f,g}(\Q)=\varnothing\).  For
\(B\ge1\), the height condition defining
\(N^{\new}_{f,g,\Q,\{\infty\}}(B)\) is the same as \(|n|\le B\), because
\(H(n)=\max\{1,|n|\}\) for integer \(n\).  Thus \(N^{\new}_{f,g}(B)\) is exactly the
displayed count.  The remaining assertions are \Cref{cor:Q-gcd-bound}.
\end{proof}

\begin{example}[Sharpness of every exponent]
\label{ex:powers}
For \(d\ge2\), take
\[
        f(X)=X,
        \qquad
        g(Y)=Y^d.
\]
There is no \(h\in\Q[X]\) with \(X=h(X)^d\).  The curve \(X=Y^d\) is rational,
has one geometric point at infinity, and is parametrized by
\[
        X=t^d,
        \qquad
        Y=t.
\]
Thus \(d_X=d\), and the integer inputs that lift are precisely the integral
\(d\)-th powers, nonnegative when \(d\) is even and with both signs when
\(d\) is odd.  The count is \(\asymp B^{1/d}\).  Over a number field, the same
example gives the sharp order
\[
        B^{[k:\Q]/d}(\log B)^{q_{k,S}}.
\]
For instance, over a real quadratic field with \(S=S_\infty\), the logarithmic
factor is \(\log B\), because \(q_{k,S}=1\).
\end{example}

\begin{example}[A square-root singular component]
For
\[
        f(X)=X^3,
        \qquad
        g(Y)=Y^2,
\]
the curve \(X^3=Y^2\) is parametrized by
\[
        X=t^2,
        \qquad
        Y=t^3.
\]
It is rational with one geometric point at infinity and \(d_X=2\), so it gives square-root
growth.  The square-root source need not be a smooth conic; singular rational
components also occur.
\end{example}

\begin{example}[A local obstruction]
For
\[
        f(X)=X,
        \qquad
        g(Y)=Y^2+\frac12,
\]
the curve is rational and has the parametrization
\[
        X=t^2+\frac12,
        \qquad
        Y=t.
\]
However it produces no integer inputs.  If \(n=y^2+1/2\) with \(n\in\Z\) and
\(y=a/b\in\Q\) in lowest terms, then
\[
        2nb^2=2a^2+b^2.
\]
If \(b\) is odd, the right side is odd and the left side is even.  If \(b\) is
even, then \(a\) is odd, and after writing \(b=2c\) the equation becomes
\[
        4nc^2=a^2+2c^2,
\]
whose two sides have opposite parity.  Thus no such rational \(y\) exists.  This
explains why the upper-bound theorem must distinguish component geometry from
arithmetic activity.

\end{example}

\begin{example}[A reducible example with graph removal]
Let
\[
        f(X)=X^4,
        \qquad
        g(Y)=Y^4.
\]
Then
\[
        f(X)-g(Y)=(X-Y)(X+Y)(X^2+Y^2).
\]
The first two factors are graph components, corresponding to \(Y=X\) and
\(Y=-X\).  The remaining factor is irreducible over \(\Q\) but not geometrically
integral; over \(\Q(i)\) it splits into the two graph lines \(Y=\pm iX\).  Hence
it contributes only finitely many rational points, in fact only \((0,0)\) over
\(\Q\).  Thus after graph removal one has \(N^{\new}_{f,g}(B)=O(1)\).  This
example illustrates why graph components and geometrically reducible components
are separated before the sparse exponent is read off.
\end{example}

\begin{example}[A Chebyshev standard-pair component]
Let \(T_2(z)=2z^2-1\) and \(T_3(z)=4z^3-3z\).  Take
\[
        f(X)=T_2(X),
        \qquad
        g(Y)=T_3(Y).
\]
The identity \(T_2(T_3(t))=T_3(T_2(t))\) gives a rational component of
\(T_2(X)-T_3(Y)=0\) parametrized by
\[
        X=T_3(t),\qquad Y=T_2(t).
\]
It is birational because \(X=t(2Y-1)\), so \(t=X/(2Y-1)\) in the function field.
The component is not a graph component: \(T_2(t)\notin\Q[T_3(t)]\), since a
nonconstant polynomial in \(T_3(t)\) has degree divisible by \(3\).  Thus it is
an active one-infinity component with \(d_X=3\), and it gives
\[
        N^{\new}_{f,g}(B)\gg B^{1/3}.
\]
This illustrates how the theorem extracts a sparse integer exponent from a
Bilu--Tichy-type component without invoking a full reclassification.
\end{example}

\begin{example}[A two-infinity Pell component]
Let \(D>1\) be a squarefree integer, and take
\[
        f(X)=X^2,
        \qquad
        g(Y)=DY^2+1.
\]
The curve \(f(X)=g(Y)\) is
\[
        X^2-DY^2=1.
\]
It is rational over \(\Q\), has two geometric points at infinity, and is not a
graph component.  Its integer points are Pell solutions.  For every nonsquare
\(D>1\), the norm-one unit group in \(\Q(\sqrt D)\) has rank one; even when a
chosen fundamental unit has norm \(-1\), its square has norm \(1\).  Hence the
integer solutions give \(\asymp_D \log B\) inputs \(|X|\le B\), in agreement with
the two-infinity polylogarithmic term in \Cref{cor:Q-sparse}.  This example also
shows why the one-infinity power terms and the two-infinity unit terms must be
separated.
\end{example}

\section{Higher multiplicity via configuration covers}
\label{sec:configuration-multiplicity}

The main theorem counts inputs admitting at least one new rational lift.  The
same logarithmic trichotomy controls inputs admitting several distinct new
lifts.  The correct geometric object is the ordered configuration cover of the
non-graph part of \(C_{f,g}\).

Let
\[
        C^{\new}_{f,g}
\]
denote the reduced union of the non-graph irreducible components of
\(C_{f,g}\).  It is finite over the \(X\)-line.  For \(r\ge1\), form the
fiber product
\[
        (C^{\new}_{f,g})^r_{\A^1_X}
        =C^{\new}_{f,g}\times_{\A^1_X}\cdots\times_{\A^1_X}C^{\new}_{f,g}.
\]
Write \(Y_1,\ldots,Y_r\) for the pulled-back \(Y\)-coordinates.  Define the
ordered \(r\)-configuration cover
\[
        C^{[r]}_{f,g}
\]
to be the reduced union of those irreducible components of
\((C^{\new}_{f,g})^r_{\A^1_X}\) on which
\[
        Y_i\ne Y_j\qquad (i\ne j)
\]
holds at the generic point.  Equivalently, \(C^{[r]}_{f,g}\) is the finite
closure of the locus of ordered \(r\)-tuples of generically distinct non-graph
lifts.  Components supported over finitely many \(X\)-values, if any arise in
the scheme-theoretic fiber product, are not included in the sets
\(\RR^{[r]}_i\) below; their contribution to input counts is absorbed in the
constant terms.

For \(r\ge1\), define the level-\(r\) new-lift input set
\[
\begin{aligned}
 \mathscr N^{\new}_{\ge r,f,g,k,S}
 =\Bigl\{a\in\OO_{k,S}:{}&
   \#\{y\in k:g(y)=f(a),\\
   &\qquad y\ne h(a)\text{ for all }h\in\HH_{f,g}(k)\}\ge r
        \Bigr\}.
\end{aligned}
\]
Its height truncation is
\[
        N^{\new}_{\ge r,f,g,k,S}(B)
        =
        \#\{a\in\mathscr N^{\new}_{\ge r,f,g,k,S}:H(a)\le B\}.
\]
Thus \(N^{\new}_{\ge1,f,g,k,S}(B)=N^{\new}_{f,g,k,S}(B)\).
For an irreducible component \(Z\) of \(C^{[r]}_{f,g}\) dominating the
\(X\)-line, let
\[
        d_X(Z)=[k(Z):k(X)].
\]
Let \(\RR^{[r]}_1(f,g;k)\) be the set of irreducible components \(Z\) of
\(C^{[r]}_{f,g}\) which dominate the \(X\)-line and such that \(Z\) is
geometrically integral,
\[
        \widetilde Z\cong\Pj^1_k,
        \qquad
        \#D_Z=1.
\]
For such a component choose a parametrization
\[
        X=A_Z(t),
        \qquad
        A_Z\in k[t],
\]
and call \(Z\) \(S\)-active if
\[
        A_Z(k)\cap\OO_{k,S}\ne\varnothing.
\]
Let \(\RR^{[r],\act}_1(f,g;k,S)\) denote the active subset.  Let
\(\RR^{[r]}_2(f,g;k)\) be the set of geometrically integral irreducible
components \(Z\) of \(C^{[r]}_{f,g}\) which dominate the \(X\)-line and have
\[
        \widetilde Z\cong\Pj^1_k,
        \qquad
        \#D_Z=2.
\]

\begin{lemma}[No degree-one ground-field components in the configuration cover]
\label{lem:configuration-no-degree-one}
Every irreducible \(k\)-component of \(C^{[r]}_{f,g}\) which dominates the
\(X\)-line has \(X\)-projection degree at least \(2\).
\end{lemma}

\begin{proof}
Suppose that \(Z\) is an irreducible \(k\)-component of \(C^{[r]}_{f,g}\) of
degree one over
\(\A^1_X\).  Then \(k(Z)=k(X)\).  For each coordinate projection to the
non-graph cover, the corresponding function \(Y_i\) is therefore an element of
\(k(X)\), say \(Y_i=\eta_i(X)\), and it satisfies
\[
        g(\eta_i(X))=f(X).
\]
The same pole argument used for \Cref{lem:degree-one-graph-components} shows
that \(\eta_i\) has no finite pole; hence \(\eta_i\in k[X]\).  Thus
\(Y=\eta_i(X)\) is a graph component of \(C_{f,g}\).  But the \(i\)-th
coordinate of \(C^{[r]}_{f,g}\) has generic image in \(C^{\new}_{f,g}\), the
union of non-graph components.  A non-graph component cannot contain the graph
generically, since distinct irreducible components of a plane curve meet only in
finitely many affine points.  This contradiction proves \(d_X(Z)\ge2\).
\end{proof}

\begin{remark}
The lemma is a ground-field assertion.  A geometrically reducible \(k\)-component
may split after base change into geometric degree-one branches.  Such components
are not members of the geometrically integral sets \(\RR^{[r]}_1\) and
\(\RR^{[r]}_2\), and their \(k\)-rational contribution is finite by
\Cref{lem:ngeo-finite}.
\end{remark}

\begin{lemma}[Normalized local configuration fiber products]
\label{lem:local-configuration-fiber-product}
Let \(K\) be an algebraically closed field of characteristic zero.  Let
\(N\ge1\) and \(r\ge1\).  For \(1\le i\le r\), suppose
\[
        K[[w]]\longrightarrow K[[s_i]],
        \qquad
        w\longmapsto s_i^N\varepsilon_i(s_i),
        \qquad \varepsilon_i(0)\ne0.
\]
Then, after replacing each \(s_i\) by a unit multiple, these maps have the form
\(w=s_i^N\).  The reduced normalization of the completed tensor product
\[
        K[[s_1]]\widehat\otimes_{K[[w]]}\cdots
        \widehat\otimes_{K[[w]]}K[[s_r]]
\]
is a finite disjoint union of formal discs.  More explicitly, after the above
change of parameters, it is the product over all
\((\zeta_2,\ldots,\zeta_r)\) with \(\zeta_i^N=1\) of copies of \(K[[s]]\), on
which
\[
        s_1=s,
        \qquad
        s_i=\zeta_i s\quad(2\le i\le r),
        \qquad
        w=s^N.
\]
In particular, on every normalized branch the order of \(w\) is \(N\), and the
order of each \(s_i\) is one.
\end{lemma}

\begin{proof}
Because \(K\) is algebraically closed of characteristic zero, every unit
\(\varepsilon_i(s_i)\in K[[s_i]]^*\) has an \(N\)-th root.  Replacing \(s_i\) by
\(s_i\varepsilon_i(s_i)^{1/N}\) gives \(w=s_i^N\).  The completed tensor product
is then
\[
        R=K[[s_1,\ldots,s_r]]/(s_1^N-s_2^N,\ldots,s_1^N-s_r^N).
\]
Its reduced irreducible components are cut out by
\[
        s_i=\zeta_i s_1\quad(2\le i\le r),
        \qquad \zeta_i^N=1.
\]
Indeed the radical of the defining ideal is the intersection of these prime
ideals, and each quotient is \(K[[s_1]]\).  The normalization of \(R_{\red}\) is
therefore the product of these regular one-dimensional complete local rings,
with the displayed parametrization.  The assertions about orders follow
immediately.
\end{proof}

\begin{lemma}[Normalizing the local factors does not change configuration branches]
\label{lem:normalizing-local-factors}
Let \(K\) be an algebraically closed field of characteristic zero and put
\(A=K[[w]]\).  For \(1\le i\le r\), let \(R_i\) be a complete one-dimensional
local \(K\)-domain, finite and torsion-free over \(A\), whose normalization is
\(\bar R_i=K[[s_i]]\).  Put
\[
        R=R_1\widehat\otimes_A\cdots\widehat\otimes_A R_r,
        \qquad
        \bar R=\bar R_1\widehat\otimes_A\cdots\widehat\otimes_A\bar R_r.
\]
Then \(\bar R\) is finite integral over \(R\), the map \(R\to\bar R\) becomes an
isomorphism after inverting \(w\), and \(R_{\red}\) and \(\bar R_{\red}\) have
the same normalization.  Consequently every normalized branch of
\(R_{\red}\) is one of the normalized branches obtained from \(\bar R_{\red}\).
\end{lemma}

\begin{proof}
The rings \(R_i\) are complete Noetherian local \(K\)-domains of dimension one;
they are excellent, hence Nagata, and their normalizations in their fraction
fields are finite.  We use only these standard consequences of excellence and
finite normalization; see \cite[Tags 07QS, 032E and 035S]{StacksProject}.  Hence
\(\bar R_i\) is finite over \(R_i\).  Since all rings involved are finite
\(A\)-modules and \(A\) is complete Noetherian, the completed tensor products
above agree with the corresponding finite ordinary tensor products over \(A\).
Integrality is preserved under base change and composition, so the iterated map
\(R\to\bar R\) is finite integral.

Since \(R_i\) is a one-dimensional domain finite over the DVR \(A\), the
localization \(R_i[1/w]\) is a zero-dimensional domain finite over \(K((w))\),
hence a field; it is the fraction field of \(R_i\).  The normalization is
birational, so
\[
        R_i[1/w]=\bar R_i[1/w].
\]
Tensoring these identities over \(K((w))\) gives
\[
        R[1/w]=\bar R[1/w].
\]
Since the extensions are finite over \(K((w))\) and the characteristic is zero,
these generic algebras are finite separable after passing to their component
fields.  Thus the equality above is an equality of the generic branch algebras;
normalizing the local factors has not changed anything over the punctured
formal disc.  Because the \(R_i\) are torsion-free over the DVR \(A\), they are
flat over \(A\); their completed tensor product is finite and flat over \(A\), so no
minimal prime of \(R_{\red}\) contains \(w\).  The same is true for \(\bar R\).
After inverting \(w\), the reduced rings \(R_{\red}[1/w]\) and
\(\bar R_{\red}[1/w]\) are finite reduced \(K((w))\)-algebras, hence products of
fields.  Since every minimal prime misses \(w\), these products are exactly the
products of the fraction fields of the irreducible components, i.e.\ the total
quotient rings of \(R_{\red}\) and \(\bar R_{\red}\).  The displayed equality
therefore identifies the two total quotient rings.  After this identification,
we view the image of \(\bar R_{\red}\) inside the common total quotient ring.  It
is integral over \(R_{\red}\).
An element of the common total quotient ring is integral over \(R_{\red}\) if
and only if it is integral over this image of \(\bar R_{\red}\): one implication
uses that \(R_{\red}\subseteq\bar R_{\red}\), and the other uses transitivity of
integrality.  Thus the two reduced rings have the same integral closure in their
common total quotient ring, which is exactly the claimed equality of
normalizations.
\end{proof}

\begin{proposition}[Analytic branch passage for ordered configuration products]
\label{prop:configuration-branch-passage}
Let \(K\) be an algebraically closed field of characteristic zero and put
\(A=K[[w]]\).  For \(1\le i\le r\), let \(R_i\) be a complete
one-dimensional local \(K\)-domain, finite and torsion-free over \(A\), whose
normalization is \(K[[s_i]]\), and suppose that the structural map sends
\[
        w\longmapsto s_i^N\varepsilon_i(s_i),
        \qquad \varepsilon_i(0)\ne0.
\]
Put
\[
        R=R_1\widehat\otimes_A\cdots\widehat\otimes_A R_r.
\]
Then every normalized analytic branch of \(R_{\red}\) is, after replacing the
parameters \(s_i\) by unit multiples, one of the formal discs
\[
        w=s^N,
        \qquad
        s_1=s,
        \qquad
        s_i=\zeta_i s\quad(2\le i\le r),
        \qquad \zeta_i^N=1.
\]
In particular, on every branch that remains after deleting any prescribed union
of minimal primes, the order of \(w\) is \(N\).
\end{proposition}

\begin{proof}
By \Cref{lem:normalizing-local-factors}, \(R_{\red}\) has the same
normalization as the reduced completed tensor product formed after replacing
each factor \(R_i\) by its normalization \(K[[s_i]]\).  By
\Cref{lem:local-configuration-fiber-product}, that normalized tensor product is
the finite product of the displayed formal discs.  It remains only to justify the
last assertion about deleting components.  If \(R'\) is a reduced Noetherian ring
and \(F\in R'\), then the closure of \(D(F)\subset\Spec R'\) is the union of the
irreducible components whose minimal primes do not contain \(F\).  Thus imposing
a generic open condition retains exactly the minimal primes on which a regular
representative of that condition is not identically zero.  For a diagonal
exclusion in the boundary local calculation, one may use
\(1/Y_i-1/Y_j\), equivalently clear denominators in the rational function
\(Y_i-Y_j\).  In the boundary applications each \(Y_i\) has a pole on every
selected branch, so \(1/Y_i\) is regular in the completed boundary local rings;
the denominators used in passing between these two expressions are nonzero in the
component function fields.  Generic vanishing is therefore unchanged.  The local
rings and their normalizations at the generic points of the retained components
are unchanged, so the orders on the remaining branches are unchanged.
\end{proof}

\begin{lemma}[Completed local model for configuration branches]
\label{lem:configuration-completed-local-model}
Let \(K\) be an algebraically closed field of characteristic zero, let \(r\ge1\),
and work on the compactified fiber product
\[
        \Pj^1_X\times_{\Pj^1_T}\Pj^1_Y
\]
near \((X,Y)=(\infty,\infty)\).  Put \(w=1/X\).  For \(1\le i\le r\), let
\(C_i\) be a geometric irreducible component of the non-graph fiber product, and
write \(Y_i\) for the \(Y\)-coordinate in the \(i\)-th factor of the ordered
product.  Choose one analytic boundary branch \(Q_i\) of the closure of \(C_i\)
above
\((\infty,\infty)\).  Let \(R_i\) be the quotient of the completed local ring of
that closure at \((\infty,\infty)\) by the minimal prime corresponding to
\(Q_i\).  Then \(R_i\) is a complete one-dimensional local \(K\)-domain, finite
and torsion-free over \(K[[w]]\), and its normalization is \(K[[s_i]]\) for a
parameter \(s_i\) on the normalized branch.  Moreover, with
\[
        M=\frac{\deg f}{\gcd(\deg f,\deg g)},
        \qquad
        N=\frac{\deg g}{\gcd(\deg f,\deg g)},
\]
the structural maps have the form
\[
        w\longmapsto s_i^N\varepsilon_i(s_i),
        \qquad
        1/Y_i\longmapsto s_i^M\delta_i(s_i),
        \qquad
        \varepsilon_i(0)\delta_i(0)\ne0.
\]
For the ambient ordered fiber product before diagonal removal, the completed
local ring along the selected tuple of analytic branches is
\[
        R_1\widehat\otimes_{K[[w]]}\cdots
        \widehat\otimes_{K[[w]]}R_r.
\]
The analytic branches of the ordered configuration cover lying over this tuple
are obtained from the minimal primes of the reduction of this completed tensor
product which survive the diagonal exclusions \(Y_i\ne Y_j\), equivalently near
this boundary tuple the conditions \(1/Y_i\ne1/Y_j\).
\end{lemma}

\begin{proof}
The closure of each \(C_i\) is finite over the \(X\)-line and dominates it.  After
completion at the chosen boundary point and quotienting by the minimal prime of
the chosen analytic branch, the resulting ring is therefore finite over the DVR
\(K[[w]]\).  It is a complete one-dimensional local domain; it is torsion-free
over \(K[[w]]\) because it is a domain and \(w\) is not zero on a component that
dominates the \(X\)-line.  The normalization of a complete analytic branch of a
curve over the algebraically closed field \(K\) is a formal disc \(K[[s_i]]\).
The displayed orders of \(w=1/X\) and \(1/Y_i\) are exactly the primitive-contact
calculation of \Cref{thm:primitive-contact-infinity} on the original fiber
product branch.

For finite algebras over the complete local base \(K[[w]]\), completion of the
local tensor product is the completed tensor product.  Hence the completed local
ring of the ambient ordered fiber product along the selected tuple is the
completed tensor product displayed above.  The rings involved are excellent
complete curve rings; their minimal primes after completion are precisely the
analytic branches over the selected algebraic branch.  Finally, near this boundary tuple the diagonal condition may be represented by
\[
        F=\prod_{i<j}(1/Y_i-1/Y_j),
\]
which is regular in the completed boundary local rings.  Equivalently, it is
obtained from \(\prod_{i<j}(Y_i-Y_j)\) by multiplying by a nonzero rational
denominator in the total quotient ring, so the minimal primes on which it
vanishes generically are the same diagonal minimal primes.  The closure of the
open locus \(F\ne0\) in a reduced Noetherian local ring is the union of the
irreducible components whose minimal primes do not contain \(F\).  Thus diagonal
removal discards exactly the branches on which some pair is generically equal
and creates no new branches.
\end{proof}

\begin{theorem}[Primitive contact for ordered configurations]
\label{thm:configuration-primitive-contact}
Let \(r\ge1\), and let \(Z\) be a geometrically integral component of the
ordered configuration cover \(C^{[r]}_{f,g}\) which dominates the \(X\)-line.
Put
\[
        m=\deg f,
        \qquad n=\deg g,
        \qquad e=\gcd(m,n),
        \qquad M=\frac m e,
        \qquad N=\frac n e.
\]
Then, at every geometric boundary point \(P\in D_Z\),
\[
        -\ord_P X=N,
        \qquad
        -\ord_P Y_i=M\quad(1\le i\le r).
\]
Consequently
\[
        d_X(Z)=N\#D_Z.
\]
In particular, every geometrically integral one-infinity component of
\(C^{[r]}_{f,g}\) dominating the \(X\)-line has \(d_X(Z)=N\).
\end{theorem}

\begin{proof}
We work over \(\bar k\).  Let \(P\in D_Z\).  The affine coordinate ring of the
corresponding configuration component is a quotient of
\(\bar k[X,Y_1,\ldots,Y_r]\).  If all affine coordinate functions
\(X,Y_1,\ldots,Y_r\) were regular at \(P\), then every element of this coordinate
ring would map into the valuation ring \(\OO_{\widetilde Z,P}\).  The affine
normalization is the integral closure of that coordinate ring in \(\bar k(Z)\).  Since valuation rings are integrally
closed, every element of this integral closure would also lie in
\(\OO_{\widetilde Z,P}\).  Thus \(\Spec\OO_{\widetilde Z,P}\) would map to the
affine normalization with generic point equal to the generic point of \(Z\), so
\(P\) would belong to the affine normalization rather than to the boundary.
Hence at least one of these functions has a pole.  From the equations
\[
        f(X)=g(Y_i)\qquad(1\le i\le r)
\]
it follows that \(X\) has a pole if and only if each \(Y_i\) has a pole.  Thus
every boundary point of \(Z\) lies over \(X=\infty\) and over \(Y_i=\infty\) for
all \(i\).

Set \(w=1/X\).  The \(i\)-th projection of the boundary branch of \(Z\)
selects a normalized boundary branch of the original fiber product above
\((X,Y)=(\infty,\infty)\).  By \Cref{thm:primitive-contact-infinity}, on that
branch there is a complete local parameter \(s_i\) such that
\[
        w=s_i^N\varepsilon_i(s_i),
        \qquad
        1/Y_i=s_i^M\delta_i(s_i),
\]
with \(\varepsilon_i(0)\delta_i(0)\ne0\).  By
\Cref{lem:configuration-completed-local-model}, the completed local branch rings
\(R_i\) of these original fiber-product branches satisfy the hypotheses of
\Cref{prop:configuration-branch-passage}, the completed local ring of the ambient
ordered fiber product along the chosen tuple is
\[
        R_1\widehat\otimes_{\bar k[[w]]}\cdots
        \widehat\otimes_{\bar k[[w]]}R_r,
\]
and the completed local branch of \(Z\) at \(P\) is obtained from one of the
minimal primes of the reduced completed tensor product which survives the
diagonal exclusions \(Y_i\ne Y_j\).  Applying
\Cref{prop:configuration-branch-passage} to this tensor product, after auxiliary
unit changes of the branch parameters, every such normalized branch is a formal
disc with
\[
        w=s^N,
        \qquad
        s_i=\zeta_i s\quad(1\le i\le r)
\]
for suitable \(N\)-th roots of unity \(\zeta_i\) with \(\zeta_1=1\).  Therefore
\(w\) has order \(N\) at \(P\), and each \(s_i\) has order one.  Since
\(1/Y_i=s_i^M\delta_i(s_i)\) with \(\delta_i\) a unit, we get
\[
        -\ord_P X=N,
        \qquad
        -\ord_P Y_i=M\quad(1\le i\le r).
\]
Summing the pole orders of \(X\) over all boundary
points gives \(d_X(Z)=N\#D_Z\).
\end{proof}

\begin{corollary}[Primitive contact spectrum for higher multiplicity]
\label{cor:multiplicity-degree-spectrum}
Let \(r\ge1\), and let \(Z\in\RR^{[r]}_1(f,g;k)\).  Equivalently, \(Z\) is a
geometrically integral \(k\)-component of the ordered configuration cover with
one geometric boundary point and dominating the \(X\)-line.  Put
\[
        N=\frac{\deg g}{\gcd(\deg f,\deg g)}.
\]
Then
\[
        d_X(Z)=N.
\]
Since \(Z\) is a geometrically integral \(k\)-component of the non-graph
configuration cover, such a component can exist only when \(N\ge2\).  This does
not exclude geometric degree-one branches inside geometrically reducible
\(k\)-components; those are outside \(\RR^{[r]}_1\) and have only finitely many
\(k\)-points.
\end{corollary}

\begin{proof}
The equality \(d_X(Z)=N\) is the one-infinity case of
\Cref{thm:configuration-primitive-contact}.  If \(N=1\), then \(d_X(Z)=1\),
contradicting \Cref{lem:configuration-no-degree-one}.
\end{proof}

\begin{theorem}[Global source expansion for ordered configurations]
\label{thm:configuration-global-source-expansion}
Let \(k\) be a number field, let \(S\) contain the Archimedean places, and let
\(f,g\in k[x]\) be nonconstant.  Fix \(r\ge1\), and choose a
denominator-control set \(S'\supseteq S\) as in
\Cref{lem:denominator-number-field}.  For each
\(Z\in\RR^{[r]}_1(f,g;k)\), choose a polynomial parametrization
\[
        X=A_Z(t),\qquad Y_i=B_{Z,i}(t)\quad(1\le i\le r).
\]
For each \(Z\in\RR^{[r]}_2(f,g;k)\), choose the splitting data supplied by
\Cref{thm:admissible-unit-lattice}, applied to the standard affine
\(\OO_{k,S'}\)-model generated by \(X,Y_1,\ldots,Y_r\); the required
two-sidedness follows from \Cref{thm:configuration-primitive-contact}:
\[
        X=A_Z(u),\qquad Y_i=B_{Z,i}(u)\quad(1\le i\le r),\qquad
        \mathcal U_Z=\bigcup_j\eta_{Z,j}\Gamma_{Z,j}.
\]
Then there is a finite set
\(E^{[r]}_{f,g,k,S}\subset\OO_{k,S}\) such that
\[
\begin{aligned}
 \mathscr N^{\new}_{\ge r,f,g,k,S}\ \triangle\
 \biggl(
     &\bigcup_{Z\in\RR^{[r],\act}_1(f,g;k,S)}
       \{A_Z(t):t\in k,\ A_Z(t)\in\OO_{k,S}\}   \\
     &\cup
       \bigcup_{Z\in\RR^{[r]}_2(f,g;k)}
       \{A_Z(u):u\in\mathcal U_Z,\ A_Z(u)\in\OO_{k,S}\}
 \biggr)
 \subseteq E^{[r]}_{f,g,k,S}.
\end{aligned}
\]
Thus the level-\(r\) input set is, up to finitely many values, a finite union of
polynomial configuration-source images and admissible unit-coset Laurent images.
Every active one-infinity configuration source appearing in this expansion has
\[
        \deg A_Z=d_X(Z)
        =
        \frac{\deg g}{\gcd(\deg f,\deg g)}.
\]
\end{theorem}

\begin{proof}
The proof is the same source-expansion argument as for
\Cref{thm:global-source-expansion}, with the non-graph curve replaced by the
ordered configuration cover.  The two inclusions are as follows.

A point on a one-infinity configuration source has the form
\[
        (A_Z(t),B_{Z,1}(t),\ldots,B_{Z,r}(t)).
\]
The functions \(B_{Z,i}\) are generically distinct by definition of
\(C^{[r]}_{f,g}\), and no \(B_{Z,i}\) is generically equal to a graph value
\(h(A_Z(t))\), because the configuration cover is built from the non-graph
cover.  Hence the functions \(B_{Z,i}-B_{Z,j}\) and
\(B_{Z,i}-h(A_Z(t))\) are nonzero on the source for all \(i\ne j\) and all graph
components \(h\).  Their zero loci are finite, so removing the corresponding
finitely many \(X\)-values leaves inputs with at least \(r\) distinct new
rational lifts.
The same argument applies to a two-infinity configuration source using the
admissible parameter set \(\mathcal U_Z\) from
\Cref{thm:admissible-unit-lattice}: on the split normalization \(Z_L^\nu\cong\Gm\),
each nonzero function \(Y_i-Y_j\) or \(Y_i-h(X)\) becomes a nonzero Laurent
polynomial in the unit parameter.  After multiplying by a suitable power of that
parameter, it is a nonzero ordinary polynomial, so it has only finitely many
zeros in \(\bar k^*\).  That theorem supplies exactly the \(S'\)-integral
\(k\)-rational configuration points outside a finite exceptional locus.

Conversely, if \(a\in\mathscr N^{\new}_{\ge r,f,g,k,S}\), choose \(r\) distinct
new rational lifts and order them.  The resulting point lies on the ordered
fiber product.  If it lies only on a component supported over finitely many
\(X\)-values, then \(a\) belongs to a finite exceptional set; otherwise it lies on
a component included in \(C^{[r]}_{f,g}\).  By
\Cref{lem:denominator-number-field}, all \(Y_i\) are \(S'\)-integral.  Applying
\Cref{prop:component-count} to the components of the configuration cover and
using \Cref{thm:admissible-unit-lattice} for the two-infinity components shows
that every infinite contribution is accounted for by either an active
one-infinity configuration source or an admissible two-infinity unit-coset
source.  The remaining components contribute only finitely many input values.

The degree formula for active one-infinity configuration sources is
\Cref{thm:configuration-primitive-contact}.
\end{proof}

\begin{theorem}[Higher-multiplicity sparse lifting]
\label{thm:higher-multiplicity-sparse}
Let \(k\) be a number field, let \(S\) contain the Archimedean places, and let
\(f,g\in k[x]\) be nonconstant.  Fix \(r\ge1\), and choose a denominator-control
set \(S'\supseteq S\) as in \Cref{lem:denominator-number-field}.  For each
\(Z\in\RR^{[r]}_2(f,g;k)\), let \(\rho_Z^{\adm}\) be the admissible rank given by
\Cref{thm:admissible-unit-lattice}, applied to the \(X\)-coordinate on \(Z\) and
to the induced \(\OO_{k,S'}\)-model of the configuration cover.  Put
\[
        \rho^{[r],\adm}_{f,g,k,S}
        =\max_{Z\in\RR^{[r]}_2(f,g;k)}\rho_Z^{\adm},
\]
with value \(-1\) if \(\RR^{[r]}_2(f,g;k)=\varnothing\).  Put
\[
        N=\frac{\deg g}{\gcd(\deg f,\deg g)}.
\]
Then
\[
\begin{aligned}
      N^{\new}_{\ge r,f,g,k,S}(B)
      &\ll_{k,S,f,g,r}
      1+
      \sum_{Z\in\RR^{[r],\act}_1(f,g;k,S)}
          B^{n_k/d_X(Z)}(\log(2B))^{q_{k,S}} \\
      &\qquad
      +\sum_{Z\in\RR^{[r]}_2(f,g;k)}
          (\log(2B))^{\max\{0,\rho_Z^{\adm}\}}.
\end{aligned}
\]
If \(\RR^{[r],\act}_1(f,g;k,S)\ne\varnothing\), then \(N\ge2\), every active
one-infinity configuration component has \(d_X=N\), and
\[
        N^{\new}_{\ge r,f,g,k,S}(B)
        \asymp_{k,S,f,g,r}
        B^{n_k/N}(\log(2B))^{q_{k,S}}.
\]
Equivalently, the level-\(r\) active source exponent, when defined, is
\[
        \theta^{[r]}_{f,g,k,S}=\frac1N,
\]
and the corresponding number-field height-count power is \(B^{[k:\Q]/N}\).
If \(\RR^{[r],\act}_1(f,g;k,S)=\varnothing\) and
\(\rho^{[r],\adm}_{f,g,k,S}\ge1\), then
\[
        N^{\new}_{\ge r,f,g,k,S}(B)
        \asymp_{k,S,f,g,r}
        (\log(2B))^{\rho^{[r],\adm}_{f,g,k,S}}.
\]
If \(\RR^{[r],\act}_1(f,g;k,S)=\varnothing\) and
\(\rho^{[r],\adm}_{f,g,k,S}\le0\), then
\[
        N^{\new}_{\ge r,f,g,k,S}(B)=O_{k,S,f,g,r}(1).
\]
\end{theorem}

\begin{proof}
The upper bound follows from the level-\(r\) source expansion in
\Cref{thm:configuration-global-source-expansion}.  Polynomial configuration
source images are counted by \Cref{lem:polynomial-S-value-set}; two-infinity
configuration source images are counted by \Cref{thm:admissible-unit-lattice}.
This gives the displayed upper bound.

If an active one-infinity component \(Z\) of the configuration cover exists,
the one-infinity case of \Cref{thm:configuration-primitive-contact} gives
\(d_X(Z)=N\), and \Cref{lem:configuration-no-degree-one} gives \(N\ge2\).  Choose
a parametrization
\[
        X=A_Z(t),
        \qquad
        Y_i=B_i(t)\quad(1\le i\le r),
        \qquad
        A_Z,B_i\in k[t],
\]
with \(\deg A_Z=N\).  The polynomial value-set lemma gives
\[
        \gg B^{n_k/N}(\log(2B))^{q_{k,S}}
\]
distinct admissible \(X\)-values.  The functions \(B_i-B_j\) are not identically
zero for \(i\ne j\), and for every graph component \(Y=h(X)\) the functions
\(B_i-h(A_Z)\) are not identically zero, because the ordered configuration cover
is built from the non-graph cover.  Their zero loci on the affine source are
finite.  Removing the corresponding finitely many \(X\)-values leaves at least
\(r\) distinct new rational lifts for each remaining input.  Since all active
one-infinity configuration components have the same \(X\)-degree \(N\), this
gives the asserted power asymptotic.

If no active one-infinity component exists but
\(\rho^{[r],\adm}_{f,g,k,S}\ge1\), choose a two-infinity configuration component
\(Z\) of maximal admissible rank.  \Cref{thm:admissible-unit-lattice} gives a
lower bound of order \((\log B)^{\rho^{[r],\adm}_{f,g,k,S}}\) for the distinct
\(X\)-values arising from a maximal-rank admissible coset on \(Z\).  We must
remove only finitely many exceptional parameters before these values are counted
by the level-\(r\) new-lift set.  By the definition of the ordered configuration
cover, each function
\[
        Y_i-Y_j\qquad(i\ne j)
\]
is not identically zero on \(Z\).  Since the configuration cover is built from
the non-graph part of \(C_{f,g}\), for every graph component \(Y=h(X)\) and every
\(i\), the function
\[
        Y_i-h(X)
\]
is also not identically zero on \(Z\).  On the split normalization
\(Z_L^\nu\cong\Gm\), each of these functions is a nonzero Laurent polynomial in
the unit parameter \(u\).  Multiplying by a suitable power of \(u\) gives a
nonzero ordinary polynomial, hence each zero locus in \(\bar k^*\) is finite.
After deleting the corresponding finitely many parameters, and the finite
normalization-exceptional locus, the remaining admissible points have at least
\(r\) distinct new rational lifts.  Deleting finitely many \(X\)-values does not change the positive
logarithmic lower bound, and the upper bound was already proved.  If the maximal
admissible rank is at most zero, every two-infinity admissible set is finite and
all other components have finite contribution, so
\(N^{\new}_{\ge r,f,g,k,S}(B)=O(1)\).
\end{proof}

\begin{corollary}[Multiplicity exponents over \(\Q\)]
\label{cor:configuration-Q}
For \(f,g\in\Q[x]\) and \(r\ge1\), put
\[
        N=\frac{\deg g}{\gcd(\deg f,\deg g)}.
\]
Define
\[
        \rho^{[r],\adm}_{f,g}
        =\max_{Z\in\RR^{[r]}_2(f,g;\Q)}\rho_Z^{\adm},
\]
with value \(-1\) if there are no two-infinity configuration components.  If the
active one-infinity configuration set is nonempty, then \(N\ge2\) and
\[
        N^{\new}_{\ge r,f,g}(B)
        \asymp_{f,g,r} B^{1/N}.
\]
If it is empty and \(\rho^{[r],\adm}_{f,g}\ge1\), then
\[
        N^{\new}_{\ge r,f,g}(B)
        \asymp_{f,g,r}(\log(2B))^{\rho^{[r],\adm}_{f,g}}.
\]
If both the active one-infinity set is empty and \(\rho^{[r],\adm}_{f,g}\le0\),
then
\[
        N^{\new}_{\ge r,f,g}(B)=O_{f,g,r}(1).
\]
Thus, across all multiplicity levels over \(\Q\), the active ordinary-integer
power exponent is either \(1/N\) or absent.
\end{corollary}

\begin{proof}
This is \Cref{thm:higher-multiplicity-sparse} with \(k=\Q\) and
\(S=\{\infty\}\).  The \(S\)-unit rank is zero in this case, so the active
one-infinity terms have no logarithmic factor.
\end{proof}

\begin{corollary}[Basic structure of the multiplicity hierarchy]
\label{cor:multiplicity-hierarchy-structure}
Let \(k\) be a number field, let \(S\) be a finite set of places containing the
Archimedean places, let \(f,g\in k[x]\) be nonconstant, and let \(r\ge1\).  Let
\[
        \Delta_{\new}=\deg(C^{\new}_{f,g}/\A^1_X)
        =\sum_C d_X(C),
\]
where the sum is over the non-graph irreducible components of \(C_{f,g}\)
which dominate the \(X\)-line.  Equivalently, the reduced coordinate ring of
\(C^{\new}_{f,g}\) is a finite flat \(k[X]\)-module of rank \(\Delta_{\new}\).
If \(r>\Delta_{\new}\), then
\[
        C^{[r]}_{f,g}=\varnothing
        \qquad\text{and}\qquad
        N^{\new}_{\ge r,f,g,k,S}(B)=0.
\]
Moreover the active power hierarchy is constant until it terminates.  If
\(\RR^{[r],\act}_1(f,g;k,S)=\varnothing\), then
\(\RR^{[r+1],\act}_1(f,g;k,S)=\varnothing\).  If the active sets for both
\(r\) and \(r+1\) are nonempty, then
\[
        \theta^{[r+1]}_{f,g,k,S}=\theta^{[r]}_{f,g,k,S}
        =\frac{\gcd(\deg f,\deg g)}{\deg g}.
\]
\end{corollary}

\begin{proof}
Let \(R\) be the coordinate ring of \(C^{\new}_{f,g}\).  The projection to the
\(X\)-line makes \(R\) a finite \(k[X]\)-algebra.  Since \(C^{\new}_{f,g}\) is
reduced and every component in its definition dominates the \(X\)-line, no
nonzero element of \(k[X]\) is a zero divisor on \(R\).  Thus \(R\) is a finite
torsion-free \(k[X]\)-module.  Because \(k[X]\) is a PID, \(R\) is finite flat
of rank
\[
        \Delta_{\new}=\deg(C^{\new}_{f,g}/\A^1_X).
\]
After base change to \(\bar k\), every geometric fiber therefore has scheme
length \(\Delta_{\new}\).  Its reduced cardinality is at most this length.  Hence
no fiber can contain more than \(\Delta_{\new}\) distinct new lifts, and the
ordered configuration cover is empty for \(r>\Delta_{\new}\).

The monotonicity assertions are arithmetic consequences of the level-\(r\)
sparse theorem, not purely geometric consequences of the configuration covers.
They follow from the inclusion
\[
        \{a:\text{at least }r+1\text{ distinct new lifts}\}
        \subseteq
        \{a:\text{at least }r\text{ distinct new lifts}\}.
\]
If an active log-negative component existed for level \(r+1\) while none existed
for level \(r\), \Cref{thm:higher-multiplicity-sparse} would give a positive
power lower bound for \(N^{\new}_{\ge r+1}\) and only a polylogarithmic upper
bound for \(N^{\new}_{\ge r}\), contradicting the inclusion for large \(B\).
If both active sets are nonempty, \Cref{thm:higher-multiplicity-sparse} gives
\(\theta^{[r]}_{f,g,k,S}=\theta^{[r+1]}_{f,g,k,S}=1/N\), where
\(N=\deg g/\gcd(\deg f,\deg g)\).
\end{proof}

\begin{example}[The hierarchy can persist or terminate]
For \(f(X)=X\) and \(g(Y)=Y^2\), there are no graph components and
\[
        C^{\new}_{f,g}: X=Y^2.
\]
An integer input \(n\) has a rational lift if and only if \(n\) is a square, and
except at \(n=0\) it has the two distinct lifts \(\pm\sqrt n\).  Thus
\[
        N^{\new}_{\ge1,f,g}(B)\asymp B^{1/2},
        \qquad
        N^{\new}_{\ge2,f,g}(B)\asymp B^{1/2},
        \qquad
        N^{\new}_{\ge3,f,g}(B)=0.
\]
The ordered two-configuration cover has the component
\[
        X=t^2,\qquad Y_1=t,\qquad Y_2=-t,
\]
so the level-two exponent is still read from an active log-negative component
of projection degree two.

For \(f(X)=X\) and \(g(Y)=Y^3\), again there are no graph components, but a
rational value has at most one rational cube root.  Hence
\[
        N^{\new}_{\ge1,f,g}(B)\asymp B^{1/3},
        \qquad
        N^{\new}_{\ge2,f,g}(B)=0.
\]
Thus higher multiplicity is not determined by the size of the set of inputs with
one lift; it is controlled by the configuration cover.
\end{example}

\section{Computing the exponent}
\label{sec:algorithm}

For explicit polynomials over \(\Q\), the exponent in \Cref{cor:Q-sparse} can
be extracted by a finite procedure.  The procedure is an explicit translation of
the theorem into standard operations on plane curves: factorization,
normalization, genus and boundary computation, parametrization of
\(\Q\)-rational components, and a finite integrality test for the
\(X\)-parametrization.  General algorithmic criteria for infinitude of integral
points on affine curves are developed by Alvanos--Bilu--Poulakis \cite{AlvanosBiluPoulakis};
the test below is the special one-variable denominator reduction needed for the
activity condition in the present counting theorem.

In this section, the effective assertion concerns the detection of the
one-infinity power exponent and the admissible unit-coset data, assuming the
standard number-field unit and congruence computations are available.  The finite
exceptional sets supplied by Siegel--Mahler are used only through their bounded
contribution and play no role in determining the exponent.

\begin{lemma}[Finite denominator test over \(\Q\)]
\label{lem:activity-denominator}
Let \(A\in\Q[t]\) be nonconstant.  There is an effectively computable integer \(M\ge1\) such
that, whenever \(t=a/b\in\Q\) is in lowest terms with \(b>0\) and \(A(t)\in\Z\),
one has \(b\mid M\).  Consequently, the condition \(A(\Q)\cap\Z\ne\varnothing\)
can be checked by finitely many explicit polynomial congruence conditions.
\end{lemma}

\begin{proof}
Write \(A(t)=q^{-1}P(t)\), with \(P\in\Z[t]\), \(q\in\Z_{>0}\), and
\[
        P(T)=\sum_{i=0}^m p_iT^i,
        \qquad p_m\ne0.
\]
Let \(p\) be a prime.  If \(p\nmid q p_m\), then the coefficients of \(P\) are
\(p\)-integral and the leading coefficient is a \(p\)-adic unit.  Hence
\(v_p(t)<0\) forces \(v_p(P(t))=m v_p(t)<0\), because the leading term has
strictly smaller valuation than every lower term.  Thus integrality of \(A(t)\)
forces \(v_p(t)\ge0\) at every such prime.

It remains only to bound denominators at the finite set of primes dividing
\(q p_m\).  Fix one such prime \(p\), and put \(w=v_p(t)\).  For each nonzero
lower coefficient \(p_i\), \(i<m\), the leading term strictly dominates the term
\(p_i t^i\) whenever
\[
        v_p(p_m)+mw<v_p(p_i)+iw,
\]
that is,
\[
        w<\frac{v_p(p_i)-v_p(p_m)}{m-i}.
\]
Ignore zero lower coefficients.  Choose an integer \(N_p\le0\) strictly smaller
than all the finitely many numbers
\[
        \frac{v_p(p_i)-v_p(p_m)}{m-i}\quad(p_i\ne0,\,i<m)
\]
and also strictly smaller than
\[
        \frac{v_p(q)-v_p(p_m)}{m}
\]
and strictly smaller than \(0\).
If \(w<N_p\), the leading term strictly dominates all lower terms, so
\[
        v_p(P(t))=v_p(p_m)+mw<v_p(q),
\]
and therefore \(v_p(A(t))<0\).  Thus \(A(t)\in\Z\) forces \(v_p(t)\ge N_p\) at
this bad prime.  Taking
\[
        M=\prod_{p\mid q p_m}p^{\max(0,-N_p)}
\]
clears all possible bounded negative valuations and gives the first claim.

For each divisor \(b\mid M\), the condition \(A(a/b)\in\Z\), with \((a,b)=1\), is
completely explicit.  Namely, write
\[
        P_b(T)=b^mP(T/b)=\sum_i p_iT^ib^{m-i}\in\Z[T],
\]
where \(P(T)=\sum_i p_iT^i\).  Then
\[
        A(a/b)\in\Z\quad\Longleftrightarrow\quad q b^m\mid P_b(a).
\]
For fixed \(b\), this divisibility depends only on the residue class of \(a\)
modulo \(q b^m\), together with the condition \((a,b)=1\).  Thus activity is
reduced to finitely many polynomial congruence checks.
\end{proof}

\begin{proposition}[Effective extraction of the ordinary-integer power exponent]
\label{prop:effective-exponent-extraction}
Let \(f,g\in\Q[x]\) be nonconstant, and put
\[
        N=\frac{\deg g}{\gcd(\deg f,\deg g)}.
\]
Assume the standard effective operations on plane curves over \(\Q\):
factorization, normalization, genus and boundary computation, geometric
irreducibility testing, and parametrization of rational genus-zero components.
Then one can decide by a finite calculation whether the ordinary-integer
new-lifting count has a one-infinity power term.  If such a term exists, its
exponent is \(1/N\); if none exists, the remaining contribution is only the
admissible two-infinity logarithmic term, or is bounded.
\end{proposition}

\begin{proof}
Factor the squarefree part of \(f(X)-g(Y)\) in \(\Q[X,Y]\), and remove the graph
factors \(Y-h(X)\) with \(h\in\Q[X]\) and \(f=g\circ h\).  By
\Cref{lem:degree-one-graph-components}, these are exactly the degree-one
components over the \(X\)-line.  For every remaining factor, test geometric
irreducibility.  A factor irreducible over \(\Q\) but not geometrically integral
contributes only finitely many rational points by \Cref{lem:ngeo-finite}.

For each geometrically integral component, compute the genus of the smooth
projective normalization and the number of boundary points.  The only components
that can produce a power term are the \(\Q\)-rational one-infinity components.
By primitive contact, every retained non-graph one-infinity component has
\(d_X(C)=N\).  Choose a polynomial parametrization
\[
        X=A(t),\qquad Y=B(t),\qquad A,B\in\Q[t],
        \qquad \deg A=N.
\]
The component is active for ordinary integers exactly when
\(A(\Q)\cap\Z\ne\varnothing\).  By \Cref{lem:activity-denominator}, this is a
finite congruence test: writing \(A=q^{-1}P\) as in that lemma, only finitely
many denominators \(b\mid M\) and residue classes modulo \(q b^{\deg A}\) can
occur for a reduced parameter \(t=a/b\).  If a solution \(t_0\) exists, then
\Cref{lem:integer-valued-coset} gives an arithmetic progression of parameters on
which \(A\) is integer-valued, and \Cref{lem:polynomial-S-value-set} gives the
full value-set order.  If no solution exists, that component contributes no
integer inputs outside the finite normalization-exceptional locus.  In the
quadratic-source case \(A(t)=\alpha t^2+\beta\), this finite test reduces to
the congruence in \Cref{cor:quadratic-source-activity-Q}.

If at least one active one-infinity component survives, \Cref{cor:Q-sparse}
gives the ordinary-integer source exponent \(\theta_{f,g}=1/N\).  If none
survives, no one-infinity power term exists; the source expansion leaves only the
admissible-rank logarithmic contribution from two-infinity components, or a
bounded set.  The finite exceptional sets supplied by Siegel--Mahler enter only
through their \(O(1)\) contribution to the height count and do not affect the
exponent decision.
\end{proof}

\section{Quadratic sources and the square-root boundary}
\label{sec:quadratic}

We now prove the reduction stated in the introduction.

\begin{proposition}[Quadratic-source reformulation]
\label{prop:quadratic-source}
Let \(k\) be a field of characteristic zero, and let \(f,g\in k[x]\) be
nonconstant.  The curve \(C_{f,g}\) has a non-graph \(k\)-rational one-infinity
component \(C\) with \(d_X(C)=2\) if and only if there exist polynomials
\(A,B\in k[t]\) such that
\[
        \deg A=2,
        \qquad k(A(t),B(t))=k(t),
\]
\[
        f(A(t))=g(B(t)),
\]
and
\[
        B(t)\notin k[A(t)].
\]
Equivalently, after a linear change of the parameter, the original
\(X\)-coordinate may be written
\[
        A(t)=\alpha t^2+\beta,
        \qquad \alpha\in k^*,\ \beta\in k.
\]
In the affine coordinate \(U=(X-\beta)/\alpha\), this says that the projection is
\(U=t^2\).
\end{proposition}

\begin{proof}
Suppose first that \(C\) is such a component.  Its smooth projective
normalization is \(\Pj^1_k\), and the affine normalization has one point at
infinity.  Hence there is a coordinate \(t\) on the affine normalization for
which
\[
        X=A(t),\qquad Y=B(t),\qquad A,B\in k[t].
\]
The projection degree to the \(X\)-line is \(\deg A\), so \(d_X(C)=2\) gives
\(\deg A=2\).  Since \(t\) is a normalization parameter,
\[
        k(C)=k(t)=k(A(t),B(t)).
\]
The identity \(f(X)=g(Y)\) on \(C\) gives \(f(A(t))=g(B(t))\).  If
\(B(t)\in k[A(t)]\), then \(Y\) is a polynomial function of \(X\), so \(C\) is a
graph component.  Thus \(B(t)\notin k[A(t)]\).

Conversely, such \(A,B\) define a rational curve contained in \(C_{f,g}\).  Its
Zariski closure is irreducible, and since \(C_{f,g}\) is a curve this closure is
an irreducible component of \(C_{f,g}\).  The condition \(k(A,B)=k(t)\) says that
the parametrization is birational onto its image.  Since \(\deg A=2\), the
projection degree to the \(X\)-line is \(2\).
The condition \(B\notin k[A]\) excludes graph components.  Finally, because
\(A,B\in k[t]\), the affine normalization has one geometric point at infinity.
Indeed, write \(A(t)=a_dt^d+a_{d-1}t^{d-1}+\cdots+a_0\) with \(a_d\in k^*\).
Then \(t\) satisfies the monic equation
\[
        T^d+\frac{a_{d-1}}{a_d}T^{d-1}+\cdots+\frac{a_0-A}{a_d}=0
\]
over \(k[A]\subseteq k[A,B]\), so \(t\) is integral over \(k[A,B]\).  Thus
\(k[t]\) is integral over \(k[A,B]\), has the same fraction field, and is
integrally closed.  To see that it is exactly the normalization, let
\(z\in k(t)\) be integral over \(k[A,B]\).  Choose a monic equation for \(z\) with
coefficients in \(k[A,B]\).  Since \(k[A,B]\subseteq k[t]\), this is also a
monic equation over \(k[t]\).  Thus \(z\) is integral over \(k[t]\), and the
integral closedness of \(k[t]\) gives \(z\in k[t]\).  Hence \(k[t]\) is the
integral closure of \(k[A,B]\) in \(k(t)\), i.e.\ the affine normalization.

In characteristic zero, every quadratic polynomial can be completed to a square.
After a linear change of the parameter we may therefore write the original
\(X\)-coordinate as \(A(t)=\alpha t^2+\beta\), with \(\alpha\ne0\).  Equivalently,
with \(U=(X-\beta)/\alpha\), the projection is \(U=t^2\).
\end{proof}

The following classical cancellation result is used only in the equal-outer-degree
case needed for quadratic sources.

\begin{lemma}[Engstrom cancellation in equal outer degree]
\label{lem:engstrom}
Let \(k\) be a field of characteristic zero.  Suppose
\(P,Q,R,S\in k[x]\) are nonconstant polynomials satisfying
\[
        P\circ Q=R\circ S,
        \qquad \deg P=\deg R.
\]
Then there is a linear polynomial \(\ell\in k[x]\) such that
\[
        P=R\circ \ell,
        \qquad S=\ell\circ Q.
\]
In particular, if \(P\circ Q=P\circ S\), then
\[
        S=\ell\circ Q,
        \qquad P\circ \ell=P
\]
for some linear \(\ell\).
\end{lemma}

\begin{proof}
This is the equal-outer-degree case of Engstrom's cancellation theorem for
polynomial substitutions \cite{Engstrom}.  We use the displayed orientation; if
one states the theorem with the linear factor on the opposite side, replacing
that linear polynomial by its inverse gives this form.  In the notation
\(P\circ Q=R\circ S\), Engstrom's theorem gives, after base change to \(\bar k\),
a polynomial \(\ell\) such that
\[
        P=R\circ\ell,
        \qquad
        S=\ell\circ Q.
\]
Because \(\deg P=\deg R\), this polynomial \(\ell\) has degree one.  Write
\(\ell(Y)=aY+b\).  Since \(Q,S\in k[x]\) and \(S=\ell\circ Q\), comparison of
leading coefficients gives \(a\in k\), and then comparison of constant terms
gives \(b\in k\).  Hence \(\ell\) is defined over \(k\).
\end{proof}

\begin{theorem}[Quadratic sources and affine involutions]
\label{thm:quadratic-symmetry}
Let \(k\) be a field of characteristic zero, and let \(f,g\in k[x]\) be
nonconstant.  The following are equivalent.
\begin{enumerate}[label=\textup{(\roman*)}]
    \item The curve \(C_{f,g}\) has a non-graph \(k\)-rational one-infinity component
    \(C\) with \(d_X(C)=2\).
    \item There exist \(\alpha\in k^*\), \(\beta,c\in k\), a nonzero polynomial
    \(E\in k[U]\), and a nonconstant polynomial \(G\in k[Z]\) such that
    \[
          g(Y)=G((Y-c)^2)
    \]
    and
    \[
          f(\alpha U+\beta)=G\bigl(U E(U)^2\bigr).
    \]
\end{enumerate}
In this case a corresponding component is parametrized by
\[
        X=\alpha t^2+\beta,
        \qquad
        Y=c+tE(t^2).
\]
Consequently, if \(g\) has no nontrivial affine involution \(Y\mapsto 2c-Y\)
with \(g(2c-Y)=g(Y)\), then \(C_{f,g}\) has no non-graph rational
one-infinity component with \(d_X=2\).  In particular, if \(\deg g\) is odd, no
such component exists.
\end{theorem}

\begin{proof}
Assume first that \(C\) is a non-graph \(k\)-rational one-infinity component with
\(d_X(C)=2\).  By \Cref{prop:quadratic-source}, after a linear change of the
parameter we may write the original \(X\)-coordinate as
\[
        X=A(t)=\alpha t^2+\beta,
        \qquad
        Y=B(t),
\]
with \(\alpha\ne0\), \(k(A,B)=k(t)\), \(f(A(t))=g(B(t))\), and
\(B(t)\notin k[A(t)]=k[t^2]\).  Since \(A(t)=A(-t)\), we have
\[
        g(B(t))=g(B(-t)).
\]
The polynomial \(B\) is nonconstant: otherwise
\(k(A(t),B(t))=k(A(t))=k(t^2)\ne k(t)\).  Thus the two nonconstant polynomials
\(B(t)\) and \(B(-t)\) have the same degree.  Applying
\Cref{lem:engstrom} to \(g\circ B=g\circ(B\circ(-1))\), there is a linear
polynomial \(\ell\in k[Y]\) such that
\[
        B(-t)=\ell(B(t)),
        \qquad
        g\circ\ell=g.
\]
Applying \(t\mapsto -t\) again gives \(B(t)=\ell^2(B(t))\), hence
\(\ell^2=\mathrm{id}\).  The map \(\ell\) is not the identity, because otherwise
\(B(t)=B(-t)\) and \(B\in k[t^2]=k[A]\), contradicting the non-graph condition.
Thus \(\ell\) is a nontrivial affine involution.  In characteristic zero,
\(\ell(Y)=2c-Y\) for some \(c\in k\).  Its invariant ring is
\[
        k[Y]^\ell=k[(Y-c)^2],
\]
so \(g\circ\ell=g\) is equivalent to
\[
        g(Y)=G((Y-c)^2)
\]
for some \(G\in k[Z]\).  The identity \(B(-t)=2c-B(t)\) says that \(B(t)-c\) is
odd, hence
\[
        B(t)=c+tE(t^2)
\]
with \(E\in k[U]\), \(E\ne0\).  Substituting \(U=t^2\) into
\(f(A(t))=g(B(t))\) gives
\[
        f(\alpha U+\beta)=G\bigl(U E(U)^2\bigr).
\]
This proves (ii).

Conversely, assume (ii).  Then
\[
        X=\alpha t^2+\beta,
        \qquad
        Y=c+tE(t^2)
\]
defines a rational curve contained in \(C_{f,g}\), since
\[
        g(Y)=G\bigl((tE(t^2))^2\bigr)=G(t^2E(t^2)^2)=f(\alpha t^2+\beta).
\]
Because \(E\ne0\), the function field generated by \(t^2\) and \(tE(t^2)\) is
\(k(t)\): explicitly, \(t=(tE(t^2))/E(t^2)\).  Therefore the parametrization is
birational onto its image.  The Zariski closure of the image is an irreducible
component of \(C_{f,g}\), its projection degree to the \(X\)-line is two, and
its affine normalization has one geometric point at infinity.  Finally,
\(c+tE(t^2)\notin k[t^2]\), so the component is not a graph component.  This
proves (i).

The final assertions follow immediately.  A nontrivial affine involution has the
form \(Y\mapsto2c-Y\), and its invariants are polynomials in \((Y-c)^2\).  If
\(\deg g\) is odd, \(g(Y)=G((Y-c)^2)\) is impossible for nonconstant \(G\).
\end{proof}

\begin{corollary}[Degree obstruction to square-root growth]
\label{cor:degree-obstruction-square-root}
Let \(f,g\in k[x]\) be nonconstant, with
\[
        m=\deg f,\qquad n=\deg g.
\]
If \(C_{f,g}\) has a non-graph \(k\)-rational one-infinity component with
\(d_X=2\), then
\[
        \frac{n}{\gcd(m,n)}=2.
\]
Equivalently,
\[
        n\mid 2m
        \qquad\text{and}\qquad
        n\nmid m.
\]
In particular, if \(n/\gcd(m,n)\ne2\), then a projection-degree-two
one-infinity source is impossible before one tests source-evenness or arithmetic
activity.  Over \(\Q\), this rules out square-root ordinary-integer growth from a
one-infinity source.
\end{corollary}

\begin{proof}
By \Cref{cor:geometric-degree-spectrum}, every one-infinity component has
\[
        d_X=\frac{n}{\gcd(m,n)}.
\]
Thus \(d_X=2\) is equivalent to \(n/\gcd(m,n)=2\).  The displayed divisibility
conditions are this equality rewritten: \(n\mid2m\) and \(n\nmid m\).
\end{proof}

\begin{definition}[Source-even polynomials]
Let \(k\) be a field of characteristic zero.  We say that \(g\in k[Y]\) is
\emph{source-even over \(k\)} if there are \(c\in k\) and \(G\in k[Z]\) such that
\[
        g(Y)=G((Y-c)^2).
\]
Equivalently, \(g\) is invariant under the nontrivial affine involution
\(Y\mapsto 2c-Y\).  This terminology avoids ambiguity with the two-sided notion of linear conjugacy
used in dynamics.
\end{definition}

\begin{corollary}[Explicit fixed-\(g\) square-root sources]
\label{cor:fixed-g-square-root}
Let \(k\) be a field of characteristic zero and let \(g\in k[Y]\) be
nonconstant.  The polynomials \(f\in k[X]\) for which \(C_{f,g}\) has a
non-graph \(k\)-rational one-infinity component with \(d_X=2\) are precisely the
following.  There exist \(c,\beta\in k\), \(\alpha\in k^*\), a nonzero polynomial
\(E\in k[U]\), and a nonconstant polynomial \(G\in k[Z]\) such that
\[
        g(Y)=G((Y-c)^2)
\]
and
\[
        f(X)=G\left(\frac{X-\beta}{\alpha}
              E\left(\frac{X-\beta}{\alpha}\right)^2\right).
\]
The corresponding component is parametrized by
\[
        X=\alpha t^2+\beta,
        \qquad
        Y=c+tE(t^2).
\]
In particular, once the involutive quotient \(g(Y)=G((Y-c)^2)\) is fixed,
quadratic sources are obtained from the choices of \(\alpha,\beta\) and \(E\).
This is a surjective description rather than an injective parametrization: for
example, \(E\) and \(-E\) give the same polynomial
\(f\) and source after the change \(t\mapsto -t\), and affine reparametrizations
of the source may change the displayed data.
\end{corollary}

\begin{proof}
This is \Cref{thm:quadratic-symmetry} with \(U=(X-\beta)/\alpha\), rewritten as
a formula for \(f\).  Conversely, the displayed formula gives
\[
        f(\alpha t^2+\beta)=G(t^2E(t^2)^2)=g(c+tE(t^2)),
\]
and \(k(t^2,tE(t^2))=k(t)\), so the parametrized curve is a non-graph rational
one-infinity component of projection degree two.
\end{proof}

\begin{definition}[The first-kind Bilu--Tichy cell used here]
Only the following cell of the Bilu--Tichy standard-pair taxonomy is needed here
\cite{BiluTichy}.  A first-kind standard pair over \(k\), up to switching the two
coordinates, has the form
\[
        \bigl(x^m,\; a x^r p(x)^m\bigr),
\]
where \(a\in k^*\), \(p\in k[x]\), \(m\ge2\), \(0\le r<m\),
\((r,m)=1\), and \(r+\deg p>0\).  In applications of the Bilu--Tichy theorem one
also allows linear changes and a common outer polynomial.  The square-root
boundary below uses only the case
\[
        m=2,
        \qquad r=1,
        \qquad a=1.
\]
\end{definition}

\begin{proposition}[Bilu--Tichy placement of quadratic sources]
\label{prop:BT-placement-quadratic-sources}
Let \(k\) be a field of characteristic zero, and let \(f,g\in k[x]\) be
nonconstant.  A non-graph \(k\)-rational one-infinity component of \(C_{f,g}\) with
\(d_X=2\) is, after the linear changes appearing in
\Cref{thm:quadratic-symmetry}, obtained from the switched first-kind
Bilu--Tichy pair
\[
        \bigl(U E(U)^2,\; V^2\bigr),
        \qquad E\in k[U]\setminus\{0\}.
\]
More precisely, such a component exists if and only if there are
\[
        c,\beta\in k,
        \qquad \alpha\in k^*,
        \qquad E\in k[U]\setminus\{0\},
        \qquad G\in k[Z]\setminus k
\]
such that
\[
        g(Y)=G((Y-c)^2)
\]
and
\[
        f(X)=G\left(
             \frac{X-\beta}{\alpha}
             E\left(\frac{X-\beta}{\alpha}\right)^2
        \right).
\]
Equivalently, with
\[
        U=\frac{X-\beta}{\alpha},
        \qquad V=Y-c,
\]
one has
\[
        f=G\circ F\circ \lambda_X,
        \qquad
        g=G\circ Q\circ \lambda_Y,
\]
where
\[
\begin{gathered}
        F(U)=UE(U)^2,
        \qquad Q(V)=V^2,\\
        \lambda_X(X)=\frac{X-\beta}{\alpha},
        \qquad \lambda_Y(Y)=Y-c.
\end{gathered}
\]
The pair \((F,Q)\) is the switched first-kind Bilu--Tichy standard pair with
\[
        m=2,
        \qquad r=1,
        \qquad a=1,
        \qquad p=E.
\]
The corresponding component is parametrized by
\[
        U=t^2,
        \qquad V=tE(t^2),
\]
or, in the original coordinates,
\[
        X=\alpha t^2+\beta,
        \qquad Y=c+tE(t^2).
\]
It is not a graph component.
\end{proposition}

\begin{proof}
By \Cref{thm:quadratic-symmetry}, every non-graph \(k\)-rational one-infinity
component with \(d_X=2\) admits the displayed source-even form
\[
        g(Y)=G((Y-c)^2),
        \qquad
        f(\alpha U+\beta)=G(UE(U)^2),
\]
with \(E\ne0\).  Put
\[
        F(U)=UE(U)^2,
        \qquad Q(V)=V^2.
\]
Then
\[
        f=G\circ F\circ \lambda_X,
        \qquad
        g=G\circ Q\circ \lambda_Y.
\]
The pair
\[
        (F,Q)=\bigl(UE(U)^2,V^2\bigr)
\]
is exactly the switched standard pair of the first kind with
\(m=2\), \(r=1\), \(a=1\), and \(p=E\).

Conversely, any datum of this form gives
\[
        G(UE(U)^2)=G(V^2).
\]
Since \(T-W\) divides \(G(T)-G(W)\) in \(k[T,W]\), the polynomial
\[
        UE(U)^2-V^2
\]
divides \(G(UE(U)^2)-G(V^2)\).  This factor is irreducible in \(k[U,V]\), since
\(UE(U)^2\) is not a square in \(k[U]\).  Hence, after the invertible linear
change from \((X,Y)\) to \((U,V)\), it survives as an irreducible factor of the
squarefree part defining \(C_{f,g}\), and its zero set is an irreducible
component.  It gives the parametrized curve
\[
        U=t^2,
        \qquad V=tE(t^2).
\]
Because \(E\ne0\),
\[
        k(t^2,tE(t^2))=k(t),
\]
so the parametrization is birational.  Also
\(tE(t^2)\notin k[t^2]\), because it is an odd polynomial in \(t\), while every
polynomial in \(t^2\) is even.  Thus \(Y\) is not a polynomial function of
\(X\), so the component is not a graph component.  Finally, the
\(X\)-projection is \(X=\alpha t^2+\beta\), hence has degree \(2\), and the
affine normalization has one point at infinity.
\end{proof}

\begin{remark}
\Cref{prop:BT-placement-quadratic-sources} is a source-level statement.  Special
choices of \(G\) or \(E\) may admit additional decompositions or overlap with
other standard-pair descriptions, but those coincidences do not create another
projection-degree-two mechanism.  For the square-root boundary, the only
Bilu--Tichy cell needed is the switched first-kind cell
\((xp(x)^2,x^2)\).
\end{remark}

\begin{corollary}[Ordinary integer activity of a quadratic source]
\label{cor:quadratic-source-activity-Q}
Let
\[
        A(t)=\alpha t^2+\beta,
        \qquad \alpha\in\Q^*,\quad \beta\in\Q.
\]
Write
\[
        \beta=\frac bd,
        \qquad b\in\Z,
        \qquad d\in\Z_{>0},
        \qquad (b,d)=1.
\]
Let \(s\in\Z\setminus\{0\}\) be the squarefree integer representing the square
class of \(d\alpha\), so
\[
        d\alpha\in s\Q^{*2}.
\]
Then the associated quadratic source is active for ordinary-integer inputs,
equivalently
\[
        A(\Q)\cap\Z\ne\varnothing,
\]
if and only if the congruence
\[
        s z^2\equiv -b\pmod d
\]
has a solution \(z\in\Z\).

Consequently, for the fixed-\(g\) source-even family over \(\Q\), the active
source-even Bilu--Tichy subfamilies are precisely
\[
        \left(
        G\left(\frac{X-\beta}{\alpha}
              E\left(\frac{X-\beta}{\alpha}\right)^2\right),
        \;
        G((Y-c)^2)
        \right),
\]
where \(E\in\Q[U]\setminus\{0\}\), \(g(Y)=G((Y-c)^2)\), and the above congruence
in \(\alpha,\beta\) is soluble.  The parameters \(c\), \(E\), and \(G\) do not
enter the ordinary-integer activity test except through the existence of the
source itself.
\end{corollary}

\begin{proof}
Choose \(q\in\Q^*\) such that
\[
        d\alpha=sq^2.
\]
The equality \(A(t)=n\), with \(t\in\Q\) and \(n\in\Z\), is equivalent to
\[
        d\alpha t^2+b=nd,
\]
or
\[
        s(qt)^2=nd-b.
\]
If such \(t,n\) exist, then \(s(qt)^2\in\Z\).  Since \(s\) is squarefree, this
forces \(qt\in\Z\): indeed, writing \(qt=a/r\) in lowest terms with \(r>0\)
gives \(s a^2/r^2\in\Z\), hence \(r^2\mid s a^2\); coprimality gives
\(r^2\mid s\), and squarefreeness of \(s\) gives \(r=1\).  Writing \(z=qt\), we get
\[
        sz^2\equiv -b\pmod d.
\]

Conversely, suppose \(sz^2\equiv -b\pmod d\).  Put
\[
        n=\frac{b+sz^2}{d}\in\Z,
        \qquad
        t=\frac zq.
\]
Then
\[
        \alpha t^2+\beta
        =\alpha\frac{z^2}{q^2}+\frac bd
        =\frac{sz^2}{d}+\frac bd
        =n.
\]
Thus \(A(\Q)\cap\Z\ne\varnothing\).  The final statement follows from
\Cref{prop:BT-placement-quadratic-sources}: the Bilu--Tichy core is always
\((UE(U)^2,V^2)\), and ordinary-integer activity depends only on the quadratic
\(X\)-parametrization \(X=\alpha t^2+\beta\).
\end{proof}

\begin{corollary}[Square-root growth over \(\Q\), explicitly]
\label{cor:Q-square-root-explicit}
Let \(f,g\in\Q[x]\) be nonconstant.  Then
\[
        N^{\new}_{f,g}(B)\asymp_{f,g} B^{1/2}
\]
if and only if there exist
\[
        c,\beta\in\Q,
        \qquad \alpha\in\Q^*,
        \qquad E\in\Q[U]\setminus\{0\},
        \qquad G\in\Q[Z]\setminus\Q
\]
such that
\[
        g(Y)=G((Y-c)^2),
\]
\[
        f(X)=G\left(
             \frac{X-\beta}{\alpha}
             E\left(\frac{X-\beta}{\alpha}\right)^2
        \right),
\]
and, writing \(\beta=b/d\) with \((b,d)=1\) and \(d>0\), and writing
\(d\alpha\in s\Q^{*2}\) with \(s\) squarefree, the congruence
\[
        sz^2\equiv -b\pmod d
\]
is soluble.
\end{corollary}

\begin{proof}
By \Cref{cor:Q-sparse}, square-root growth is equivalent to the existence of an
active non-graph \(\Q\)-rational one-infinity component with \(d_X=2\).  The
finite symmetric-difference source expansion and the lower-bound theorem ensure
that, for such an active non-graph source, deleting graph intersections and other
finite exceptional values does not change the \(B^{1/2}\) order of growth.  By
\Cref{prop:BT-placement-quadratic-sources}, such components are exactly the
source-even switched first-kind Bilu--Tichy sources displayed above.  By
\Cref{cor:quadratic-source-activity-Q}, ordinary-integer activity for the
quadratic \(X\)-parametrization is exactly the displayed congruence.
\end{proof}

\begin{corollary}[Monomial square-root sources]
\label{cor:monomial-square-root}
Let \(k\) be a field of characteristic zero and let \(g(Y)=Y^m\) with
\(m\ge2\).  If \(m\) is odd, then no non-graph \(k\)-rational one-infinity component
with \(d_X=2\) can occur for any \(f\in k[X]\).  If \(m\) is even, then such
components occur exactly for polynomials of the form
\[
        f(X)=\left(\frac{X-\beta}{\alpha}\right)^{m/2}
              E\left(\frac{X-\beta}{\alpha}\right)^m,
        \qquad
        \alpha\in k^*,\ \beta\in k,
\]
with \(E\in k[U]\setminus\{0\}\), where
\(U=(X-\beta)/\alpha\) is the displayed affine \(X\)-coordinate.  The
corresponding component is
\[
        X=\alpha t^2+\beta,
        \qquad
        Y=tE(t^2).
\]
\end{corollary}

\begin{proof}
For \(g(Y)=Y^m\), the condition \(g(Y)=G((Y-c)^2)\) is impossible when \(m\) is
odd.  When \(m\) is even, the only affine involution fixing \(Y^m\) in
characteristic zero is \(Y\mapsto -Y\), so \(c=0\) and \(G(Z)=Z^{m/2}\).  Apply
\Cref{cor:fixed-g-square-root}.
\end{proof}

\begin{theorem}[Geometry of the source-even locus]
\label{thm:source-even-locus}
Let \(k\) be an algebraically closed field of characteristic zero, let \(m\ge1\),
and let \(\mathcal P_m\) be the coefficient space of degree-exactly-\(m\)
polynomials in one variable.  Let \(\mathcal E_m\subset\mathcal P_m\) be the locus of
source-even polynomials, i.e.\ polynomials of the form
\[
        g(Y)=G((Y-c)^2)
\]
for some \(c\in k\) and some polynomial \(G\).

If \(m\) is odd, then \(\mathcal E_m=\varnothing\).  If \(m=2r\), then
\(r\ge1\), and \(\mathcal E_m\) is the constructible image of the injective incidence map
\[
        \A^1_c\times \mathcal P_r
        \longrightarrow \mathcal P_{2r},
        \qquad
        (c,G)\longmapsto G((Y-c)^2).
\]
Consequently \(\mathcal E_{2r}\) is an irreducible constructible subset whose
Zariski closure has dimension \(r+2\), and this closure has codimension
\[
        (2r+1)-(r+2)=r-1
\]
in \(\mathcal P_{2r}\).  In particular, for every \(m\ge3\), the source-even
locus is contained in a proper Zariski-closed subset of \(\mathcal P_m\).
\end{theorem}

\begin{proof}
Odd-degree polynomials cannot be of the form \(G((Y-c)^2)\), so assume
\(m=2r\).  The displayed map is a morphism from the irreducible variety
\(\A^1\times\mathcal P_r\), whose dimension is \(1+(r+1)=r+2\).  Its image is
constructible by Chevalley's theorem.

It is injective on \(k\)-points.  Indeed, a nonconstant polynomial invariant
under a reflection \(Y\mapsto2c-Y\) has a unique center: two different centers
would generate a nontrivial translation, and a nonconstant polynomial in
characteristic zero cannot be invariant under a nontrivial translation.  Once
\(c\) is known, \(G\) is recovered uniquely from \(g(Y)=G((Y-c)^2)\).  Since
the morphism is of finite type over the algebraically closed field \(k\),
injectivity on \(k\)-points implies that every fiber is zero-dimensional:
a positive-dimensional finite-type \(k\)-scheme would contain infinitely many
\(k\)-points.  The fiber dimension theorem therefore gives that the constructible
image has dimension
\[
        \dim(\A^1\times\mathcal P_r)=r+2.
\]
The image of an irreducible variety is irreducible as a constructible subset, and
its Zariski closure is irreducible of the same dimension.  Since
\(\dim\mathcal P_{2r}=2r+1\), the codimension is \(r-1\).  For \(m\ge3\), this
is either the empty locus when \(m\) is odd or a proper closed condition when
\(m=2r\ge4\).
\end{proof}

\begin{corollary}[Generic absence of square-root sources]
\label{cor:generic-no-square-root}
Let \(k\) be an algebraically closed field of characteristic zero, and fix an
integer \(m\ge3\).  For a Zariski-generic polynomial \(g\in k[Y]\) of
degree \(m\), there is no non-graph \(k\)-rational one-infinity component with
\(d_X=2\) in \(C_{f,g}\), for any polynomial \(f\in k[X]\).  More precisely, the
exceptional set of degree-\(m\) polynomials \(g\) for which such a source is not
excluded is contained in the proper Zariski-closed set
\(\overline{\mathcal E_m}\) from \Cref{thm:source-even-locus}.
\end{corollary}

\begin{proof}
By \Cref{thm:quadratic-symmetry}, a square-root source forces
\(g(Y)=G((Y-c)^2)\), i.e. \(g\in\mathcal E_m\).  For \(m\ge3\),
\Cref{thm:source-even-locus} says that \(\mathcal E_m\) is contained in a proper
Zariski-closed subset of the degree-\(m\) coefficient space.  On the complement
of that closed set, square-root sources are impossible.
\end{proof}

\begin{proposition}[Detecting source-even polynomials]
\label{prop:detect-source-even}
Let \(k\) be a field of characteristic zero, and let
\[
        g(Y)=a_nY^n+a_{n-1}Y^{n-1}+\cdots+a_0\in k[Y]
\]
be nonconstant.  If \(g\) is source-even over \(k\), then \(n\) is even and the
center \(c\) is unique.  For even \(n\), the only possible center is
\[
        c=-\frac{a_{n-1}}{n a_n}.
\]
Thus source-evenness is checked by one finite coefficient test: translate by
this candidate \(c\) and verify that all odd powers of \(Z\) in \(g(c+Z)\)
vanish.
\end{proposition}

\begin{proof}
Source-evenness is the existence of an affine reflection \(Y\mapsto2c-Y\) fixing
\(g\), equivalently \(g(c+Z)=g(c-Z)\).  If two different centers \(c_1,c_2\)
worked, then the two corresponding reflections would generate the nontrivial
translation \(Y\mapsto Y+2(c_2-c_1)\).  A nonconstant polynomial in
characteristic zero cannot be invariant under a nontrivial translation, so the
center is unique.

If \(g(c+Z)\) is even, its degree must be even.  Also the coefficient of
\(Z^{n-1}\) in \(g(c+Z)\) must vanish.  That coefficient is
\(n a_n c+a_{n-1}\), giving the displayed formula for \(c\).  Once this unique
candidate is known, the condition is exactly the vanishing of the finitely many
odd coefficients in \(g(c+Z)\).
\end{proof}

\begin{corollary}[A practical square-root obstruction]
\label{cor:no-quadratic-symmetry}
Let \(k\) be a number field, let \(S\) contain the Archimedean places, and let
\(f,g\in k[x]\) be nonconstant.  Put
\[
        N=\frac{\deg g}{\gcd(\deg f,\deg g)}.
\]
If \(g\) is not source-even over \(k\), then no active component with
\(d_X=2\) exists.  Consequently, over \(\Q\) with ordinary-integer counting, no
one-infinity source can produce square-root growth \(B^{1/2}\).  More precisely,
if \(N=2\), then there is no one-infinity power term at all, so the count is
purely polylogarithmic or bounded; if \(N\ge3\), then any active one-infinity
power term has source exponent \(1/N\le1/3\) and number-field height-count power
\(B^{[k:\Q]/N}\), up to the standard \(S\)-unit logarithmic factor; and if
\(N=1\), there is again no one-infinity power term after graph removal.  If
\(\deg g\) is odd, then \(g\) is not source-even.
\end{corollary}

\begin{proof}
By \Cref{thm:quadratic-symmetry}, a quadratic one-infinity source forces
\(g(Y)=G((Y-c)^2)\) for some \(c\in k\), i.e. \(g\) is source-even over \(k\).
If this does not happen, no component with \(d_X=2\) exists.  The remaining
alternatives are exactly \Cref{thm:gcd-exponent-bound} and
\Cref{thm:number-field-exact}.  The degree of a source-even polynomial is even,
so odd degree excludes source-evenness.
\end{proof}

\begin{remark}
\Cref{prop:quadratic-source} reduces square-root sources to nontrivial quadratic
polynomial pullback identities
\[
        f(\alpha t^2+\beta)=g(B(t)),
        \qquad k(\alpha t^2+\beta,B(t))=k(t),
        \qquad B(t)\notin k[\alpha t^2+\beta].
\]
\Cref{thm:quadratic-symmetry} classifies these identities in the form needed for
the height-counting problem, and \Cref{prop:BT-placement-quadratic-sources}
identifies their exact Bilu--Tichy source cell.  Bilu's theorem on quadratic
factors remains relevant low-degree background \cite{BiluQuadratic}, but it
concerns total degree of a factor rather than the projection degree \(d_X\).
Here the square-root boundary is a projection-degree-two statement, and the
required taxonomy is therefore the single switched first-kind cell isolated
above.
\end{remark}

\section{Concluding remarks}

The central geometric point of the paper is the following.  The reduced fiber product
\(C_{f,g}\) is a finite cover of the \(X\)-line, but after graph components are
removed its sparse behavior is controlled by logarithmic source curves.  The
normalization and boundary divisor determine the trichotomy
\[
        \A^1\text{-source},\qquad \Gm\text{-source},\qquad
        \text{log-hyperbolic source},
\]
and primitive contact at infinity determines the projection degree:
\[
        -\ord_P X=\frac{\deg g}{\gcd(\deg f,\deg g)},
        \qquad
        -\ord_P Y=\frac{\deg f}{\gcd(\deg f,\deg g)}.
\]
Thus the main source exponent is governed by primitive boundary contact on a
logarithmic curve; over number fields
the actual height-count exponent is obtained by multiplying by \([k:\Q]\), and
activity is refined by the elementary condition
\(A(k)\cap\OO_{k,S}\ne\varnothing\).  The
corresponding arithmetic refinement is the finite symmetric-difference source expansion
\[
        \mathscr N^{\new}_{f,g,k,S}
        =_{\mathrm{fin}}
        \bigcup_{C\in\RR_1^{\act}}(A_C(k)\cap\OO_{k,S})
        \ \cup\
        \bigcup_{C\in\RR_2} A_C(\mathcal U_C),
\]
together with the source-separation theorem, which upgrades this finite union to
a canonical power expansion
\[
        N^{\new}_{f,g,k,S}(B)
        =
        \sum_{\xi}\#\mathscr P_\xi(B)+O((\log B)^\rho)
\]
whenever active one-infinity sources exist.  The infinite contributions are
precisely the polynomial source images, the admissible unit-coset Laurent images,
and the polylogarithmic overlaps among distinct primitive \(X\)-source classes.

The same structure explains why the sparse lifting theorem is stable under higher
multiplicity.  To count inputs with at least \(r\) distinct new rational lifts,
one replaces the non-graph cover by its ordered \(r\)-configuration cover and
applies the same logarithmic trichotomy and the same primitive contact theorem.
The same Diophantine mechanisms govern higher multiplicity: the source exponent
is \(1/N\) at every active level, with number-field height-count power
\(B^{[k:\Q]/N}\), and it disappears when the active configuration source
disappears.

The source-stratum result suggests a natural algebraic-geometric
refinement.  For fixed degrees \((m,n)\), the locus in coefficient space
admitting an \(\A^1\)-source is constructible and lives at the single primitive
degree
\[
        d=N=\frac{n}{\gcd(m,n)}.
\]
This suggests the following problem, which is a geometric refinement of the
classical reducibility and separated-variable classification problems rather
than a new height-counting problem.

\begin{question}[Geometry of the primitive source stratum]
Fix \(m,n\), and put \(N=n/\gcd(m,n)\).  Can one describe the irreducible
components and dimensions of the constructible primitive source stratum from
\Cref{prop:source-strata-constructible}, especially after quotienting by affine
changes of the source parameter and by common outer composition?  How do the
Bilu--Tichy standard pairs sit inside this single-degree stratification?
\end{question}

The admissible-unit-lattice theorem describes the log-zero case.  For each
two-infinity source, \Cref{thm:admissible-unit-lattice} constructs a finite
union of descent-compatible unit cosets
\[
        \mathcal U_C=\bigcup_j\eta_j\Gamma_j
\]
and proves that the component contributes exactly
\[
        (\log B)^{\max_j\operatorname{rk}\Gamma_j}
\]
when the maximal rank is positive, and only finitely many inputs otherwise.  The
Pell example is the rank-one norm-torus case of this theorem.  Thus the
arithmetic trichotomy is exact: active \(\A^1\)-sources give power growth,
positive-rank admissible \(\Gm\)-sources give logarithmic unit-rank growth, and
inactive one-infinity sources, rank-zero admissible \(\Gm\)-sources, and
log-hyperbolic sources give finiteness.

\begin{remark}[Boundary data and reduced components]
Derived fiber products and cohomological realizations are natural in many
intersection-theoretic or finite-field problems, but the invariants used here are
the reduced irreducible components, their normalizations, their boundary
divisors, and the pole orders of \(X\) and \(Y\) at those divisors.  Nilpotent or
derived structure at component intersections and ramification points does not
alter those data.  The sparse exponents are therefore read from logarithmic
curves and their contact orders.
\end{remark}

\begin{remark}[Finite-field analogue]
The corresponding finite-field problem has a different character.  Over good
reductions, the non-graph cover gives a finite cover of the \(X\)-line over
finite fields.  In that setting the relevant density of residue classes is a
monodromy question, with Frobenius acting on the geometric fiber and square-root
error terms supplied by the Grothendieck--Lefschetz trace formula; it is not a
height-counting problem.
\end{remark}

The square-root boundary admits a complete elementary source-level description.
A projection-degree-two source forces an involutive quotient
\[
        g(Y)=G((Y-c)^2),
\]
and every such source has the parametrization
\[
        X=\alpha t^2+\beta,
        \qquad
        Y=c+tE(t^2).
\]
Thus the source-level Bilu--Tichy core is exactly the switched first-kind pair
\((UE(U)^2,V^2)\).  Over \(\Q\), ordinary-integer activity is decided only by
\[
        \beta=\frac bd,
        \qquad d\alpha\in s\Q^{*2},
        \qquad sz^2\equiv -b\pmod d.
\]
Thus the square-root case reduces to the displayed quadratic congruence.

\subsection*{Conflict of interest}

The author declares no competing interests.

\end{document}